\documentclass[11pt]{amsart}
\nonstopmode
\usepackage{amsbsy,amssymb,amscd,amsfonts,latexsym,amstext,delarray, mathrsfs, amsmath,graphicx,color,caption}
\usepackage{graphicx}
\usepackage{enumerate}
\input xy

\xyoption{all}
\usepackage{euscript}
\usepackage{multirow}
\usepackage{etex, pictexwd,dcpic}
\usepackage[normalem]{ulem}

\usepackage{amssymb}

\usepackage{url}
\usepackage[margin=1in]{geometry} 

\newtheorem{theorem}{Theorem}[section] 
\newtheorem{claim}[theorem]{Claim}
\newtheorem{corollary}[theorem]{Corollary}
\newtheorem{lemma}[theorem]{Lemma} 

\theoremstyle{definition} 

\newcommand{\Hom}{\operatorname{Hom}}
\newcommand{\rank}{\operatorname{rank}} 

\usepackage{amsbsy,amssymb,amscd,amsfonts,latexsym,amstext,delarray, amsmath,graphicx,color,caption}

\usepackage{mathtools}
\usepackage{mathrsfs}
\usepackage{amsmath}
\usepackage{amssymb}

\usepackage{csquotes}
\usepackage{amscd}
\usepackage{tikz}
\usetikzlibrary{positioning}
\usepackage{extarrows}
\DeclareMathOperator{\Aut}{Aut} 

\usepackage[colorlinks=true, pdfstartview=FitV,
 linkcolor=blue,citecolor=blue,urlcolor=blue]{hyperref}
\def\wgal{\widetilde w}
\renewcommand{\a}{\alpha}

\newcommand{\h}{\mathfrak h}

\newcommand{\en}{\begin{enumerate}}
\newcommand{\te}{\end{enumerate}}
\newcommand{\Z}{{\mathbb Z}}

\newcommand{\C}{{\mathbb C}}

\newenvironment{red}{\relax\color{red}}{\relax}
\newenvironment{blue}{\relax\color{blue}}{\relax}

\newcommand{\ber}{\begin{red}}
\newcommand{\er}{\end{red}}
\newcommand{\beb}{\begin{blue}}
\newcommand{\eb}{\end{blue}}

\usepackage{fullpage}
\newtheorem {proposition}[theorem]    {Proposition}

\newtheorem{definition}[theorem] {Definition}

\def\wgal{\widetilde w}

\def\F{\mathbb{F}}
\def\Z{\mathbb{Z}}

\def\C{\mathbb{C}}

\def\bsl{\backslash}
\def\calU{\mathcal{U}}

\numberwithin{equation}{section}
\DeclareMathOperator{\ad}{ad}
\def\Sol{\textrm{Sol\,}}
\def\C{{\mathbb C}}
\def\Hom{\textrm{Hom}}
\def\Vol{\textrm{Vol}}
\def\ker{\textrm{ker\,}}
\def\rank{\textrm{rank\,}}

\def\im{\textrm{Im\,}}

\def\card{{\mathrm card}\hskip1pt}

\DeclareMathOperator{\re}{{re}}
    \evensidemargin \oddsidemargin
\begin{document}

\title{\bf{Eisenstein series on arithmetic quotients of rank 2 Kac--Moody groups over finite fields}}
\date{\today}

\author[A. Ali]{Abid Ali}
\address{Department of Mathematics and Statistics, University of Saskatchewan, Saskatoon, SK S7N 5A2, Canada}
\email{abid.ali@usask.ca}

\author[L. Carbone]{Lisa Carbone$^{\dagger}$}
\address{Department of Mathematics, Rutgers University, Piscataway, NJ 08854-8019, U.S.A.}
\email{carbonel@math.rutgers.edu}
\thanks{$^{\dagger}$This work was partially supported by the Simons Foundation, Mathematics and Physical Sciences-Collaboration Grants for Mathematicians, Award Number 961267}

\author[Paul Garrett]{Paul Garrett}
\address{School of Mathematics, The University of Minnesota, 206 Church Street SE
Minneapolis, MN 55455}
\email{garrett@umn.edu }

\begin{abstract}
Let $G$ be an affine or hyperbolic rank 2 Kac--Moody group over a finite field $\mathbb F_q$. Let $X=X_{q+1}$ be the Tits building of $G$, the $(q+1)$--homogeneous tree, and let $\Gamma$ be a non-uniform lattice in $G$. When $\Gamma$ is a standard parabolic subgroup for the negative $BN$--pair, we define Eisenstein series on $\Gamma \bsl X$ and  prove its convergence in a half space using Iwasawa decomposition of the Haar measure on $G$. A crucial tool is a description of the vertices of $X$ in terms of Iwasawa cells. We also prove meromorphic continuation of the Eisenstein series. This requires us to construct an integral operator on the Tits building $X$ and a truncation operator for the Eisenstein series. We also develop the functional analytic framework necessary for proving meromorphic continuation in our setting, by refining and extending Bernstein's Continuation Principle.
\end{abstract}


\maketitle
\section{Introduction}
Let $\mathrm{k}=\F_q$ be a finite field of cardinality $q$ and $G$ be a complete affine or hyperbolic rank 2 Kac--Moody group over $\mathrm{k}$.  Let $X=X_{q+1}$ be the Tits building of $G$. Then by \cite{CG}, $X$ is the $(q+1)$--homogeneous tree.  Let $\Gamma=P_1^-$ be the standard parabolic subgroup for the negative $BN$-pair associated to a choice of simple root. Then $\Gamma$ is a non-uniform lattice in $G$. We define Eisenstein series on $\Gamma\bsl X$. The definition of Eisenstein series for $\Gamma'=P_2^-$ works just as well with the obvious modifications.
Our Eisenstein series is a combinatorial  analog of the classical non-holomorphic Eisenstein series on the Poincar\'e upper half plane, where the Tits building $X$ of $G$ plays the role of the upper half plane.

To discuss our construction of Eisenstein series on Kac--Moody groups, we first consider the analog for $G = \mathrm{PGL}_2(\mathrm{k}((t^{-1})))$, where $\mathrm{k}((t^{-1}))$ denotes the field of formal Laurent series over $\mathrm{k}$. Let $K$ be a maximal compact subgroup of $G$, such as $K = \mathrm{PGL}_2(\mathrm{k}[[t^{-1}]])$, and let $\Gamma \leq G$ be a discrete subgroup, such as $\mathrm{PGL}_2(\mathrm{k}[t])$. One can define a combinatorial Laplace operator $T$, which is induced by the \textit{adjacency operator}, acting on functions defined on the vertices of the Tits building. The vertices of the Tits building are encoded in the coset space $G/K$. One may also consider $\Gamma$-automorphic functions on the quotient $\Gamma \backslash G/K$.

The \textit{spectrum} $\mathrm{Spec}(T)$ consists of eigenvalues of $T$ on $\Gamma \backslash G/K$. It is known that
$
\mathrm{Spec}(T) \subset \mathbb{R},
$
that $\mathrm{Spec}(T)$ is symmetric about the origin, contains a continuous part including the origin, and has finitely many discrete values (\cite{E}). Eisenstein series are eigenfunctions of $T$ on $\Gamma \backslash G/K$ whose eigenvalues correspond to the continuous spectrum of $T$.

 More precisely, let $X$ be the Bruhat--Tits tree of $G={\rm PGL}_2(\mathrm{k}((t^{-1})))$ and let $\Gamma={\rm PGL}_2(\mathrm{k}[t])$. Let $T$ denote the adjacency operator operating on functions on the vertices of $X$.  
Efrat showed in (\cite{E}) that for
$\Gamma={\rm PGL}_2({\mathbb F}_q[t])$, the discrete spectrum of the adjacency operator on
$\Gamma\backslash X$ consists only of
$\pm(q+1)$. Thus two one-dimensional eigenspaces exist, namely the constant 
functions with eigenvalue $(q+1)$ and the alternating functions
with eigenvalue $-(q+1)$. Efrat also showed that there are continuous spectra 
described explicitly by Eisenstein
series and parametrized by the interval $[-2\sqrt{q},2\sqrt{q}]$.  The constant and 
alternating eigenfunctions correspond to poles of  Eisenstein series
$E_{}(g,s)$ at
$s=1$ and $s=1-\pi i/\log(q)$, respectively. Efrat then
obtained a decomposition of
$L^2(\Gamma\backslash X)$ into the
$T$-invariant subspaces generated by  discrete 
$L^2$-eigenfunctions  and a family of continuous eigenfunctions of $T$ given by suitable Eisenstein series. Namely, he showed that $L^2(\Gamma\bsl X)=R\oplus E$ where $R$ is generated by the constant and alternating functions, and
$E$ is generated by a family of continuous eigenfunctions of $T$ that satisfy a functional equation.

Later, using the Bruhat--Tits tree, Nagoshi \cite{N} provided new examples of the Selberg trace formula for principal congruence subgroups of ${\rm PGL}_2(\mathrm{k}[t])$. He expressed the Selberg zeta function as a determinant of the Laplacian, which is composed of both discrete and continuous spectra. 

Previously, Harder \cite{H} defined Eisenstein series for Chevalley groups over function fields as a function field analog of Langlands' theory of Eisenstein series for semisimple algebraic groups.

These Eisenstein series converge in a half space. Harder proved analytic continuation which is simpler than Langlands' method in the number field case. He also proved that the Eisenstein series are rational functions and showed that they satisfy a functional equation.  Building on Harder's work, Li \cite{Li} developed a full theory of Eisenstein series for $GL_2$ over function fields. She studied the intertwining operators arising from constant Fourier coefficients, proved that they are rational and showed that they satisfy a functional equation. She developed the theory of spectral decomposition for automorphic eigenfunctions of a certain Hecke operator, writing them as a sum of an Eisenstein series and a cusp form.

In \cite{Morris}, Morris defined Eisenstein series for reductive groups $G$ over global function fields $F$. He determined Eisenstein series necessary for a detailed spectral decomposition of the Hilbert space $L_2(G(F)\backslash G({\bf A}))$, where ${\bf A}$ is a ring of ad\'eles of the function field $F$. He  investigated the analytic properties of the Eisenstein series, in particular, the proof of functional equations for them.

In this work, we consider an extension of the above results to rank 2 affine or hyperbolic Kac--Moody groups. Although Garland showed that affine Kac--Moody groups are central extensions of Chevalley groups over fields of formal Laurent series \cite{GarlandLoopGp},  there is no such structure theorem for hyperbolic Kac--Moody groups. For these groups we must develop a theory of  Eisenstein series from first principles.

The complete rank 2 Kac--Moody group  $G=G^{\lambda}_A(\mathrm{k})$ considered here was constructed in \cite{CG} using an integrable highest weight module $V^{\lambda}$ of the underlying Kac--Moody algebra, corresponding to a dominant  integral highest weight $\lambda$. The group $G$ is  locally compact and totally disconnected, and in particular it admits a Haar measure. Such Kac--Moody groups have a twin $BN$--pair corresponding to positive and negative roots and a corresponding Tits building, a homogeneous tree $X$, which is locally finite.  As mentioned at the beginning, the standard parabolic subgroup $\Gamma :=P_1^-$ of the negative $BN$--pair is a non-uniform lattice subgroup of $G$ \cite{CG, RR}, analogous to ${\rm SL}_2(\mathbb{Z})$ in ${\rm SL}_2(\mathbb{R})$ and ${\rm SL}_2(\mathrm{k}[t])$ in ${\rm SL}_2(\mathrm{k}((t^{-1})))$.  That is, $\Gamma$ is a discrete subgroup of $G$ with finite covolume $\mu(\Gamma \backslash G)$ relative to a Haar measure $\mu$ on $G$ and has non-compact quotient $\Gamma \backslash G$. It follows that $\Gamma$ acts on the tree $X$ with finite vertex stabilizers (see \cite{CG} and \cite{AC}).  We note that the twin $BN$-pair structure on $G$ admits a twin structure $X^\pm$ on the Tits building. However, we will not use the twin structure here and we denote the Tits building as $X:=X^+$.

We define Eisenstein series for the combinatorial Laplacian on the Tits building of $G$. This allows us to use the Haar measure to establish the convergence of the Eisenstein series.  
To define Eisenstein series, we start with the unipotent part of the minimal parabolic subgroup for the spherical $BN$--pair, choose a lattice subgroup of $G$ and define a discrete eigenfunction (quasi-character) on the spherical torus. This function is extended to all of $G$ via the Iwasawa decomposition. We then average over an appropriate coset relative to a subgroup of the stabilizer of the standard apartment.

For our  Kac--Moody group $G$,  the Weyl group is infinite. To determine the constant term and to prove convergence of Eisenstein series for the rank 2 Kac--Moody group $G$, we avoid integrating over infinitely many Bruhat cells by associating a spherical building for $G$ with respect to a finite `spherical' Weyl group with the corresponding Bruhat decomposition:
$$G =  \mathcal{B}\sqcup \mathcal{B}w_{1}\mathcal{B} =  \mathcal{B}\sqcup \mathcal{B} w_{2}\mathcal{B},$$
 where $\mathcal{B}$ is the stabilizer of the end of the fundamental apartment of the Tits building (see also \cite{CG}).

 Since the Weyl group in our Kac--Moody setting is infinite, determining the constant term and proving the convergence of Eisenstein series for $G$ requires avoiding integration over infinitely many Bruhat cells. We achieve this by associating a spherical building to $G$, using a finite ``spherical'' Weyl group and the corresponding Bruhat decomposition:
$$
G = \mathcal{B} \sqcup \mathcal{B} w_1 \mathcal{B} = \mathcal{B} \sqcup \mathcal{B} w_2 \mathcal{B},
$$
where $\mathcal{B}$ is the stabilizer of the end of the fundamental apartment of the Tits building (see also \cite{CG}).

There are several other ingredients that are crucial to our study of Eisenstein series. Our results depend heavily on the structure of the fundamental domain for a non-uniform lattice $\Gamma\leq G$. We will work with the  lattice $\Gamma=P_1^-$, whose fundamental domain on the Tits building is a single vertex to which one cusp ( a semi-infinite ray) is attached.

For meromorphic continuation of Eisenstein series in Section ~\ref{trunc}, we use an analog of the classical truncation operator  due to Arthur \cite{A}. We also construct integral operators on the  Tits building $X$  in Section~\ref{operators}. With minor modifications, we can  obtain Eisenstein series relative to other lattices in $G$, such as the minimal parabolic subgroup $B^-$. 
It would be interesting to determine if there are cusp forms on $\Gamma\bsl X$ for $\Gamma$ a congruence subgroup of a lattice, as constructed in \cite{AC}.

To develop the functional analytic framework necessary for proving meromorphic continuation in our setting, we expand the results of \cite{BL},  as detailed in Section~\ref{appendix}. 
  We use a refinement of Selberg's method for the meromorphic continuation of Eisenstein series, due to Bernstein (see \cite{BL}, \cite{Garr1, Garr2}).
 Specifically, our approach relies on the \emph{Continuation Principle} (Theorem~\ref{bernstein}) and the \emph{Compact Operator Criterion} (Corollary~\ref{compactop}), both presented in the Appendix (Section~\ref{appendix}). These results require a detailed analysis of the extension  from weak holomorphy to strong holomorphy for families of functions,  also addressed in Section~\ref{appendix}.

Our setting, where the Tits building is one dimensional, allows us to simplify the  proof of meromorphic continuation in  the classical case.  In particular, the quotient graph $\Gamma \bsl X$ has  the simple structure of a semi-infinite ray. This allows us to work with an exact fundamental domain, rather than a Siegel set as in the classical case. We find that our analog of the truncation operator (Section~\ref{trunc}) is identically zero on the exact fundamental domain. Therefore, it is  a compact operator. While it is possible to define Siegel sets and related notions, this turns out not to be required for our proof of meromorphic continuation of Eisenstein series.

There are certain challenges associated with generalizing our results to higher rank. Our methods crucially depend on the relatively simple structure of the root system of a rank $2$ Kac--Moody algebra (see, for example, page 38 of \cite{CG}).  Determining the structure of the root systems of higher rank Kac--Moody algebras would need to be done on a case-by-case basis. Moreover, our results also depend on a spherical $BN$-pair for rank 2 Kac--Moody groups (constructed in \textit{op.cit.}). Such spherical $BN$-pairs do not exist in higher rank (see \cite{Ti4}, \cite{We}).  Our techniques for meromorphic continuation of Eisenstein series use harmonic analysis on the Tits building, which in rank 2 is a tree. There are further complications in extending these results to higher rank Tits buildings. Such a theory is not currently well developed for the  buildings associated to higher rank hyperbolic Kac--Moody groups.

Garland extended the classical theory of Eisenstein series to affine Kac--Moody groups (\cite{Ga1}, \cite{Ga2},  \cite{Ga5}, \cite{Ga6}, \cite{Ga7}, \cite{Ga8}). 
Lee and Lombardo (\cite{LL}) studied the constant terms of Eisenstein series on affine Kac--Moody groups over function fields  with finite constant fields. 
Liu (\cite{Liu}) made generalizations to affine Kac--Moody groups over number fields. Braverman and Gaitsgory in (\cite{BG}) developed a theory of \emph{geometric Eisenstein series}  for affine Kac--Moody groups in the framework of the geometric Langlands correspondence. This involves  a geometric reformulation of number theoretic and representation theoretic notions in terms of algebraic curves and vector bundles. Outside the affine case, Carbone, Lee and Liu (\cite{CLL}) defined Eisenstein series for rank $2$ hyperbolic Kac--Moody groups over $\mathbb{R}$ and established its convergence almost-everywhere.

The theory of Eisenstein series on Kac--Moody groups also has interesting intersections with mathematical physics. For example, in \cite{FK} and \cite{FGKP}, the authors studied Eisenstein series, their Fourier coefficients for $E_9$, $E_{10}$ and $E_{11}$, along with their applications in string theory.

The following questions remain open. We conjecture that our Eisenstein series on rank 2 Kac--Moody groups is a rational function. It would then remain  to determine its poles and  Fourier coefficients. 
 We conjecture that the only eigenvalues $\lambda$ with $|\lambda|>2q$ for which the 
corresponding eigenfunctions lie in $L^2$ are 
$\lambda=\pm (q+1)$.
There are further open questions about the full spectrum of the adjacency operator and if the discrete eigenfunctions correspond to poles of the Eisenstein series.

The authors  wish to thank the referee for their careful and systematic reading and  helpful comments which  improved the paper. We  are grateful to  Steve Miller and Peter Sarnak  for guidance and encouragement. We thank Howard Garland, Kyu-Hwan Lee and Dongwen Liu who shared many of their ideas and helped with earlier stages of the project.  We also thank Dmitry Gourevich  for helpful comments.

\section{Rank 2 Kac--Moody groups }

Let $I=\{1,2\}$ and
$A=(a_{ij})_{i,j\in I}$
be the symmetric generalized Cartan matrix defined by $a_{ii}=2$ and for $i\ne j$ $a_{ij}=a_{ji}=-m$ for all $i,j\in I$ with $m\ge 2$.
When $m=2$, the matrix $A$ is  of affine type, and when $m\geq 3$, it  is  of hyperbolic type.
 In this section, we summarize the construction of a Kac--Moody algebra $\mathfrak{g}=\mathfrak{g}(A)$, and a complete Kac--Moody group $G=G_{A}(\mathrm{k})$ over $k=\F_q$ associated with $A$.  We also describe its associated Tits building, which is a homogeneous tree of degree $q+1$.

\subsection{Kac--Moody algebras}\label{kmalg}
 
Let $\frak{h}$ be a $\C$-vector space and let $\langle\cdot,\cdot\rangle: \frak{h}^*\times\frak{h}\to \C$ denote the natural pairing,
that is, $\langle \varphi, h\rangle:= \varphi(h)$ for $\varphi\in\h^*$, $h\in\h$. 
Following \cite{K}, we say that $(\h,\Pi,\Pi^\vee)$ is a \emph{realization} of $A$ if

\begin{itemize}
\item $\h$ has dimension $n:= 4-\rank(A)$,

\item $\Pi=\{\alpha_1,\alpha_{2}\}\subseteq \frak{h}^*$ and
$\Pi^\vee=\{\alpha_1^\vee,\alpha_{2}^\vee\}\subseteq \frak{h}$ are  linearly independent sets, and

\item $\langle\alpha_j,\alpha_i^{\vee}\rangle=\alpha_j(\alpha_i^{\vee})=a_{ij}$ for $i,j\in I$.
\end{itemize}
We call the elements $\a_i$ \emph{simple roots} 
and $\a^\vee_i$ \emph{simple coroots}.

As in \cite[Theorem 9.11]{K}, the associated \emph{Kac--Moody algebra} $\frak{g}=\mathfrak g(A)$ is the Lie algebra over $\C$ with generating set  
$\frak{h}\cup\{e_i,f_i\mid i\in I\}$ and defining relations:
\begin{align*}
 [h,h']&=0;\\
[h,e_i]&=\langle\alpha_i,h\rangle e_i;&
[h,f_i]&=-\langle\alpha_i,h\rangle f_i;\\
\label{Lef} [e_i,f_i]&=\alpha_i^{\vee};&
[e_i,f_j]&=0;\\
(\ad e_i)^{-a_{ij}+1}(e_j)&=0;&
(\ad f_i)^{-a_{ij}+1}(f_j)&=0;
\end{align*}
for $h,h'\in\frak{h}$, and $i,j\in I$ with $i\ne j$.

The \emph{roots}  of $\frak{g}$ are the nonzero  elements
$\a\in\mathfrak h^*$ for which the corresponding \emph{root space}
$$\frak{g}_{\alpha} := \{x\in \frak{g}\mid[h,x]=\alpha(h)x\text{ for all $h\in \frak{h}$}\}$$
is nonzero. 
The simple roots $\a_i$ and their negatives have root spaces $\frak{g}_{\a_i}=\C e_i$ and $\frak{g}_{-\a_i}=\C f_i$, respectively.  
Every root $\alpha$ can be written in the form $\alpha=\sum_{i=1}^{\ell} k_i\alpha_i$ where the $k_i$ are integers with either all $k_i \ge 0$, in which case $\alpha$ is called {\it positive}, or all $k_i\leq 0$, in which case $\alpha$ is called {\it negative}. 

Let $Q$ be the root lattice $\mathbb{Z}\alpha_{1}\oplus \mathbb{Z}\alpha_{2}$, $Q_{+}=\mathbb{Z}_{\ge 0}\alpha_{1}\oplus \mathbb{Z}_{\ge 0}\alpha_{2}$ and $Q_{-}=-Q_{+}$. Let $\Delta\subset Q$ denote the set of roots, $\Delta_{+}\subset Q_{+}$ is the set of positive roots, and $\Delta_{-}\subset Q_{-}$ is the set of negative roots. For $i=1,2$, we define a simple reflection 
$$w_i(\alpha_j)
    :=
    \alpha_j - \alpha_j(\alpha_i^{\vee})\alpha_i.$$
     The simple reflections $w_i,\; i=1,2$, generate a subgroup $W=W(A)\subset \Aut(\mathfrak{h}^{\ast})$, called the {\it Weyl group}.  Hence, every element $ w \in W$ can be written as a product of simple reflections
\begin{equation*}
    w = w_{j_1} w_{j_2} \dots w_{j_p},
\end{equation*}
where $ j_1, \dots, j_p \in I$. A minimal such expression is called a {\it reduced decomposition} of $w$. We define a function 
$\ell\colon W\rightarrow \mathbb{Z}$
by
\begin{equation}\label{lengw}
    \ell(w) = p,
\end{equation}
and call it the {\it length} of  $w$. The set $\Delta^{\re}= W\Pi\subset\Delta$ is known as the set of \emph{real roots.} The remaining roots $\Delta\backslash \Delta^{\re}$ are called
\emph{imaginary roots.} We denote the set of positive (resp. negative) real roots by $\Delta^{\re}_{+}=\Delta^{\re} \cap \Delta_{+}$ (resp. $\Delta^{\re}_{-}=\Delta^{\re}\cap \Delta_{-}$).

The Lie algebra $\mathfrak g$ has a root  space decomposition
$$\frak{g}=\h\oplus\bigoplus_{\a\in\Delta}\frak{g}_\a$$ 
and a triangular decomposition by (\cite[Theorem 1.2]{K})
$$\frak{g}=\frak{n}^- \oplus \frak{h} \oplus \frak{n}^+,$$ 
where 
$\frak{n}^+ =\bigoplus_{\alpha\in\Delta^+}\frak{g}_{\alpha}$
and
$\frak{n}^- = \bigoplus_{\alpha\in\Delta^-}\frak{g}_{\alpha}$.

 The space $\mathfrak{h}^{\ast}$ can be equipped with a partial order $\le $ defined as: for $\lambda, \mu \in \mathfrak{h}^{\ast}$, we write $\mu\le \lambda$ if and only if $\lambda-\mu\in Q_{+}$. Similarly, we can define a partial order on $\mathfrak{h}$, which we denote by the same symbol $\le $, setting $Q^{\vee}_{+}=\mathbb{Z}_{\ge 0}\alpha_{1}^{\vee}\oplus \mathbb{Z}_{\ge 0}\alpha_{2}^{\vee}$ and imposing the same defining condition as above. An element $\lambda\in \mathfrak{h}^{\ast}$ is {\it integral} if $\langle \lambda,\alpha_{i}^{\vee}\rangle\in \mathbb{Z}$, is {\it dominant} if $\langle \lambda,\alpha_{i}^{\vee}\rangle \ge 0$, and is {\it regular} if $\langle \lambda,\alpha_{i}^{\vee}\rangle \ne 0$, for all $i\in I$. 
The {\it weight lattice} $\Lambda\subset\mathfrak{h}^{\ast}$ is defined to be
 $$\Lambda:=\{\lambda\in \mathfrak{h}^{\ast}\mid \langle \lambda, \alpha_i^{\vee}\rangle \in\Z,\ i=1,2\}.$$
 Let $\Lambda^{\vee}=\operatorname{Hom}_{\mathbb{Z}}(\Lambda, \mathbb{Z}) $ be the  coweight lattice. We denote by $\Lambda_{+}$ the set of dominant weights and $\Lambda_{reg}$ the set of regular weights. Similarly we define the sets $\Lambda_{+}^{\vee}$ and $\Lambda_{reg}^{\vee}$.

\subsection{Highest weight representations}
Let $\mathcal{U}=\mathcal{U}(\mathfrak{g})$ be the universal enveloping algebra of $\mathfrak{g}$. The triangular decomposition of $\mathfrak{g}$ yields the decomposition 
$$\mathcal{U}=\mathcal{U}(\mathfrak{n}^{+})\otimes \mathcal{U}( \mathfrak{h})\otimes \mathcal{U}( \mathfrak{n}^{-}),$$ 
where $\mathcal{U}(\mathfrak{n}^{+})$ and $\mathcal{U}(\mathfrak{n}^{-})$ are the universal enveloping algebras of $\mathfrak{n}^{+}$ and $\mathfrak{n}^{-}$, respectively.

For $\lambda\in \Lambda_{+}$, a $\mathfrak{g}$-representation $V^{\lambda}$ over $\mathbb{C}$ is a {\it highest weight representation} with the {\it highest weight} ${\lambda}\in \mathfrak{h}^{*}$ and a {\it highest weight vector} $v_{\lambda}$ if:
\begin{enumerate}
\item[(i)] $\mathfrak{n}^{+}\cdot v_{\lambda}=0$,
\item[(ii)] $h\cdot v_{\lambda}=\lambda(h)v_\lambda$, for all $h\in \mathfrak{h}$, 
 \item[(iii)] $V^{\lambda}=\mathcal{U}\cdot v_{\lambda}$,\;\text{ through the induced action of $\mathcal{U}$ on $V^{\lambda}$. }
  \end{enumerate}
 
 Moreover, the representation $V^{\lambda}$ is said to be an {\it integrable} highest weight representation if: 
\begin{enumerate}
\item[(iv)] for all $i\in I$, the elements $e_{i}$ and $f_{i}$ act as locally nilpotent operators on $V$, that is, for each $v\in V^{\lambda}$ and $i \in I$ there exist positive integers $M=M(v)$ and $N=N(v)$ such that $e_{i}^{M}\cdot v=f_{i}^{N}\cdot v=0$.
 \end{enumerate}

The representation $V^{\lambda}$ has a {\it weight space decomposition} 
\begin{eqnarray}
V^{\lambda}=\displaystyle{\bigoplus_{\mu\in \mathfrak{h}^{\ast}}V^{\lambda}_{\mu}},
\end{eqnarray} 
where $V^{\lambda}_{\mu}=\{v\in V^{\lambda}\mid h\cdot v=\mu(h)v\; \text{ for all }h\in \mathfrak{h}\}.$ 
 Denote by $P_{\lambda}$ the set of weights of $V^{\lambda}$.  
 The set $P_{\lambda}$ inherits the partial order from $\mathfrak{h}^{\ast}$. Each $\mu\in P_{\lambda}$ satisfies $\mu\le \lambda$, which means $ \lambda-\mu=\sum_{i\in I}n_{i}\alpha_{i}$ with $n_{i}\in \mathbb{Z}_{\ge 0}$, for all $i\in I$.

\subsection{Minimal Kac--Moody groups}\label{relations}
In this work, we use the minimal Kac--Moody group and its completion from \cite{CG}. (See also \cite{CLL}.) We describe the former group below and leave the latter for Section \ref{completion}.

 Let $n\in \mathbb Z$ and $h\in \mathfrak{h}$. We set
$\binom{h}{n}:=\frac{h(h-1)\dots (h-n+1)}{n!}.$
Let
${\mathcal U}_{\mathbb Z}\subset {\mathcal U}$ be the ${\mathbb Z}$-subalgebra generated by
$e_i^n/n!$, $f_i^n/n!$, and $\binom{h}{n}$, for $i=1,2$,  $h\in\Lambda^{\vee}$ and $n\geq 0$. Then
${\mathcal U}_{\mathbb Z}$ is a ${\mathbb Z}$-form of ${\mathcal U}$, that is, ${\mathcal U}_{\mathbb Z}$ is a subring and the canonical map
${\mathcal U}_{{\mathbb Z}}\otimes{\mathbb C}\longrightarrow {\mathcal U}$ is bijective. Set
$\mathfrak{g}_{\mathbb Z}=\mathfrak{g}\cap {\mathcal U}_{{\mathbb Z}}.$
For a field $\mathrm{k}$, let ${\mathcal U}_{\mathrm{k}}={\mathcal U}_{\mathbb Z} \otimes \mathrm{k}$ and $\mathfrak{g}_{\mathrm{k}}=\mathfrak{g}_{\mathbb Z} \otimes \mathrm{k}$.

Let $\lambda$ be a dominant integral regular weight and let $V^{\lambda}$ be the corresponding integrable highest weight representation with a fixed highest weight vector $v_{\lambda}$ as introduced in the previous subsection.
Let
$$V^{\lambda}_{\mathbb Z}={\mathcal U}_{\mathbb Z}\cdot v_{\lambda}.$$

Then $V^{\lambda}_{\mathbb Z}$ is a ${\mathbb Z}$-form of $V^{\lambda}$ as well as a
${\mathcal U}_{\mathbb Z}$-module.
Similarly, $V^{\lambda}_\mathrm{k}:= \mathrm{k}\otimes_{\mathbb Z} V^{\lambda}_{\mathbb Z}$
is a $\mathfrak{g}_{\mathrm{k}}$ and ${\mathcal U}_{\mathrm{k}}$-module.

A minimal Kac--Moody group over $\mathrm{k}$, which we denote by $G_{0}^{\lambda}(\mathrm{k})$, has a generating set
$$\langle \chi_{\pm \alpha_{i}}(u)\mid i=1,2,\; u\in \mathrm{k}\rangle \subset \mathrm{Aut}( V^\lambda_{\mathrm{k}}),$$ where for each $i=1,2$
\begin{eqnarray}\label{kmr1}
\chi_{\alpha_{i}}(u):=\exp({ue_{i}})=\sum_{n=0}^{\infty} u^n \dfrac{e_i^n}{n!}
\end{eqnarray}
and 
\begin{eqnarray}\label{kmr2}
\chi_{-\alpha_{i}}(u):=\exp({uf_{i}})=\sum_{n=0}^{\infty} u^n \dfrac{f_i^n}{n!}.\label{gendef}
\end{eqnarray}
Since $V^{\lambda}$ is integrable, the elements $e_i$ and $f_i$ are locally nilpotent on $V^{\lambda}$ and hence the expressions on the right-hand side of (\ref{kmr1}) and (\ref{kmr2}) are well defined
automorphisms of $V^{\lambda}_{\mathrm k}$. From now on, we drop $\mathrm{k}$ and $\lambda$ from the notations and just write
$G_{0}=G_{0}^{\lambda}(\mathrm{k})$.

\subsection{Unipotent subgroups}\label{sugrps}

The unipotent subgroups generated by the elements below will be used in the next sections. See \cite{Ti2} and \cite{CER} for the full set of defining relations for $G_{0}$.

\noindent For each $i=1,2$ and $ u \in \mathrm{k}^\ast$, set
\[ \wgal_{i}(u) :=\chi_{ \alpha_i}(u)\chi_{- \alpha_i}(-u^{-1})\chi_{ \alpha_i}(u),  \quad \wgal_{i}:=\wgal_{i}(1) \quad \text{ and } \quad  h_{i}(u) :=\wgal_{i}(u)\wgal_{i}^{-1}.\]

The elements $\wgal_{i}$, $i=1,2$,  generate a group $\widetilde{W}$, and we have
\[ \wgal_{i}\chi_{\alpha_{j}}(u)\wgal_{i}^{-1}= \chi_{w_{i}\alpha_{j}}(\eta_{ij}u), \] where $\eta_{ij} \in \{ \pm 1 \}$ for $ i,j=1,2$.

 The group $\widetilde{W}$ is a central extension of $W$.
That is, there is a surjective homomorphism $\epsilon: \widetilde{W}\to W$
which sends $\widetilde{w}_i$ to $w_i$, for  each $i$. The kernel of $\epsilon$ is an elementary
abelian $2$-group  generated by $\{(\widetilde{w}_i)^2\}$  (\cite[3.3]{Ti2}). Given
$\wgal\in \widetilde{W}$ and $w\in W$, we will say that $\wgal$ is a representative of
$w$ if $\epsilon(\wgal)=w$.  We will identify  $W$ (non-canonically)
with a subset (not a subgroup) of $\widetilde{W}$
which contains exactly one representative of every element of $W$. By a slight abuse of
notation, that set of  representatives will also be denoted by $W$.

The set $\{h_{i}(u)\mid i=1,2\;\text{and}\; u \in \mathrm{k}^*\}$ generate a subgroup $H$ which is isomorphic to $\mathrm{k}^{\ast}\times \mathrm{k}^{\ast}$. For $w\in W$, take a reduced word $w=w_{i_1}w_{i_2}\cdots w_{i_k}$ and set
$$ \widetilde{w}:= \widetilde {w}_{i_1} \widetilde {w}_{i_2}\cdots \widetilde{w}_{i_k}.$$   By abusing notation slightly, we will sometimes write $w$ for $\widetilde w$ and identify $\widetilde W$ with $W$.
For a real root $\alpha$, we fix $w \in W$ and $i \in I$ such that  $\alpha=w\alpha_{i}$. 
For $u\in \mathrm{k}$, define 
$$\chi_{\alpha}(u) :=\widetilde{w}\chi_{\alpha_{i}}(u)(\widetilde{w})^{-1}.$$
The element $\chi_{\alpha}(u)$ has an expression analogous to (\ref{gendef}), where $e_{i}$ and $f_{i}$ are replaced with those corresponding to positive and negative real root $\alpha$, respectively. 
We introduce the {\it root subgroup}
\begin{eqnarray}\label{rootsubgrp}
U_{\alpha}:=\{ \chi_{\alpha}(u)\mid u\in \mathrm{k}\}
\end{eqnarray}
corresponding to the real root $\alpha$. Finally, define a pair of unipotent subgroups
\begin{eqnarray*}
U^{+}_{0}& :=&\langle \chi_{\alpha}(u)\mid \alpha\in\Delta^{\re}_+,\  u\in \mathrm{k}\rangle=\langle U_{\alpha}\mid \alpha\in\Delta^{\re}_+\rangle\label{defU+},\\
U^{-}_{0}& :=&\langle \chi_{\alpha}(u)\mid \alpha\in\Delta^{\re}_{-},\ u\in \mathrm{k}\rangle=\langle U_{\alpha}\mid \alpha\in\Delta^{\re}_-\rangle,\label{defU-}
\end{eqnarray*}
where $ \Delta^{\re}_{\pm}=\Delta^{\re}\cap\Delta_{\pm}$.

For a real root $\alpha$ and $u\in \mathrm{k}^{\ast}$,
we write
\begin{eqnarray}
 \widetilde{w}_{\alpha}(u)=\chi_{\alpha}(u)\chi_{-\alpha}(-u^{-1})\chi_{\alpha}(u)   \end{eqnarray}
 and set $\widetilde{w}_{\alpha}=\widetilde{w}_{\alpha}(1)$ and for 
 \begin{eqnarray}\label{defh}
    h_{\alpha}(u)=\widetilde{w}_{\alpha}(u)\widetilde{w}_{\alpha} 
 \end{eqnarray}
where $u\in \mathrm{k}^{\ast}$.
We will need the following lemma for the upcoming sections.
\begin{lemma}
For any real root $\alpha$ and $u\in \mathrm{k}^{\ast}$, we have
\begin{eqnarray}
\chi_{\alpha}(u)=&\chi_{\alpha}(-u^{-1})h_{\alpha}(-u^{-1})\widetilde{w}_{\alpha}^{-1}\chi_{\alpha}(-u^{-1})
\end{eqnarray}
\end{lemma}

\begin{proof}
For any $t\in \mathrm{k}^{\ast}$, Equation \ref{defh} implies 
\begin{eqnarray*}
   \widetilde{w}_{\alpha}(t)= h_{\alpha}(t)\widetilde{w}_{\alpha}^{-1}.
 \end{eqnarray*}
 Using the definition of $\widetilde{w}_{\alpha}(t)$, we get
 $\chi_{\alpha}(t)\chi_{-\alpha}(t^{-1})\chi_{\alpha}(t)=h_{\alpha}(t)\widetilde{w}_{\alpha}^{-1}$
 and hence 
 \begin{eqnarray*}
 \chi_{-\alpha}(t^{-1})&=&\chi_{\alpha}(t)^{-1}h_{\alpha}(t)\widetilde{w}_{\alpha}^{-1}\chi_{\alpha}(t)^{-1}\\
 &=&\chi_{\alpha}(-t)h_{\alpha}(t)\widetilde{w}_{\alpha}^{-1}\chi_{\alpha}(-t).\\
 \end{eqnarray*}
 Finally, rewriting $t^{-1}$ as $t^{-1}=u\in \mathrm{k}^{\ast}$, and interchanging $-\alpha$ and $\alpha$ gives 
 \begin{eqnarray*}
 \chi_{\alpha}(u)&=&\chi_{\alpha}(-u^{-1})h_{\alpha}(-u^{-1})\widetilde{w}_{\alpha}^{-1}\chi_{\alpha}(-u^{-1}).
 \end{eqnarray*}

\end{proof}

\subsection{Completion of the minimal group}\label{completion}
In this subsection, we follow  \cite{CG, CLL}. The module  $V_{\mathrm{k}}^\lambda$
 admits a {\it coherently ordered basis} $\Psi=\{v_1,v_2,\dots\}$ which is defined as follows.  For each weight
$\mu$ of $V^{\lambda}$, we have $\mu=\lambda-\sum_{i=1}^\ell k_i\alpha_i$, where $k_i\in{\Bbb Z}_{\geq 0}$. The {\it depth} of
$\mu$ is defined to be
$$depth(\mu)\quad=\quad \sum_{i=1}^\ell k_i.$$
A basis $\Psi=\{v_1,v_2,\dots\}$ of $V^{\lambda}$ is called {\it coherently ordered} if
\begin{enumerate}
\item $\Psi$ consists of weight vectors.
\item If $v_i\in V^{\lambda}_{\mu}$, $v_j\in V^{\lambda}_{\mu'}$ and $depth(\mu')\ >\ depth(\mu)$, then $j>i$.
\item $\Psi\cap V^{\lambda}_{\mu}$ consists of an interval $v_k,v_{k+1},\dots,v_{k+m}$.
\end{enumerate}

We denote by $B_0$ the subgroup of $G_0$ consisting of the elements represented by upper triangular matrices with respect to $\Psi$.
For $t \in \mathbb Z_{>0}$, we let $V_t$ be the span of the $v_s \in \Psi$ for $s \le t$. Then $B_0 V_t \subset V_t$ for each $t$. Let $B_t$ be the image of $B_0$ in $\mathrm{Aut}(V_t)$. We then have surjective homomorphisms
\[ \pi_{tt'}: B_{t'} \longrightarrow B_t, \quad t' \ge t. \] We define $B$ to be the projective limit of the projective family $\{ B_t, \pi_{tt'} \}$. 

Now define a topology on $G_0$ as follows: for a base of neighborhoods of the identity, we take sets $U_t$ defined by 
\[ U_t :=\{ g \in G_0 : g v_i =v_i, \ i=1, 2, \dots , t \} \quad \text{ for } t \in \mathbb Z_{\ge 0}. \]
Let $g=(g_i)$ and $h=(h_i)$ be Cauchy sequences in $G_0$. By definition, for any $s$, there exists $N_1 >0$ such that $g_i g_j^{-1} (v) =v$ for $v \in V_s$ whenever $i,j > N_1$. We can find $t \ge s$ such that $\{g_i^{-1}(V_s) \}$ is contained in $V_t$ for any $i > N_1$, since $g_j^{-1} v = g_i^{-1} g_i g_j^{-1} v = g_i^{-1} v$ for $v \in V_s$ and $j > N_1$.  Take $N > N_1$ sufficiently large so that  $h_ih_j^{-1}v =v$ for any $v \in V_t$ whenever $i,j > N$.  Assume that $i, j >N$. Then we have
\[ g_i h_i h_j^{-1} g_j^{-1} (v) = g_i g_j^{-1} (v) = v \ \text{ for any} \ v \in V_s, \]
since $g_j^{-1}(v) \in V_t$. 
 This proves that $g_i h_i h_j^{-1} g_j^{-1} \in U_s$ whenever $i, j >N$. Therefore $gh=(g_i h_i)$ is Cauchy as well.

Define $G=G_{A}^{\lambda}(\mathrm{k})$ to be the completion of $G_0$, i.e. the equivalence classes of all Cauchy sequences of $G_0$, which is called the  complete (or  maximal) Kac--Moody group.

\subsection{$BN$-pairs and lattice subgroups }
We have the decomposition  $B=HU$  where $U$ is the completion of $U^{+}_{0}$. The subgroup $B\subset G$ is called a {\it minimal parabolic subgroup}. The group $B$  is compact and forms a profinite neighborhood of the identity in $G$ \cite{CER}. We will also write $B^+=B$. 
Suppose $N$ is the subgroup generated by $H$ and $\widetilde{w}_{i}$ for $i=1,2$.
Then we have $N=N_{G}(H)$.

In \cite{CG}, the authors showed that $G$ has Bruhat decomposition
$$G = \bigsqcup _{w\in W}\  BwB.$$
We define $B^-:=HU_0^{-}$ and write $U^-=U_0^-$ for convenience.   The {\it standard parabolic subgroups} are given as
$$P_i^{\pm} := B^{\pm} \sqcup   B^{\pm}w_iB^{\pm}, \qquad i=1,2.$$

The subgroup $B^-$ is a non-uniform lattice in $G$ (see Theorem 8.2 in \cite{CG}). This result is also obtained independently by R\'emy \cite{Re1}. It follows that for $i=1,2$, the negative parabolic subgroups $P_i^{-}=B^{-}\sqcup B^{-}w_iB^{-}$ are non-uniform lattices. There are bijective correspondences (\cite{Ti2})
\[
B^-\bsl G/B  \cong W,\quad
P_i^-\bsl G/B  \cong   W^+, \quad i=1,2,\]
where $W^+$ denotes an index 2 subgroup of $W$. 

From now and on, we often drop the `$+$' and refer to $P_i^+$ as $P_i$ for $i=1,2$ and $U_{0}^+$ as $U_{0}$.

\section{Decompositions of unipotent groups}\label{Udec}
 For $w\in W$, we set
\begin{eqnarray}
 S_{w}&:=&\{\beta\in \Delta_{+}\mid w\beta\in \Delta_{-}\}=\Delta_{+}\cap w^{-1}(\Delta_{-})\label{inverset1} \label{inverset2}
 \end{eqnarray}
and call it the {\it inversion set} associated to $w$. Then $S_{w}$ is a finite subset of $\Delta^{\re}_{+}$.
 
 Let $\alpha_{i}$ be a positive simple root and $w_{i}$ be the corresponding simple reflection for $i=1,2$. Suppose $w\in W$ is such that $w=w'w_{i}$ with $\ell(w)=\ell(w')+1$, then 
 \begin{eqnarray} \label{rem3.1}
S_{w}=\{\alpha_{i}\}\cup w_{i}S_{w'},
 \end{eqnarray}
 where $w_{i}S_{w'}=\{w_{i}\beta\mid \beta\in S_{w'}\}$.

We introduce the subgroup 
 \begin{eqnarray*}
 U_{0,w}:=\langle U_{\beta}\mid \beta\in S_{w}\rangle \subset U_0,
 \end{eqnarray*}
 where $U_{\beta}$ is the root subgroup corresponding to $\beta$ as defined in (\ref{rootsubgrp}). Since $S_{w}$ consists of finitely many elements, the group $U_{0,w}$ is a finite-dimensional subgroup of $U_{0}.$

Combining \eqref{rem3.1} with Lemma~6.3 of \cite{CG}, we obtain the following.
 \begin{lemma}\label{lemdec1}
For $w\in W$ such that $w=w'w_{i}$ with $\ell(w)=\ell(w')+1$, $i=1,2$, we have 
\begin{eqnarray}
U_{0,w}=(w_{i}^{-1}U_{0,w'}w_{i})\rtimes U_{\alpha_{i}}.
 \end{eqnarray}
\end{lemma}
 
We set
$$U_{0}^{w}:=U_{0}\cap w^{-1}U_{0}w.$$
\begin{lemma}\label{conlem1}
Let $w\in W$, then 
\begin{eqnarray}\label{sec3eq1}
U_{0}=U^{w}_{0}U_{0,w}=U_{0,w}U_{0}^{w}.
\end{eqnarray}
\end{lemma}

\begin{proof}
We prove $U_0 =  U_0^w U_{0,w}$ by  induction on $\ell(w)$.  
 For the simple reflections $w_{i}\in W$, $i=1,2$, the statement follows from Lemma~6.3 of \cite{CG}, and we can write an element $u\in U_{0}$ as 
\begin{eqnarray}\label{lem3.2eq1}
u=u^{w_{i}}u_{\alpha_{i}},
\end{eqnarray}
 for some $u_{\alpha_{i}}\in U_{\alpha_{i}}$ and $u^{w_{i}}\in U_{0}^{w_{i}}$.
We assume now that the assertion holds for all $w'\in W$ with $\ell(w')=n-1$.   Suppose $w\in W$
is such that $w=w'w_{i}$, where $w_{i}$ is a simple root reflection and $w'\in W$ with $\ell(w)=\ell(w')+1$. 
Continue to consider $u\in U_{0}$ with the decomposition (\ref{lem3.2eq1}) and write $u^{w_i}= w_i^{-1} \widetilde u w_i$ for $\widetilde u \in U_0$.  
 By the induction hypothesis, there exists  $\widetilde u_{w'}\in U_{0,w'}$ and $\widetilde u^{w'}\in U_0^{w'}$ such that 
\begin{eqnarray}\label{lem3.2eq3}
\widetilde u=\widetilde u^{w'}\widetilde u_{w'}.
\end{eqnarray}

Substituting (\ref{lem3.2eq3}) into (\ref{lem3.2eq1}), we get 
\begin{eqnarray}\label{lem3.2eq4}
u&=&[w_{i}^{-1}\widetilde u^{w'} \widetilde u_{w'}w_{i}]u_{\alpha_{i}}\nonumber\\
&=& [w_{i}^{-1}\widetilde u^{w'} w_{i}][w_{i}^{-1}\widetilde u_{w'}w_{i}]u_{\alpha_{i}}.
\end{eqnarray}
 For the simple root reflection $w_{i}$, $U^{w_{i}}$ is a normal subgroup of $U_{0}$ and $U^{w_{i}}=\langle U_{\beta}\mid \beta\in \Delta\setminus\{\alpha_{i}\}\}$ (see \cite[Lemma 6.3]{CG}).
This and the fact that  $\widetilde u\in U^{w_{i}}$ gives $w_{i}^{-1}\widetilde u^{w'} w_{i} \in U_0$. In particular, $w_{i}^{-1}\widetilde u^{w'} w_{i} \in U_0^w$.
 By Lemma~\ref{lemdec1}, $[w_{i}^{-1}\widetilde u_{w'}w_{i}]u_{\alpha_{i}} \in U_{0,w}$. This proves $u \in U_0^w U_{0,w}$ and thus $U_0 =  U_0^w U_{0,w}$. 

The other case follows from the equations
$$U_0 =U_0^{-1}= (U_0^w U_{0,w})^{-1}=(U_{0,w})^{-1} (U_0^w)^{-1}=U_{0,w} U_0^w.$$
\end{proof}

The decomposition (\ref{sec3eq1}) of $U_{0}$ in the above lemma extends to its completion $U$. More precisely, 

\begin{corollary}\label{sec3cor1}
For $w\in W$ we have 
$$U=U^{w}U_{w}=U_{w}U^{w}$$
where $U^{w}=U\cap w^{-1}Uw$ and $U_{w}= U_{0,w}$ in $U$.
\end{corollary}
The above decomposition of $U$ is also obtained in slightly different contexts in \cite[Corollary 6.5]{Ga2} and \cite[Lemma 6.1.3]{Ku}.

\section{The Tits building of $G$, a tree}
We can associate Tits buildings $X^{+}$ and $X^{-}$ to each  $BN$--pair  $(G,B^+,N)$ and  $(G,B^-,N)$ of $G$, respectively. The buildings
$X^+$ and $X^-$ are isomorphic as chamber complexes (see \cite[Appendix TKM]{DJ}). We will consider only $X^{+}$ and denote it by $X$. 
 The building $X$ is a simplicial complex of dimension $1$ and $X=X_{q+1}$, the $q+1$-homogeneous tree.

\subsection{Vertices and edges of $X$}
The set $VX$ of vertices of $X$ is defined to be the set of conjugates of $P_1$ and $P_2$ in $G$.  Vertices $Q_1$ and $Q_2$ are connected by an edge if the intersection $Q_1 \cap Q_2$ contains a conjugate of $B$. Let $EX$ denote the set of edges of $X$. It is known that $P_1$ and $P_2$ are not conjugate and that $P_i$, $i=1,2$ and $B$ are self-normalizing in $G$. Thus we have the following bijective correspondences \cite[Section 9]{CG} assigning $gHg^{-1}$ to $gH$ for $H=P_i$, $i=1,2$ and for $H=B$:
\begin{eqnarray*}
VX&\cong& G/P_1\sqcup G/P_2,\\
EX&\cong& G/B.
    \end{eqnarray*}
Using these bijections, we identify the elements of $VX$ and $EX$ with the corresponding cosets. 

There are natural projections
on cosets induced by the inclusion of $B$ in $P_1$ and $P_2$:
$$\pi_i:G/B\longrightarrow G/P_i,\ i=1,2.$$ For a vertex $v \in G/P_i$, $i =1,2$, define $St^X(v):=\pi_i^{-1}(v)$ to be the set  of edges with origin  $v$. 
The following describes how
the cosets
$Bw_1 B$ and
$Bw_2 B$ are indexed modulo
$B$:
\[Bw_{i} B/B=\{\chi_{i}(t) w_{i} B/B \mid t\in \mathrm{k}\} \quad \text{for } i=1,2.
\]
It follows that the edges emanating from $P_1$ and $P_2$ may be indexed as follows:
$$St^X(P_i)=\{B\}\sqcup\{ \chi_{i}(t)w_{i} B/B \mid  t\in \mathrm{k}\} \quad \text{for } i=1,2.$$
The sets $St^X(v)$ of other vertices $v$ are obtained by translating (conjugating) these. 
We have the diagram of tree associated with $G$ in Figure \ref{fig1}, where we use cosets to label vertices.

The coset  $1\cdot B$ corresponds to an edge called the {\it standard simplex} with vertices $P_{1}$ and $P_{2}$. Apartments in $X$ are bi-infinite paths. The {\it
standard apartment}, denoted by ${\mathcal{A}}_0$, in $X$ consists of all Weyl group translates of the standard simplex.  For a parabolic subgroup $P$ of $G$, we will denote the corresponding simplex (vertex or edge) by $\sigma_{P}$.
\begin{figure}[h!]
    \centering
    \includegraphics[width=0.66\textwidth]{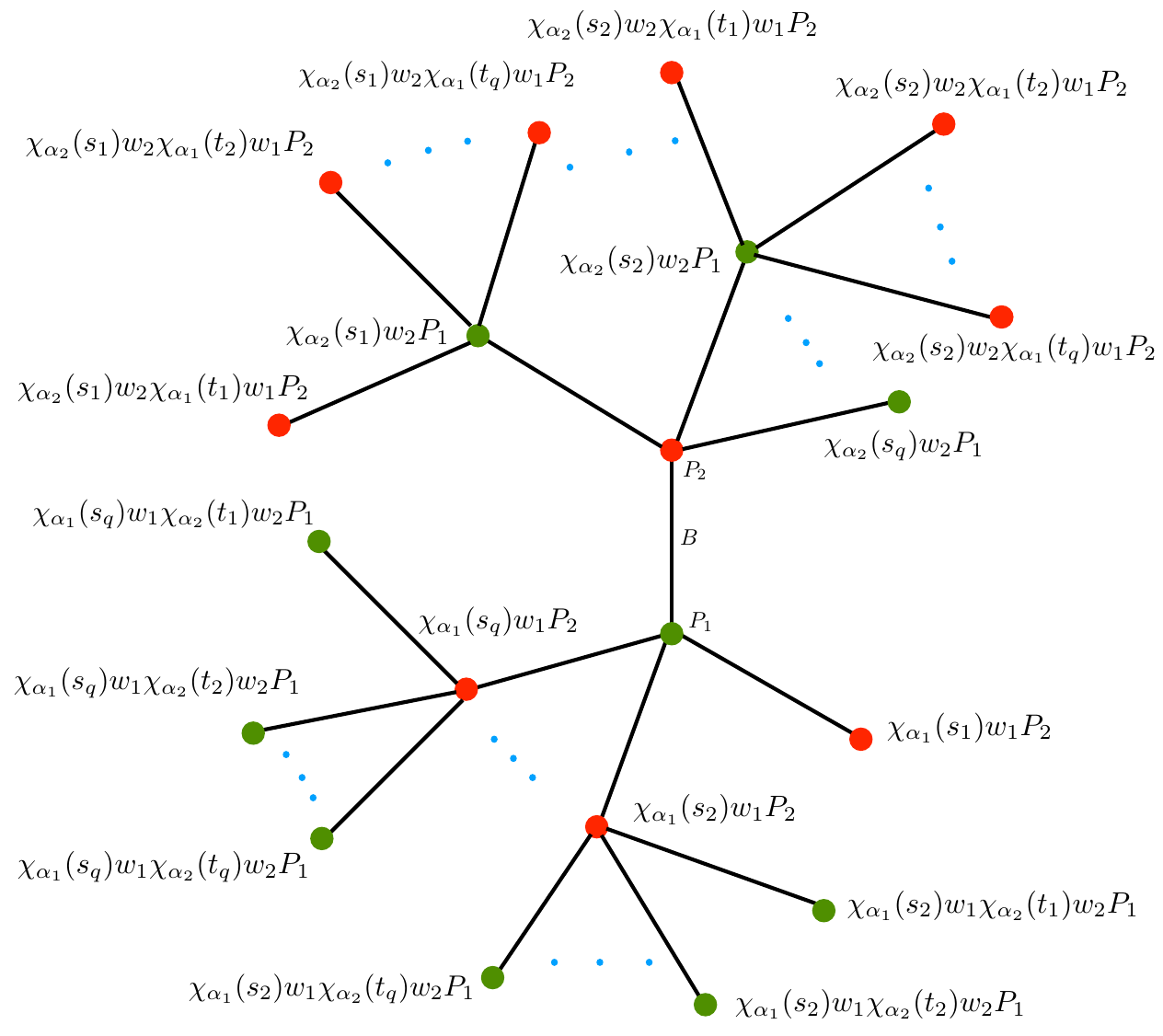}
    \caption{ Tree of $G$ with coset labels. The `dots' indicate that the tree is constructed over the field with $q$ elements.}
    \label{fig1}
\end{figure}

\subsection{Adjacency and labeling}
Define the {\it distance function} $d$ on $X$ to be given by  
$$d\colon VX\times VX\longrightarrow \mathbb{Z}_{\ge 0} $$
such that for all $v, v'\in VX$
$$d(v,v')=\text{number of edges in a shortest path connecting $v$ and $v'$}.$$ The vertices $v$ and $v'$ are said to be {\em adjacent} if $d(v, v')=1$.
For  a vertex $v \in VX$, let $\Omega_{v}$ be the set of vertices adjacent to $v$. That is,
$$\Omega_{v}:=\{y\in VX\mid d(y,v)=1\}.$$
 The following lemma gives an explicit labeling of $\Omega_{v}$.

\begin{lemma}\label{adjverloc}
Let $v \in VX$ be a vertex corresponding to  $gP_{j}$ ($j =1,2$). Then $\Omega_{v}$ is a disjoint union of the two sets given by 
\[\Omega_{v}^{1}:=\{gP_{3-j}\} \qquad \text{ and } \qquad  \Omega_{v}^{q}:= \{g\chi_{\alpha_{j}}(s)w_{j}P_{3-j}\mid s\in \mathrm{k}\}.
\]
\end{lemma}

The statement of Lemma~{\ref{adjverloc}} can be understood through the following local picture in Figure \ref{fig2} of the vertex $v$ and its adjacent vertices when labeled in terms of the corresponding cosets of the parabolic subgroups.

\begin{figure}[h]
    \centering
 \includegraphics[width=0.57\textwidth]{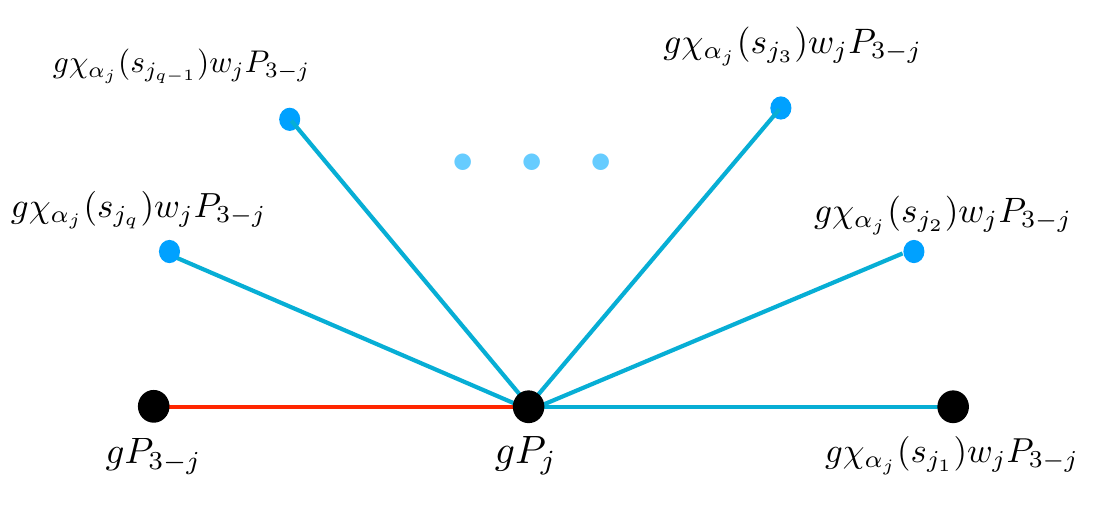}
    \caption{Local Picture}
    \label{fig2}
\end{figure}

\begin{proof}[Proof of Lemma~\ref{adjverloc}]   
Note that
$$ St^X(v)= \{gB\}\sqcup\{ g\chi_{\alpha_j}(s)w_j B/B\mid s\in \mathrm{k}\}. 
$$
Since we have 
\begin{eqnarray*}
&& g\chi_{\alpha_j}(s)w_jB(g\chi_{\alpha_j}(s)w_j)^{-1} \\
&\subset&  g\chi_{\alpha_j}(s)w_j \left( P_{j}\cap P_{3-j} \right)(g\chi_{\alpha_j}(s)w_j )^{-1}\\
&=& \left( g\chi_{\alpha_j}(s)w_jP_{j}(g\chi_{\alpha_j}(s)w_j)^{-1} \right)\cap \left( g\chi_{\alpha_j}(s)w_jP_{3-j}(g\chi_{\alpha_j}(s)w_j )^{-1}\right)\\
&=&gP_{j}g^{-1}\cap (g\chi_{\alpha_{j}}(s)w_{j})P_{3-j}(g\chi_{\alpha_j}(s)w_{j} )^{-1},
\end{eqnarray*}
the vertex $v$ is connected to the vertices (corresponding to) $g\chi_{\alpha_{j}}(s)w_{j}P_{3-j}$ for $s \in \mathrm k$. Similarly, $gBg^{-1}\subset gP_{j}g^{-1}\cap gP_{3-j}g^{-1}$, and $v$ is also connected to $gP_{3-j}$.
This shows $\Omega_{v}=\Omega_{v}^{1}\sqcup \Omega_{v}^{q}$ as desired.
\end{proof}


\section{Spherical Tits system and Iwasawa decomposition for $G$}
As before,  $G$ denotes a rank 2 complete Kac--Moody group over $\mathrm{k}=\mathbb{F}_q$. Let $X$ be the Tits building of $G$, the homogeneous tree of degree $q+1$.

\subsection{Spherical Bruhat decomposition}

It follows from the definition of $X$ that the group $G$ acts on $X$ by conjugation. Let $\mathcal{A}$ be an apartment in $X.$ Fix a `base edge' of $\mathcal{A}$, denoted $\beta(\mathcal{A})$,  and let $g\in G$ be such
that 
\begin{equation} \label{eqn-AB}
g\cdot \beta(\mathcal{A}) = B,
\end{equation}
where $B$ is set to be the base edge $\beta (\mathcal{A}_0)$ of the standard apartment $\mathcal{A}_0$. Such a $g\in G$ exists since $G$ acts transitively on the set of edges \cite{CG}. By Lemma \ref{adjverloc}, we may assume that 
\[ g\cdot Q_1  =  P_1\quad \text{and} \quad g\cdot Q_2  =  P_2, \]
where $Q_1$ and $Q_2$ are the endpoints of $\beta(\mathcal A)$ that are conjugate to $P_1$ and $P_2$, respectively.

We let $\mathcal{A}^{+}$ be the subcomplex of $\mathcal{A}$ 
consisting of the ray joined to $\beta(\mathcal{A})$ at $Q_2$ and
not containing $\beta(\mathcal{A}).$ We let $\mathcal{A}^{-}$ be the
subcomplex of $\mathcal{A}$ consisting of $\mathcal{A}\setminus (\mathcal{A}^{+}\sqcup
\beta(\mathcal{A})).$  Then we have 
$$
\mathcal{A}  =  \mathcal{A}^{+}\sqcup \mathcal{A}^{-}\sqcup \beta(\mathcal{A}).
$$
We present a finite part of $\mathcal{A}_{0}$ and $\mathcal{A}_{0}^{\pm}$ in Figure \ref{fig3.0}.

\begin{figure}[h]
    \centering
    \includegraphics[width=0.8\textwidth]{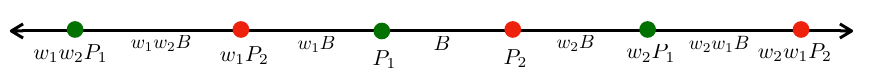}
    \caption{ The labeling of the standard apartment with its positive and negative halves}
    \label{fig3.0}
\end{figure}

We have the following lemma and its corollaries from Lemma~14.1 and Corollaries~14.1 and 14.2 of \cite{CG}, which are slightly modified for our purposes.

\begin{lemma}[\cite{CG}]\label{doubletransitivity}
Let $\mathcal{A}$, $\mathcal{A}_{0}$, $\mathcal{A}^{\pm}$, and $\mathcal{A}_{0}^{\pm }$ be as introduced above and let $g \in G$ be as in \eqref{eqn-AB}. We can find elements
 $b_{1},$ $b_{2}\in U$ such that 
\[
b_{1} { g} \mathcal{A}^{+}  =  \mathcal{A}_{0}^{+}  
\quad \text{ and } \quad  
b_{2}b_{1} { g} \mathcal{A}^{-} = \mathcal{A}_{0}^{-},  
\]
and furthermore $b_{2}$ stabilizes $\mathcal{A}_{0}^{+}$ pointwise so that \ \ $b_{2}b_{1} { g} \mathcal{A}^{+} = \mathcal{A}_{0}^{+}$.
\end{lemma}

Define an equivalence relation on the  set of paths in $X$ as follows: two paths in $X$ are equivalent if their intersection is infinite. An {\it end}  of $X$ is defined to be an equivalence class of semi-infinite rays in $X$. Then there is a 1-1 correspondence between apartments in $X$ and pairs of ends of $X$.

\begin{corollary}[\cite{CG}]\label{TransPi}
The group  { $G$} acts doubly transitively on the set of ends of $X.$
\end{corollary}

\begin{definition}
Define $\mathcal{B}_{1 }$ (resp. $\mathcal{B}_{2}$) to be the stabilizer of the end containing $\mathcal{A}_{0}^{+ }$ (resp. $\mathcal{A}_{0}^{- }$). 
\end{definition}

\begin{corollary}[\cite{CG}]\label{cor5.3}  
We have 
$$
G  =  \mathcal{B}_{i }\sqcup \mathcal{B}_{i }w_{j}\mathcal{B}_{i } \qquad \text{ for }  \ i,j=1,2.
$$
\end{corollary}
This decomposition of $G$ is called  the {\it spherical Bruhat decomposition} of $G$, which we will use in the computation of constant term of the Eisenstein series in Section~\ref{conster}.

\subsection{Explicit description of $\mathcal{B}_{i}$} \label{sphB}
In this subsection we give a decomposition of $\mathcal{B}_i$, $i=1,2$, which will be used later. The set of real roots $\Delta^{\re}$ can be written as a disjoint union
 \begin{eqnarray}\label{rrtbrnch}
 \Delta^{\re}=\Delta^{\re}_1\sqcup\Delta^{\re}_2,
 \end{eqnarray}
 where $\Delta^{\re}_1$ and $\Delta^{\re}_2$ are given by  
\begin{eqnarray*}
  \Delta^{\re}_1\ &:=&\ \{ -\alpha _{2},\ -w_{2}\alpha _{1},\ -w_{2}w_{1} \alpha
_{2},\dots \}\cup \{ \alpha _{1},\ w_{1}\alpha _{2},\ w_{1}w_2 \alpha
_{1},\dots\},\\
\Delta^{\re}_{2}\ &:=&\ \{-\alpha_1,\ -w_1\alpha_2,\ -w_1w_2\alpha_1,\dots\}\cup\{\alpha_2,\ w_2\alpha_1,\
w_2w_1\alpha_2,\dots\}.
\end{eqnarray*}
Let $\mathcal{U}_i$ be the closure of the group  generated by $U_{\alpha}$ with  $\alpha   \in  \Delta^{\re}_i$, $i=1,2$.
  The group ${\mathcal U}_i$ is commutative by Lemma~14.2 in \cite{CG}\footnote{In \cite{CG}, real roots are called Weyl roots.}. 
 Since $$w_1\cdot \alpha_1=-\alpha_1,\quad w_2\cdot \alpha_2=-\alpha_2,$$
$w_1$ and $w_2$ interchange $\Delta^{\re}_{1}$ and $\Delta^{\re}_{2}$.

Define
$$\mathcal T: =\{(w_{1}w_2)^{n}\mid n\in {\mathbb Z}\}$$ to be the subset of $W$ consisting of even-length elements. Then we obtain the following immediate consequences from the definition of $\mathcal T$.

\begin{lemma}\label{Iwalemm1}
\hfill
\begin{itemize}
\item[(a)] For $i=1,2$ and $m\in\mathbb{Z}$, we have $w_{i}(w_{i}w_{3-i})^{m}=(w_{i}w_{3-i})^{-m}w_{i}.$
\item[(b)] The simple reflections $w_{i}$ normalize $\mathcal T$ for $i =1,2$.
\item[(c)] The subgroup $\mathcal T$ normalizes $\mathcal{U}_i$ for $i=1,2$.
\end{itemize}
 
\begin{proof}
We prove (a) and part (b) directly follows from it. 
For $i=1,2$ and $m\in\mathbb{Z},$
\begin{eqnarray*}
w_{i}(w_{i}w_{3-i})^{m}&=&(w_{3-i}w_{i})^{m-1}w_{3-i}
=(w_{i}w_{3-i})^{1-m}w_{3-i}\\
&=&(w_{i}w_{3-i})^{-m}(w_{i}w_{3-i})w_{3-i}
=(w_{i}w_{3-i})^{-m}w_{i}.
\end{eqnarray*}
Part (c) follows from $$
(w_{1}w_2)^{m}\Delta^{\re}_{1}=\Delta^{\re}_{1}\quad \text{and}\quad (w_{1}w_2)^{m}\Delta^{\re}_{2}=\Delta^{\re}_{2}
$$
for any $m\in\mathbb{Z}$.
\end{proof}
\end{lemma}
 
The following decomposition of $\mathcal{B}_{i }$ ($i=1,2$) is proven in \cite[Theorem 14.1]{CG} (See also \cite[Theorem 8.8]{CFF}). 
 Let $\displaystyle{\mathcal{B}_{\mathcal{I}}   :=  \bigcap _{w\in W}wBw^{-1}\subset B}$.
\begin{theorem} \label{thm5.4} For $i=1,2$, we have 
\begin{eqnarray*}
\mathcal{B}_{i }&=&\mathcal{B}_{\mathcal{I}}\mathcal{U}_{i}\mathcal T=\mathcal{B}_{\mathcal{I}}\mathcal T
\mathcal{U}_{i}=\mathcal{U}_{i}\mathcal{B}_{\mathcal{I}}\mathcal T=\mathcal T\mathcal{B}_{\mathcal{I}
}\mathcal{U}_{i}=\mathcal{U}_{i}\mathcal T\mathcal{B}_{\mathcal{I}}=\mathcal T\mathcal{U}_{i}\mathcal{B}_{\mathcal{I} }.
\end{eqnarray*}
\end{theorem}

\subsection{The Iwasawa Decomposition of $G$} \label{IwaG}

\begin{theorem} \label{sppolar}
{ For $i,j =1,2$, we have
\begin{eqnarray*}
    G&=&P_{j}\mathcal{B}_{i } =\mathcal{B}_{i }P_{j}.
\end{eqnarray*}
}

\begin{proof} We prove the assertion of theorem for ${\mathcal B}_{1}$ only. For ${\mathcal B}_{2}$, it follows similarly.
We set
$\beta=\beta(B)$. Then every end of $X$
is represented by a ray beginning at $\beta$. By Corollary~\ref{TransPi} the groups $P_i$, $i=1,2$
act transitively on the ends of $X$. Let
$g\in G$. Then by Lemma~\ref{doubletransitivity}, the end determined by $g \mathcal{A}_{0}^{+}$ is equivalent to the end 
determined by 
$b \mathcal{A}_{0}^{+}$ for some $b\in U\leq P_i$. So $b^{-1}g \mathcal{A}_{0}^{+}$ is 
equivalent to the end determined by 
$\mathcal{A}_{0}^{+}$. Therefore $b^{-1}g\in {\mathcal B}_{1}$, so $g\in K{\mathcal B}_{1}$, 
$K=P_1,P_2$. Similarly $g^{-1}b\in {\mathcal B}_{1}$, so $g\in  {\mathcal B}_{1}K$, 
$K=P_1,P_2$.

\end{proof}
\end{theorem}

Since $\mathcal{B}_{\mathcal{I}}\subset B \subset P_i$ for $i,j=1,2$, the following result is a straightforward consequence of Theorems~\ref{thm5.4} and \ref{sppolar}.

\begin{corollary}\label{Iwasawa} For $i,j=1,2$, the group $G$ admits the following decomposition: 
\[ G=\mathcal{U}_i\mathcal T P_j. \]
\end{corollary}
The above decomposition of $G$ will be called the {\it Iwasawa decomposition} of $G$. Indeed, if we fix $i$ and $j$ and write $\mathcal U =\mathcal U_i$ and $K= P_j$, the decomposition becomes \[ G = \mathcal U \mathcal T K, \] which can be considered as an analogue of the Iwasawa decomposition of a $p$-adic reductive group, and we may thus write $g\in G$ as 
$$
g=u_{g}t_{g}k_{g}, \quad u_{g}\in \mathcal{U},  \ t_{g}\in \mathcal T,  \ k_{g}\in K.
$$

\medskip

The following lemma shows that we have uniqueness of the $\mathcal T$--component. 

\begin{lemma}\label{uniquenessoft}  Let $g\in G$. For fixed $i$ and $j$, let  $\mathcal{U}=\mathcal{U}_i$ and $K=P_j$. If $g=utk=u't'k'$ with $u, u'\in \mathcal{U}$, $t, t'\in \mathcal T$ and $k,k'\in K$, then $t=t'$.
\end{lemma}

\begin{proof}
Let $\xi$ be a vertex on the standard apartment fixed by $K$.
Then
$$ut\cdot\xi=u't'\cdot\xi$$ so
$$t'^{-1}u'^{-1}ut\cdot\xi=\xi.$$
Thus
$$t'^{-1}u''t\in K$$
for $u''=u'^{-1}u$. But
$$t'^{-1}u''t=t'^{-1}t(t^{-1}u''t).$$
Writing $u'''=(t^{-1}u''t)$ and $t''=t'^{-1}t$, we have $t''u'''\cdot\xi=\xi$.

Let $\eta$ be a vertex sufficiently far out on the positive half of the standard apartment. We recall that  there is a 1-1 correspondence between apartments in $X$ and pairs of ends of $X$. Then $\eta$ belongs to the end determined by the standard apartment and thus $u'''\in Stab(\eta)$. 
The action of $u'''$ on the end determined by the standard apartment permutes the lines that meet at this end. Sufficiently far out,  $\eta$ is on the intersection of these lines. Since $\eta$ is fixed by $u'''$, we have $u'''\eta=\eta$. Left multiplying by $t''$, we obtain $$t''u'''\cdot\eta=t''\eta.$$
But $t''$ is a translation along the standard apartment and $t''u'''$ preserves translation length. Thus we have
$$t''u'''\cdot\xi=\xi$$ as above, and 
$$t''u'''\cdot\eta=t''\eta.$$
But this implies that 
 $t''=1$ and hence $t=t'$.
\end{proof}

\begin{proposition}\label{bounded} Let $\Gamma$ be the parabolic subgroup $P_1^-$ of  $G$.  For fixed $i$ and $j$, let  $\mathcal{U}=\mathcal{U}_i$ and $K=P_j$. Suppose
$\gamma=u_{\gamma}t_{\gamma}k_{\gamma} \in\Gamma$ with 
$u_{\gamma}\in {\mathcal U}$, $t_{\gamma}=(w_iw_{3-i})^{n_i(\gamma)}\in \mathcal T$ and $k_{\gamma}\in K$. Then 
$n_i(\gamma)$ is bounded below, that is, there
exists $n_0\in{\mathbb Z}$ such that $n_i(\gamma)\geq n_0$ for all 
$\gamma\in\Gamma$.\end{proposition}

\begin{proof} We assume on contrary that $n_i(\gamma)$ is not bounded below. Let $1\neq \widetilde \gamma\in \Gamma\cap{\mathcal U}$. Hence we may 
suppose  that for each $n\in\mathbb{Z}$ there exists
$\gamma_n\in \Gamma$ with $n_i(\gamma_{n})<n$. We have
$$\gamma_n^{-1}\widetilde \gamma\gamma_n= k_{\gamma_n}^{-1}t_{\gamma_n}^{-1}\widetilde \gamma t_{\gamma_n}k_{\gamma_n}\in\Gamma \bsl\{1\},$$
since $u_{\gamma}$ and $\widetilde \gamma$ both belong to $\Gamma\cap\mathcal{U}$ and so $u_{\gamma}$ centralizes $\widetilde \gamma$. We may take $\widetilde \gamma\in U_{\alpha}\subset {\mathcal 
U}$ for some $\alpha\in \Delta^{\re}_i$. Now 
$t_{\gamma_n}=(w_iw_{3-i})^{n_i(\gamma_n)}$ acts on the set of roots   $\Delta^{\re}_i$ by translation. The height 
of $t_{\gamma_n}^{-1}(\alpha)$ grows as $n\to-\infty$. If $\alpha$ is positive, then a sufficiently high element $t_{\gamma_n}^{-1}(\alpha)$ will flip $\alpha$ to a negative root and thus
$t_{\gamma_n}^{-1}U_{\alpha}t_{\gamma_n}\to \{1\}$ in $G$. If $\alpha$ is a negative root, the same conclusion holds.
But if $\gamma_n^{-1}\widetilde \gamma\gamma_n =1$, then $\widetilde \gamma=1$. Since we chose $\widetilde \gamma\neq 1$, we have a contradiction.
\end{proof}

\subsection{Iwasawa cells and vertices}
For $i=1,2$, we introduce the following sets:
\[
I_{i}:=\left(\bigcup_{x\in \mathcal{U}_{i}}\bigcup_{n\in \mathbb{Z}_{\ge 0}}x(w_{i}w_{3-i})^{n}P_{1}\right)\bigcup \left(\bigcup_{x\in \mathcal{U}_{i}}\bigcup_{n\in \mathbb{Z}_{\ge 0}}x(w_{i}w_{3-i})^{n}P_{2}\right)\label{eq5.1}
\]
The Iwasawa decomposition in Corollary \ref{Iwasawa} of $G$ that $I_1=G$ and $I_2=G$ (as sets). Thus we may write (with some redundancy):
\begin{eqnarray}\label{eq5.3}
G&=&I_{1}\cup I_{2}.
\end{eqnarray}

Since the vertices of the building are defined to be $G/P_1\sqcup G/P_2$, we may relabel the set of vertices as follows
\[
\left(\bigcup_{x\in \mathcal{U}_{i}}x\mathcal TP_{1}\right)/P_1\quad \bigsqcup \quad \left(\bigcup_{y\in \mathcal{U}_{i}}y\mathcal TP_{2}\right)/P_2.
\]
This choice of labeling of the vertices will be convenient for the rest of this section.

The next result gives a labeling of $X$ which allows us to relate the $\mathcal T$-component of the Iwasawa cells and the vertices of $X$. 
 For $w\in W$, we recall the subgroup $U_{0,w}$ of $U_{0}$ and its isomorphic image $U_{w}$ in the completion $U$ of $U_{0}$, as introduced in Section~\ref{Udec}. 
\begin{lemma}\label{conlem2}
If $w\in W$ has the reduced decomposition $w=w_{i_{1}}w_{i_{2}}\dots w_{i_{k}}$, then $U_{w}\subset {\mathcal U}_{i_{k}}$, for $i_{k}=1,2$. 
\begin{proof}
 The assertion of lemma is an implication of the explicit description of the set $S^{+}_{w}$, which makes it a subset of $\Delta^{\re}_{i_{k}}$, for $i_{k}=1,2$.
 \end{proof}
\end{lemma}

 \begin{proposition}\label{iwacor}
Every vertex on $X$ corresponds to a coset in the Iwasawa decomposition of $G$.
\end{proposition}
\noindent Proposition~\ref{iwacor} provides us a relationship between the Bruhat labels and Iwasawa labels of the vertices of the tree. The Iwasawa labels are given in Figure~\ref{fig3} ({\it cf}. Figure~\ref{fig1}).

\begin{figure}[h]
    \centering
    \includegraphics[width=0.8\textwidth]{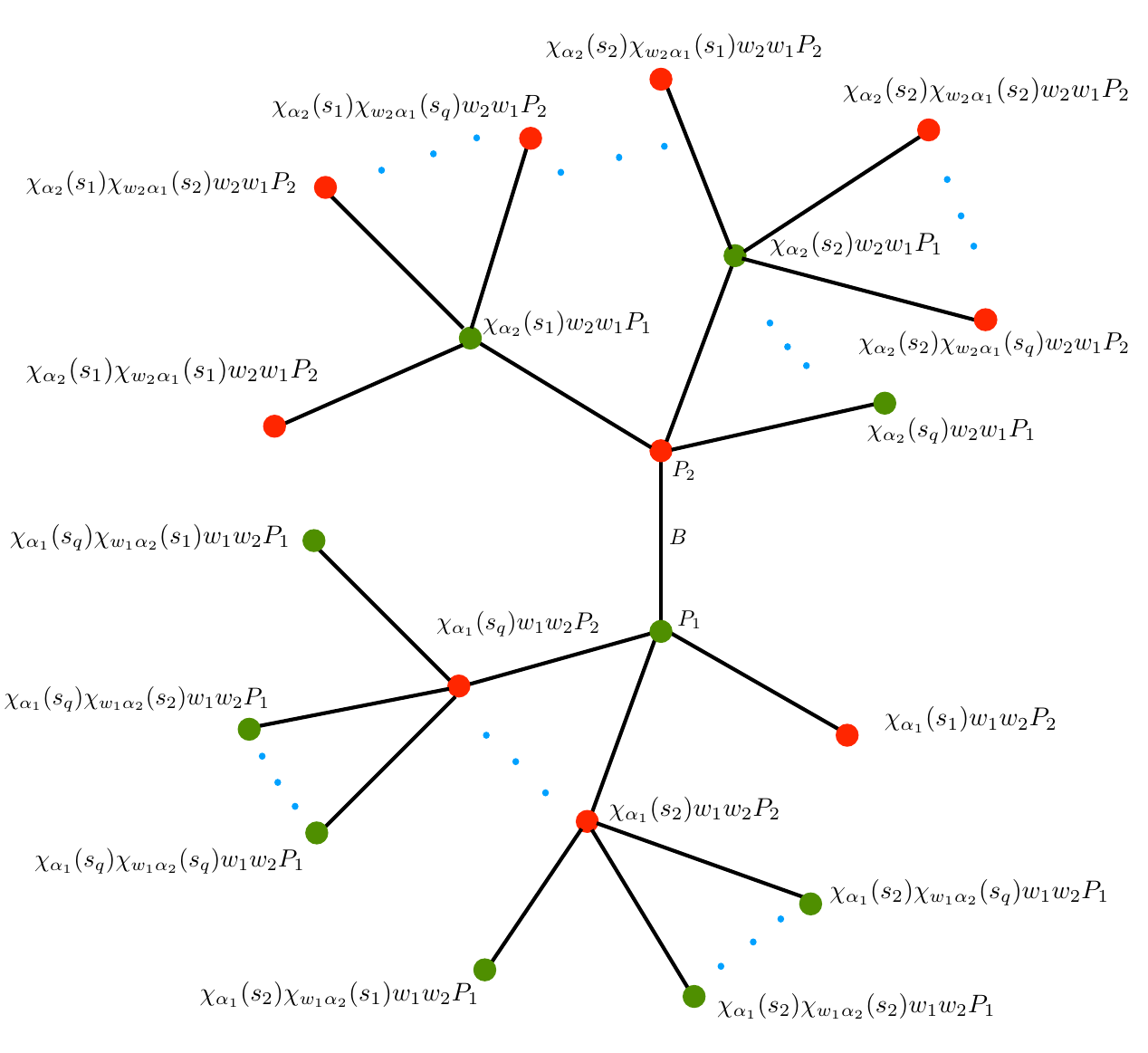}
    \caption{The tree labelled by Iwasawa Cells. The Iwasawa labels contain Weyl group elements of even lengths.}
    \label{fig3}
\end{figure}

We use the decomposition (\ref{eq5.3}) of $G$ to give a proof of Proposition~\ref{iwacor}.
\begin{proof} We show that each vertex on $X$ corresponds to a coset in the above decompositions. Let $v$ be a vertex on $X$ and let  $\sigma_{gP_{j}}$ the corresponding simplex of $v$, for $j=1,2$ and $g\in G$. Let $g\in BwB$ for some $w\in W$ with $\ell(w)=k$, where $\ell$ is the length function as defined in (\ref{lengw}) and the reduced decomposition $w=w_{i_{1}}w_{i_{2}}\dots w_{i_{k}}$. Corollary~\ref{sec3cor1} implies that
\begin{eqnarray}\label{lem12eq1}
BwB&=&UHwB=UwB=U_{w}U^{w}wB=U_{w}w(w^{-1}U^{w}w)B=U_{w}wB
\end{eqnarray}
In the second last equality, we used the fact that $w^{-1}U^{w}w\subset U$.
 Write $g=u_{w}wb$ for some $u_{w}\in U_{w}$ and $b\in B$.  Then $gP_{j}=u_{w}wbP_{j}= u_{w}wP_{j}$. Next, if $k$ is even, then $w=(w_{2}w_{1})^{n}$ for some $n\in \mathbb{Z}$ and hence $gP_{j}=u_{w}(w_{2}w_{1})^{n}P_{j}$. If $k$ is odd, then we write $w'=ww_{j}$ and $w'=(w_{2}w_{1})^{n}$ for some $n\in \mathbb{Z}\setminus\{0\}$ and $u_{w}w'P_{j}=u_{w}wP_{j}=gP_{j}.$
 
In both cases Lemma~\ref{conlem2} implies that $gP_{j}\in \mathcal{U}_{i}\mathcal TP_{j}$ for $i, j=1,2$, and this completes the proof.
\end{proof}

\begin{corollary} For all $g\in G$, there exists $w'\in W$ such that $gP_{j}=uw'P_{j}$ for some $u\in  \mathcal{U}_{i}$, where $\ell(w')$ is even.
\end{corollary}

We use the following description of elements of $\mathcal T$:
$$\mathcal T=\{(w_{i}w_{3-i})^{n}\mid n\in \mathbb{Z}_{\ge 0},\; i=1,2\}.$$
We denote an Iwasawa cell with $\mathcal T$-component of length $n_i$ by $\mathcal{U}_{i}(w_{i}w_{3-i})^{n_i}P_{j}$.  For the sake of brevity, we will use $n$ to denote $n_i$. 

Suppose $v$ is a vertex on $X$, which corresponds to an Iwasawa cell $\mathcal{U}_{i}(w_{i}w_{3-i})^{n_i}P_{j}$ for $i,j=1,2$. Further assume that the element of the set $\Omega_{v}^{1}$ corresponds to the Iwasawa cell with $\mathcal T$-component length $\ell^{1}(i)$ and the elements of $\Omega_{v}^{q}$ correspond to the Iwasawa cell with $\mathcal T$-component length $\ell^{q}_{r}(i)$, where $r=1,2,\dots, q$. 

\begin{proposition}\label{realres}
In the above notations
\begin{itemize}
\item [(1)] We have \[\ell^{1}(i)=
    \begin{dcases}
        n-1&\text{if}\; j=3-i\\
       n &\text{if}\; j=i.
       \end{dcases}
\]
\item [(2)] For each $r=1,2,\dots, q$, \[\ell^{q}_{r}(i)=
    \begin{dcases}
        n&\text{if}\; j=3-i\\
       n+1 &\text{if}\; j=i.
       \end{dcases}
\]
\end{itemize}
\end{proposition}

\begin{proof}
For (1), if $j=3-i$ then
$$\mathcal{U}_{i}(w_{i}w_{3-i})^{n_i}P_{j}=\mathcal{U}_{i}(w_{i}w_{3-i})^{n_i-1}w_{i}w_{3-i}P_{3-i}=\mathcal{U}_{i}(w_{i}w_{3-i})^{n_i-1}w_{i}P_{3-i}.$$
 By using the argument similar to the one used in the proof of Proposition~\ref{iwacor}, we can assume that $v$ corresponds to the conjugate $gP_{3-i}$ with $g=u_{w}wb$ for some $b\in B$, $w=(w_{i}w_{3-i})^{n_i-1}w_{i}$ and $u_{w}\in U_{w}$. Therefore, the element of $\Omega_{v}^{1}$ corresponds to the coset
\begin{eqnarray}
u_{w}wbP_{i}&=&u_{w}wP_{i}\nonumber\\
&=&u_{w}(w_{i}w_{3-i})^{n_i-1}w_{i}P_{i}\nonumber\\
&=&u_{w}(w_{i}w_{3-i})^{n_i-1}P_{i}.
\end{eqnarray}
Thus the element of $\Omega_{v}^{1}$ corresponds to the Iwasawa cell with $\mathcal T$-component length $n-1$.
The proof for $j=i$ can be obtained along the same lines.

In the proof (2), we assume $j=i$. The other case follows similarly. Let $\sigma_{g\chi_{\alpha_{i}}(s)w_{i}P_{3-i}}\in \Omega_{v}^{q}$ for some $s\in \mathrm{k}$.
\begin{eqnarray}
g\chi_{\alpha_{i}}(s)w_{i}P_{3-i}&=&u_{w}wb\chi_{\alpha_{i}}(s)w_{i}P_{3-i}\nonumber\\
&=&u_{w}w\chi_{\alpha_{i}}(s)w_{i} 
\left( w_{i}\chi_{\alpha_{i}}(-s)b\chi_{\alpha_{i}}(s)w_{i} \right) P_{3-i}\nonumber\\
&=&u_{w}\chi_{w\alpha_{i}}(s)ww_{i} \left( w_{i}\chi_{\alpha_{i}}(-s)b\chi_{\alpha_{i}}(s)w_{i} \right) P_{3-i}\nonumber\\
&=&u_{w}'ww_{i} \left( w_{i}\chi_{\alpha_{i}}(-s)b\chi_{\alpha_{i}}(s)w_{i} \right) P_{3-i}\label{prop13eq1}
\end{eqnarray}
for $u_{w}'=u_{w}\chi_{w\alpha_{i}}(s)\in U$. Using \cite[Lemma 6.3]{CG}, we write
\begin{eqnarray}\label{prop13eq2}
w_{i}\chi_{\alpha_{i}}(-s)b\chi_{\alpha_{i}}(s)w_{i}=u_{-\alpha_{i}}u^{\alpha_{i}}h,
\end{eqnarray}
for some $u_{-\alpha_{i}}\in U_{-\alpha_{i}}$, $u^{\alpha_{i}} \in U^{ w_{i}}=w_{i}Uw_{i}\cap U$ and $h\in H$. Using (\ref{prop13eq2}) in (\ref{prop13eq1}), we get
\begin{eqnarray}
g\chi_{\alpha_{i}}(s)w_{i}P_{3-i}&=&u_{w}'ww_{i}u_{-\alpha_{i}}u^{\alpha_{i}}hP_{3-i}\nonumber\\
&=&u_{w}'ww_{i}u_{-\alpha_{i}}P_{3-i}\nonumber\\
&=&u_{w}'u_{w\alpha_{i}}ww_{i}P_{3-i}\nonumber\\
&=&u_{w}''ww_{i}P_{3-i},\label{prop13eq3}
\end{eqnarray}
for some $u_{w\alpha_{i}}\in U_{w\alpha_{i}}$ and $u_{w}''=u_{w}'u_{w\alpha_{i}}\in U$. We set $w'=ww_{i}$ and use Corollary~\ref{sec3cor1} to write
$$u_{w}''=u_{w'}u^{w'}$$
for some $u_{w'}\in U_{w'}$ and $u^{w'}\in U^{w'}$. We use $w'=ww_{i}$ and the above expression for $u_{w}''$ in the right hand side of (\ref{prop13eq3}) to get
\begin{eqnarray}
g\chi_{\alpha_{i}}(s)w_{i}P_{3-i}&=&u_{w'}u^{w'}w'P_{3-i}\nonumber\\
&=&u_{w'}w'P_{3-i}\nonumber\\
&=&u_{w'}ww_{i}P_{3-i}.\label{prop13eq4}
\end{eqnarray}
Finally we get the assertion by putting $w=(w_{i}w_{3-i})^{n_i}$ in the right hand side of (\ref{prop13eq4}), which gives
$$u_{w'}(w_{i}w_{3-i})^{n_i}w_{i}P_{3-i}=u_{w'}(w_{i}w_{3-i})^{n_i}w_{i}w_{3-i}P_{3-i}=u_{w'}(w_{i}w_{3-i})^{n_i+1}P_{3-i}.$$
This completes the proof.
\end{proof}

The following local picture (Figure \ref{fig4}) gives a graphical interpretation of the above proposition.
\begin{figure}[h!]
\begin{minipage}[t]{0.48\textwidth}
\includegraphics[width=\linewidth,keepaspectratio=true]{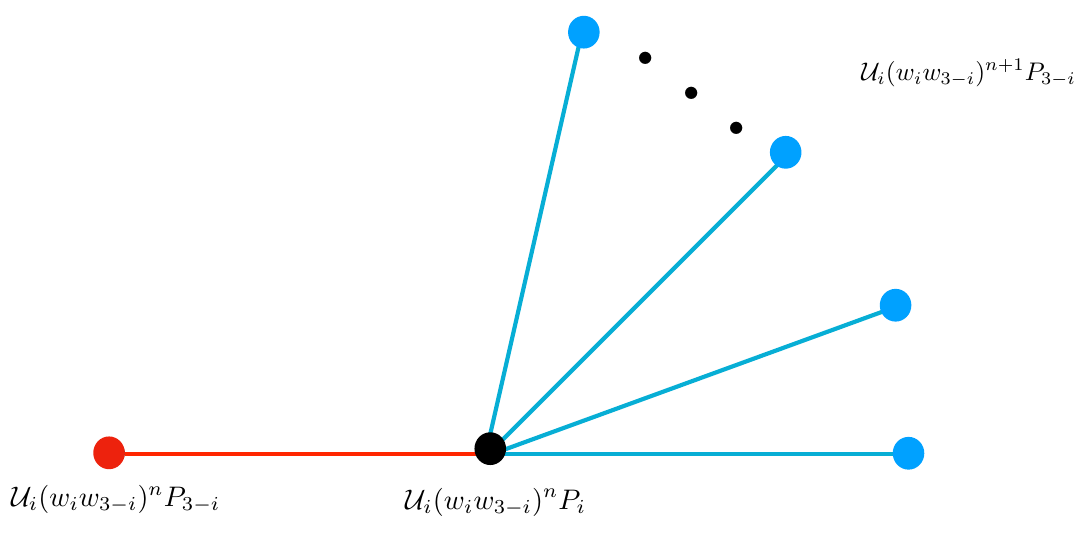}

\end{minipage}
\hspace*{\fill} 
\begin{minipage}[t]{0.48\textwidth}
\includegraphics[width=\linewidth,keepaspectratio=true]{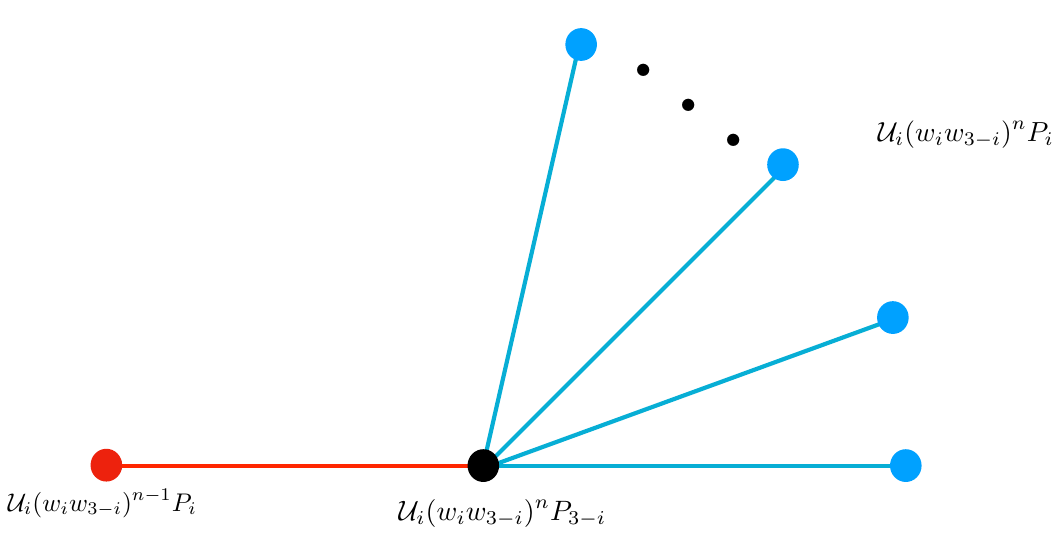}

\end{minipage}
 \caption{Local picture of adjacent vertices}
 \label{fig4}
\end{figure}

\section{Operators and functions on the Tits building}

In this section, we define a character on the vertices of the tree $X$ associated with $G$ that we will use to define Eisenstein series on quotients of $X$ by subgroups of the Kac--Moody group $G$.

\subsection{ Adjacency operator and eigenfunctions}
 For $x,y\in VX$, we let $d(x, y)$ denote  the number of edges in the shortest reduced path 
between the vertices $x$ and $y$. Let $\mathcal{F}(X)=\{f:VX\longrightarrow\mathbb{C}\}$. We define the adjacency operator $T$ on $\mathcal{F}(X)$ as follows
 \begin{eqnarray}
Tf(x)=\sum_{d(x,y)=1} f(y)=\sum_{e\in EX, o(e)=x} f(t(e))\label{adjop}
\end{eqnarray}
where $x\in VX$, and $o(e)$ and $t(e)$ denote the origin and terminus of $e$ respectively.

We recall that a non-uniform lattice is a discrete subgroup of finite covolume. 
The adjacency operator $T$ may also be defined on vertices of a quotient $\Gamma\bsl X$ for a non-uniform lattice $\Gamma\in\Aut(X)$. The operator $T$ is defined on a sequence of vertices $a_{n-1},a_n,a_{n+1}$  approaching the end of the 
ray  by
\[(T f)(a_n)= qf(a_{n-1}) + f(a_{n+1}).\]

We say that a function $f: VX\to \mathbb{C}$ is an \textit{eigenfunction} for the adjacency operator $T$ with eigenvalue $\lambda$ if 
$Tf(x)=\lambda f(x)$ for all but finitely many $x\in VX$.

Let $g\in G$, then by the Iwasawa decomposition $g\in \mathcal{U}_{i}(w_{i}w_{3-i})^{n_i(g)}P_{j}$ for some $i=1,2$ and $j=1,2$.
 We define the function $\Psi_{i, s}\colon G\longrightarrow \mathbb{C}^{\ast}$ as follows.
Let $q$ be the cardinality of finite field $\mathrm{k}$, $s\in \mathbb{C}$, then for $g\in \mathcal{U}_{i}(w_{i}w_{3-i})^{n_i(g)}P_{j}$ we set

 \begin{eqnarray}\label{defch}
 \Psi_{i,s}(g)=
    \begin{dcases}
        q^{-2n_i(g)s} &\text{if}\; j=i\\
       q^{-(2n_i(g)-1)s} &\text{if}\;j=3-i.
       \end{dcases}
\end{eqnarray}
The function $\Psi_{i,s}$ is left $\mathcal{U}_{i}$ invariant and is locally constant on each Iwasawa cell $\mathcal{U}_{i}(w_{i}w_{3-i})^{n_i}P_{j}$. By combining this with the decomposition (\ref{eq5.3}), one can deduce that $\Psi_{i,s}$ becomes a function on the set of vertices of tree $X$ for $i=1,2$.

\begin{theorem}
For $i=1,2$ and $s\in \mathbb{C}$, the function $\Psi_{i, s}$ is an eigenfunction for the adjacency operator $T$ with an eigenvalue $q^{1-s}+q^{s}$.
\end{theorem}

\begin{proof}
Let $v$ be a vertex on $X$, and it corresponds to a coset in the Iwasawa cell $\mathcal{U}_{i}(w_{i}w_{3-i})^{n_i}P_{j}$ for $i,j=1,2$. We use the  notation $n=n_i$ and discuss the following two cases:

{\bf Case 1: $j=3-i$}.

By definition $\Psi_{i,s}(v)=q^{-(2n-1)s}$. 
By Proposition~\ref{realres}, the set of a vertex $\Omega_{v}^{1}$ corresponds to the Iwasawa cell with $\mathcal T$-component of length $n-1$ and $q$ elements of $\Omega_{v}^{q}$ correspond to the Iwasawa cells with the fixed $\mathcal T$-component of length $n$. Therefore
\begin{eqnarray*}
T\Psi_{i,s}(v)&=&q\Psi_{i,s}(\mathcal{U}_{i}(w_{i}w_{3-i})^{n}P_{i})+\Psi_{i,s}(\mathcal{U}_{i}(w_{i}w_{3-i})^{n-1}P_{i})\\
&=&q \cdot q^{-2ns}+q^{-(2n-2)s}\\
&=&(q^{1-s}+q^{s})q^{-(2n-1)s}\\
&=&(q^{1-s}+q^{s})\Psi_{i,s}(v).\\
\end{eqnarray*}

{\bf Case 2: $j=i$}.

In this case $\Psi_{i,s}(v)=q^{-2ns}$. 
By Proposition~\ref{realres}, the set of a vertex $\Omega_{v}^{1}$ corresponds to the Iwasawa cell with $\mathcal T$-component of length $n$ and $q$ elements of $\Omega_{v}^{q}$ correspond to the Iwasawa cells with the fixed $\mathcal T$-component of length $n+1$. Therefore
\begin{eqnarray*}
T\Psi_{i,s}(v)&=&q\Psi_{s}(\mathcal{U}_{i}(w_{i}w_{3-i})^{n+1}P_{3-i})+\Psi_{i,s}(\mathcal{U}_{i}(w_{i}w_{3-i})^{n}P_{3-i})\\
&=&q\cdot q^{-(2n+1)s}+q^{-(2n-1)s}\\
&=&(q^{1-s}+q^{s})q^{-2ns} \\
& = & (q^{1-s} + q^s) \Psi_{i, s}(v).
\end{eqnarray*}

This completes the proof. 
\end{proof}

\section {Eisenstein series}
We fix  $\Gamma=P_1^-$. For $i=1,2$, the groups ${\mathcal U}_i $ and $\Gamma$ contain the subgroup generated by the root subgroups corresponding to the real roots $\{ -\alpha _{3-i},\ -w_{3-i}\alpha _{i},\ -w_{3-i}w_{i} \alpha
_{3-i},\dots \}$. Hence $\Gamma\cap{\mathcal B}_i\neq \varnothing$ for $i=1,2$.
Define the Eisenstein series $E_{i,s}$ on $\Gamma\bsl X= \Gamma\bsl G/P_1 \sqcup \Gamma\bsl G/P_2$ corresponding to $\mathcal{B}_i$,
\begin{equation}\label{eis}
E_{i,s}(gP_j)= \sum_{\gamma\in \Gamma\cap{\mathcal{B}_i \bsl \Gamma}}\quad 
\Psi_{i,s}(\gamma g P_j),\quad j=1,2.
\end{equation}
From now on we will write the Eisenstein series by
$E_{i,s}(g)=E_{i,s}(gP_j)$ by fixing the coset representative $g$ of $gP_j$.


The analysis of $E_{1,s}$ is similar to that of $E_{2,s}$. So, for the rest of this paper, we just consider $E_{1,s}$.
Moreover, for  notational simplicity, we set $$\mathcal{U}=\mathcal{U}_1,\quad \mathcal{B}=\mathcal{B}_1$$ and write 
\[
 \Psi_{s}=\Psi_{1,s} \quad \textrm{and}\quad  E_{s}=E_{1,s}
\]
 by dropping $1$ in the notation. 


\subsection{Iwasawa decomposition of Haar measure on $G$}
Let $\pi:{\mathcal U}\times \mathcal T\times K\longrightarrow G$ be the product map with $K=P_{j}$ for $j=1,2$. Then $\pi$ is surjective by the Iwasawa decomposition. 
Let $du, da, dk$ be Haar measures  
on ${\mathcal U},\mathcal T,K$ respectively.

 We normalize $du$ on $\mathcal{U}$
such that $du(\Gamma\cap \mathcal{U} \bsl \mathcal{U})=1$ since $\Gamma\cap \mathcal{U} \bsl \mathcal{U}$ is compact. We take $da$ to be the counting
measure on the discrete group $\mathcal T$. If $Y\subset \mathcal T$ then $da(Y)$ is the cardinality of $Y$. We normalize $dk$ to equal 1 on $K$ and its conjugates. 

 We let $\pi_*$ denote the induced map $(du, da, dk)\mapsto d\mu$. 

The following theorem indicates that we can decompose the Haar measure on $G$ as in the 
Iwasawa decomposition of $G$.

\begin{theorem}\label{thm:haar} (Iwasawa decomposition of Haar measure)

  (i)   ${\mathcal U}$ is normalized by $\mathcal T$ 
and $\mathrm{Ad}(a)du=a^{2\rho} du$,
where $a^{2\rho}:= q^{- 2n(a)}$ if $a= (w_1 w_2)^{n(a)}$.

  (ii) Let $d\nu=a^{-2\rho}du~da~dk$. Then $\pi_*\nu=\mu$, the Haar measure on $G$, and 
$\pi_*\nu$ is bi-invariant.

 (iii) $a^{-2\rho}du~da$ is a left $\mathcal T$--invariant 
Haar measure on ${\mathcal U}\mathcal T$.

 \end{theorem}

\begin{proof} Recall that $w_1w_{2}$ acts on $\Delta^{\re}_1$ by translation. It is enough to consider the subgroup of $\mathcal{U}$ generated by positive root groups. Let $\mathcal{U}^+\subset \mathcal{U}$ be the completion of the
group generated by $U_\alpha$ with $\alpha\in\Delta^{\re}_1$, $\alpha>0$, then $\mathcal{U}^+$ is a subgroup of $U$ and $$\textrm{Ad}((w_1w_{2})^n)\mathcal{U}^+\cap \mathcal{U}^+\le U^{w}\cap \mathcal{U}^+$$ where $U^{w}=U\cap w^{-1}Uw$. Then,  by Corollary~\ref{sec3cor1} 
$$\frac{\mathcal{U}^+}{ \textrm{Ad}((w_1w_{2})^n)\mathcal{U}^+\cap \mathcal{U}^+ }\cong U_{(w_1w_{2})^n}$$

which has cardinality $q^{\ell((w_1w_{2}))^n}=q^{2n}$.

 For (ii) we  note that  $K$ acts on $G$ by right translation.  Also
 $\pi_*\nu$ and $\mu$ agree and coincide with the Haar measure on 
$K\subset G$ which is an open compact
subgroup acting transitively on $G$. Hence $\pi_*\nu$ must be the Haar measure on all of $G$. 
We note also that $\pi_*\nu$ is right $K$-invariant and  left   ${\mathcal U}\mathcal T$-invariant.

For (iii) we note that ${\mathcal U}\mathcal T$ acts on $G$ by left translation. Since  $\pi_*\nu$ is left   ${\mathcal U}\mathcal T$-invariant, 
 it follows that $a^{-2\rho}du~da$ is  left $\mathcal T$--invariant 
on ${\mathcal U}\mathcal T$ (see also \cite[1.1.1]{H}), which proves (iii).
\end{proof}

\subsection {Convergence of Eisenstein series}

 The following theorem shows that the series $E_{s}(g)$ converges  uniformly and absolutely on compact
subsets.   We assume that $s\in\mathbb{C}$ is `sufficiently regular', that is, $Re(s)>1$.

\begin{theorem} \label{thm:conv}
The Eisenstein series  \[
E_{s}(g)= \sum_{\gamma\in \Gamma\cap \mathcal{B} \bsl\Gamma}\Psi_{s}(\gamma g)
\]
converges absolutely, provided $
{Re}(s) > 1$.
\end{theorem}

\begin{proof} Without loss of generality we may assume that $s$ is real. Using the Iwasawa decomposition $ G=\mathcal{U}\mathcal T K$, we may choose a neighborhood $\Omega $ of the identity $e$ in $ G$ such that $\Omega 
\subset K$ and for $\gamma=u_{\gamma}t_{\gamma}k_{\gamma} \in\Gamma$ and $g=u_gt_gk_g \in G$ we have $\gamma g\Omega \neq \gamma'g\Omega$
whenever $\gamma \neq \gamma'$, $\gamma, \gamma'\in\Gamma$.

Since $a_{\gamma g\omega }=a_{\gamma g }$ for $\omega \in \Omega 
$ (even for $\omega \in K)$, we may average $\Psi_{s}(\gamma g)$ over the coset $\gamma g\Omega  $:

$$\Psi_{s}(\gamma g)=\int_{\gamma g\Omega  } \Psi_{s}(x) dx,$$
where we normalize the measure of $\Omega$ to be 1.
Hence $$E_{s}(g)= \sum_{\gamma\in \Gamma\cap\mathcal{B}\bsl\Gamma} 
\Psi_{s}( \gamma g) = \sum_{\gamma\in \Gamma\cap \mathcal{B}\bsl\Gamma} 
\int_{\gamma g\Omega  } \Psi_{s}(x) dx = \int_{\Gamma\cap \mathcal{B}\bsl\Gamma g\Omega  } \Psi_{s}(x) dx.$$
Applying a variant of Proposition \ref{bounded}, we may conclude that
  $\Gamma g \Omega   \subset   {\mathcal U}\mathcal T(n_0)K$ for some $n_0\in\mathbb{Z}$, 
where $\mathcal T(n_0)=\{(w_1 w_{2})^n\in \mathcal T \mid n\geq n_0\}$. When $s>0$ from (\ref{defch}) it is easy to deduce the estimate
\begin{eqnarray}\label{chb}
 \Psi_{s}(a) \leq q^{s-2n(a)s},\quad a\in \mathcal T.   
\end{eqnarray}
Then we have
\begin{eqnarray}
E_{s}(g)&\leq & \int_{\Gamma \cap \mathcal{B}\bsl \mathcal{U} \mathcal T(n_{0})K }\Psi
_{s}(x)dx\nonumber\\
&=& C\int_{\mathcal T(n_{0})}\Psi_{s}(a) a^{-2\rho }da\nonumber
\\
  &\leq& Cq^s\int_{\mathcal T(n_{0})} q^{-2 n(a)(s-1)}da\nonumber\\
&=&Cq^s\sum_{n\geq n_0} q^{-2 n (s-1)}\label{EB1}
 \end{eqnarray}
where $C=\Vol(\Gamma\cap\mathcal{B} \bsl \mathcal{B})\Vol(\mathcal{B_I} \bsl K)$, noting that
$\Gamma\cap\mathcal{B} \bsl \mathcal{B}$ and $K$ are compact. 
Hence $E_{s}(g)$ converges absolutely provided
${Re}(s)>1$.
\end{proof}
For $\ell\in\mathbb{R}_{\geq 0}$, let 

\begin{eqnarray}
L^2(\Gamma\bsl X,\ell) := \{f:\Gamma\bsl X\longrightarrow \mathbb{C} 
\mid \int_{\Gamma\bsl X} |f(g)|^2(\Psi_{s}(g))^{-2\ell} dg<\infty\}
\end{eqnarray}

denote the space of $\ell$-moderate growth functions on $\Gamma\bsl X$.\\

\begin{proposition}
We have $E_{s}\in L^2(\Gamma\bsl X, \ell)$. That is
$$\int_{\Gamma\bsl X} |E_{s}(g)|^2(\Psi_{s}(g))^{-2\ell} dg<\infty.$$

\begin{proof}   Without loss of generality, we may assume that $s$ is real. Using the definition of $E_{s}(g)$ and the triangle inequality, we write
\begin{eqnarray}
\int_{\Gamma\bsl X} |E_{s}(g)|^2(\Psi_{s}(g))^{-2\ell} dg&=&\int_{\Gamma\bsl X}\left( |\sum_{\gamma\in \Gamma\cap\mathcal{B}\bsl\Gamma} 
\Psi_{s}( \gamma g) |^2\right)(\Psi_{s}(g))^{-2\ell} dg\nonumber\\
&\le&\int_{\Gamma\bsl X} \left(\sum_{\gamma\in \Gamma\cap\mathcal{B}\bsl\Gamma} 
|\Psi_{s}( \gamma g) |^2\right)(\Psi_{s}(g))^{-2\ell} dg.
\end{eqnarray} 
For $g \in \Gamma \backslash X$ and $\gamma \in \Gamma \cap \mathcal{B} \backslash \Gamma$, both $g$ and $\gamma g$ correspond to the same Iwasawa cells. By \cite[Theorem 3.10]{AC}, the set of vertices $\{P_{1}, P_{2}, w_{2}P_{1}, w_{2}w_{1}P_{2}, \dots\}$ forms the fundamental domain for the action of $\Gamma$ on $X$. By Proposition~\ref{iwacor}, these vertices correspond to the Iwasawa cells ${\mathcal U}\mathcal T'K$ for $\mathcal T' = \{(w_{1}w_{2})^n \mid n \le 0\}$. Combining these two facts with the inequality (\ref{chb}) and an argument similar to the one used in the proof of Theorem~\ref{thm:conv}, we can deduce that for some constant $C$.

\begin{eqnarray*}
\int_{\Gamma\bsl X} |E_{s}(g)|^2(\Psi_{s}(g))^{-2\ell} dg
&\le&C\int_{\mathcal T'} 
|q^{s-2n(a)s} |^2(q^{s-2n(a)s})^{-2\ell} a^{-2\rho }da.\\
&=&Cq^s\int_{\mathcal T'} 
|q^{s-2n(a)s} |^2(q^{s-2n(a)s})^{-2\ell} da.\\
&=&Cq^s\sum_{n\le 0} q^{2s(2-2n)-2s\ell(1-2n)}\\
&=&Cq^s\sum_{n\le 0} q^{2s(2-2n)-2s\ell(2-2n-1)}\\
&=&Cq^s\sum_{n\le 0} q^{2s(2-2n)-2s\ell(2-2n)}q^{2s\ell}\\
&=&Cq^{s(2\ell+1)}\sum_{n\le 0} q^{2s(1-\ell)(2-2n)}.
\end{eqnarray*} 
So, for $\ell$ sufficiently large, we have $\int_{\Gamma\bsl X} |E_{s}(g)|^2(\Psi_{s}(g))^{-2\ell} dg<\infty$ and hence $E_{s}\in L^2(\Gamma\bsl X, \ell)$.
\end{proof}

\end{proposition}

It follows from the definitions and Theorem~\ref{thm:conv} that for $\gamma\in\Gamma$, we have
$E_{s}(g)=E_{s}(\gamma g)$.
\subsection {The constant term}\label{conster}
Suppose now that $f\in L^2(\Gamma\bsl X, \ell)$, then $f$ is a $\Gamma$-left and $K$-right invariant function on $G$, where $K=P_{j}$, $j=1,2$ as before. We define the {\it constant term} of $f$ along $\mathcal{U}$, to be
\begin{eqnarray*}
C_{\mathcal{U}}f(g) :=\int_{\Gamma\cap \mathcal{U}\bsl \mathcal{U}}f(ug)du,
\end{eqnarray*}
where $du$ is the induced Haar measure. In this subsection we compute the constant term $C_{\mathcal{U}}^{0}(E_{s})$. 
 

 Since $\mathcal{U}\subset \mathcal{B}$ and $\mathcal{B}$ is the stabilizer of the end of the positive half $\mathcal{A}_0^+$ of the standard apartment, we may view the constant terms as being computed `along $\mathcal{A}_0^+$'. 

\begin{proposition}\label{ESCT}  The constant term of $E_{s}$ is given as follows,
$$C_{\mathcal{U}}E_{s}(g) = \Psi_{s}(g) +c(s)\Psi_{1-s}(g),$$
where  $c(s)$ is holomorphic for $Re(s)>1$ and is independent of $g$. 
\end{proposition}

\begin{proof}
The absolute convergence of the integral defining $C_{\mathcal{U}}E_{s}$ for $Re(s)>1$ follows easily
from Theorem \ref{thm:conv} and the fact that $\Gamma\cap \mathcal{U}\bsl \mathcal{U}$ is compact. This also implies that $C_{\mathcal{U}}E_{s}$ is holomorphic in $Re(s)>1$. 

Since $C_{\mathcal{U}}E_{s}(g)$ is $(\mathcal{U}, K)$-bi-invariant, without loss of generality we may assume that $g =a \in \mathcal T$. By the Bruhat decomposition $G=\mathcal{B}\sqcup \mathcal{B}w_{1}\mathcal{B}$, we have

\begin{eqnarray}\label{CTD1}
C_{\mathcal{U}}E_{s}(a)&=&\int_{\Gamma\cap\mathcal{U}\bsl \mathcal{U}} \Psi_{s}(u a) du+\int_{\Gamma\cap\mathcal{U}\bsl \mathcal{U}} \sum_{\gamma\in (\Gamma\cap \mathcal{B}) \bsl (\Gamma \cap \mathcal{B}w_{1}\mathcal{B})} \Psi_{s}(\gamma ua) du \nonumber \\
& =&  \Psi_s(a) + \int_{\Gamma\cap\mathcal{U}\bsl \mathcal{U}} \sum_{\gamma\in (\Gamma\cap \mathcal{B}) \bsl (\Gamma \cap \mathcal{B}w_{1}\mathcal{B})} \Psi_{s}(\gamma ua) du.
\end{eqnarray}
To evaluate the second term in \eqref{CTD1}, we  exchange the integration and summation using absolute convergence, and  reassemble to obtain 
\begin{equation} \label{2ndterm}
\sum_{\gamma\in (\Gamma\cap \mathcal{B})\bsl (\Gamma\cap \mathcal{B}w_1\mathcal{B})/ (\Gamma\cap \mathcal{U})}\int_{\mathcal{U}} \Psi_s(\gamma ua)du.
\end{equation}
By Theorem \ref{thm:haar} and a change of variable $u\mapsto {\rm Ad}(a)(u)$, \eqref{2ndterm} equals 
\[
 \sum_{\gamma\in (\Gamma\cap \mathcal{B})\bsl (\Gamma\cap \mathcal{B}w_1\mathcal{B})/ (\Gamma\cap \mathcal{U})} a^{2\rho}\int_{\mathcal{U}}\Psi_s(\gamma a u)du.
\]
 We can view each double coset $\gamma\in (\Gamma\cap \mathcal{B})\bsl (\Gamma\cap \mathcal{B}w_1\mathcal{B})/ (\Gamma\cap \mathcal{U})$ as an arbitrary representative in $\Gamma\cap \mathcal{B}w_1\mathcal{B}$ and decompose it as
\begin{eqnarray}\label{ctaac1}
\gamma= u_{\gamma}  w_1 a'_{\gamma} ,
\end{eqnarray}

where $u_\gamma\in \mathcal{U}$ and $a'_\gamma\in \mathcal T$. By using part (b) of Lemma~\ref{Iwalemm1}, we write(\ref{ctaac1}) 
as 
\begin{eqnarray}\label{ctaac2}
\gamma= u_{\gamma} a_{\gamma}  w_1 ,
\end{eqnarray}
for some $a_\gamma\in \mathcal T$.
It is easy to check that
\[
\int_{\mathcal{U}}\Psi_s(\gamma a u)du = \int_{\mathcal{U}}\Psi_s(a_\gamma  w_1 au)du.
\]

Write $a=(w_1w_2)^n$, then $a^{2\rho} = q^{-2n}$. By part (a) of Lemma~\ref{Iwalemm1},
$$w_1 a=w_{1}(w_1w_2)^{n}=(w_1w_2)^{-n}w_{1},$$
which gives that
\[
\Psi_s(a_\gamma w_1 au) = \Psi_s(a_\gamma (w_1w_2)^{-n}w_1 u) = q^{2ns} \Psi_s(a_\gamma w_1 u).
\]
It follows that 
\begin{align*}
a^{2\rho}\int_{\mathcal{U}}\Psi_s(\gamma au)du &= q^{-2n(1-s)} \int_{\mathcal{U}}\Psi_s(a_\gamma w_1 u)du \\
& = q^{-2n(1-s)}\int_{\mathcal{U}}\Psi_s(\gamma u)du \\
& = \Psi_{1-s}(a) \cdot \delta(s) \int_{\mathcal{U}}\Psi_s(\gamma u)du,
\end{align*}
where
\[
\delta(s) :=\begin{cases}
    1, & \textrm{if }K=P_1, \\
    q^{s-1}, & \textrm{if }K=P_2.
\end{cases}
\]
In summary, it follows that \eqref{2ndterm} equals $c(s)\Psi_{1-s}(a)$, where
\[
c(s) := \delta(s) \sum_{\gamma\in (\Gamma\cap \mathcal{B})\bsl (\Gamma\cap \mathcal{B}w_1\mathcal{B})/ (\Gamma\cap \mathcal{U})} \int_{\mathcal{U}} \Psi_s(\gamma u)du.
\]
This finishes the proof of the proposition. 
\end{proof}



 We ask if it is possible to determine if one can deduce the convergence of $E_{s}(g)$ in Theorem~\ref{thm:conv} using the computation of the constant term?

\section{Integral operators on the Tits building}\label{operators}
\subsection{Spaces of functions on $G$ and $X$}

Let $\mathcal{C}(K\backslash G/K)$ denote the space of continuous $\mathbb{C}$-valued  bi-$K$-invariant functions on $G$.   That is, for $f\in \mathcal{C}(K\backslash G/K)$ we have $f(kgk')=f(g)$ for all $g\in G$, $k,k'\in K$. For $f\in\mathcal{C}(K\bsl G/K)$ and $g\in G$ we have 
\[
 |f(g)| = \int_{gK} |f(g')| \ dg',
 \]
where $dg'$ is the  Haar measure on $G$ such that $K$ has total measure 1.

Let $X=X_{q+1}$ denote the Tits building of $G$. 
A function $f$ on $VX$ is a pair of functions $(f_1,f_2)$ on $G$ such that $f_i$ is $P_i$-right invariant. 
For $p\ge 1$, set 
\begin{eqnarray*}
\Vert f \Vert_p&=&( \int_G   |f(g)|^p dg)^{1/p},\\
\Vert f \Vert'_p&=&( \sum_{x\in VX}|f(x)|^p )^{1/p}.
\end{eqnarray*}
Let
$$
L^p(G)=\{f:G\longrightarrow\mathbb{C}\mid\Vert f \Vert_p<\infty\},
$$
and 
$$
L^p(X)=\{f:VX\longrightarrow\mathbb{C}\mid \Vert f \Vert'_p<\infty\},
$$
 be the spaces of $\mathbb{C}$-valued  integrable functions on $G$ and $VX$ respectively. 
We give $
L^p(X)$ the topology of uniform convergence on compact sets.

 We have $$(\Gamma\backslash G/P_1)\sqcup (\Gamma\backslash G/P_2)=\Gamma\backslash X.$$ Thus $L^p(X)$ consists of pairs of functions $(f_1,f_2)$ such that $f_i\in L^p(G)$ is $P_i$-right invariant.

Let   $g\in \mathcal{U} a_g K$, with $a_g\in\mathcal T$ and $K=P_{j}$ for $j=1,2$. 
 We define 
 $$L^2(\Gamma\backslash G/K,\ell)=\{f:\Gamma\backslash G/K\longrightarrow\mathbb{C}\mid 
  \int_{\Gamma\backslash G/K}   |f(g)|^2 a_g^{-2\ell}dg<\infty\}$$

For $s\in\C$, recall that $a_g^{s}=\Psi_{s}(g) $. Thus by a slight abuse, we identify 
 $L^2(\Gamma\backslash G/K,\ell)$ with $L^2(\Gamma\backslash X,\ell)$.
 
Recall that $\Gamma$ has finite covolume $\mu(\Gamma \backslash G)$ relative to a Haar measure $\mu$ on $G$. We let

$$
L^p(\Gamma\backslash G)=\{f:\Gamma\backslash G\longrightarrow\mathbb{C}\mid\Vert f \Vert_p<\infty\}.
$$

\subsection{Integral operators}\label{intop1}
In this subsection, we define the integral operators we will use to prove meromorphic continuation of Eisenstein series. 
From now on, we consider the following quotient of $G$:
$$\overline{G}:=G/\left(\cap_{g\in G} gBg^{-1}\right)$$
which acts faithfully on $X$. The defining homomorphism
$$\rho:G/\left(\cap_{g\in G} gBg^{-1}\right)\hookrightarrow \Aut(X)$$
is continuous and the image is closed (\cite{CG}). Thus we may identify $\overline{G}$ with a subgroup of $\Aut(X)$. By an abuse of notation, we drop the `$-$' and identify $\overline{G}$ with $G$.

Let $\mathcal{X}:=  G/K$. Let $g,g'\in \mathcal{X}$ be coset representatives and write $g'^{-1}$ to denote the inverse of $g'$. Let $\mathcal{K}\in L^1(G\times G)$ and assume that $\mathcal{K}$ is radial, so that $\mathcal{K}(g,g')=F(g'^{-1}g)$ for some $F\in L^1(G)\cap \mathcal{C}(K\bsl G/K)$. Then $F(g'^{-1}g)$ makes sense for all $K$-bi--invariant functions $F$ since it is independent of the choice of coset representatives for $g$ and $g'$. We will assume that $\mathcal{K}$ is $K$-bi-invariant.

Then  for $g,h\in G$, $k\in K$ we have
 $$\mathcal{K}(gk,hk)=\mathcal{K}(g,h),$$
 $$\mathcal{K}(kg,kh)=\mathcal{K}(g,h).$$

As in \cite{CMS}, for a function $f$ on $\mathcal{X}$ we set
$$\mathcal{L}_{\mathcal{K}}(f)(g)=\int_{\mathcal{X}}  f(h)\mathcal{K}(h^{-1},g) dh,$$
whenever the integral makes sense, with $dh$ the induced left-invariant measure on $\mathcal{X}$. On the set of vertices $V\mathcal{X}$=$\{\sigma_0,\sigma_1, \sigma_2, \sigma_3, \sigma_4, \dots\}$  we have 
 $$\mathcal K(\sigma_0,\sigma_n)=F(n)$$ for some radial function $F$, where $n=d( \sigma_0,\sigma_n)$. By \cite{CMS}, since $\mathcal{K}$ is radial, for $x,y\in\mathcal{X}$ we have
$$\mathcal{L}_{\mathcal{K}}(f)(x)=\sum_{n=0}^\infty F(n)\sum_{d(x,y)=n} f(y).$$
Since $\mathcal{K}$ is radial, by \cite{CMS}, for $x,y\in\mathcal{X}$ we have
$$\mathcal{L}_{\mathcal{K}}(f)(x)=\sum_{n=0}^\infty F(n)\sum_{d(x,y)=n} f(y).$$

Since $\mathcal{K}(g,h)\in L^2(\mathcal{X}\times\mathcal{X})$, $\mathcal{L}_{\mathcal{K}}$ is a bounded operator on $L^2(\mathcal{X})$. It is well known that  $\mathcal{L}_{\mathcal{K}}$ is a Hilbert--Schmidt compact operator (\cite{Bu}). 
We assume that $\mathcal{L}_{\mathcal{K}}$ is self-adjoint, which holds if and only if $$\mathcal{K}(g,h)=\overline{\mathcal{K}(h,g)}.$$
We further assume that $\mathcal{K}$ is chosen such that $\mathcal{L}_{\mathcal{K}}$ is $K$-conjugation invariant on $G$.

\subsection{Rapid decay and compact operators on $L^2(\Gamma\backslash X,\ell)$}\label{RD}
Choose a basepoint $x_0\in VX$ and let $x\in VX$. Define $|x|=d(x_0,x)$. We say that $f:VX\longrightarrow \mathbb{C}$ is {\it rapidly decreasing} (as in \cite{Cow}) if for all $k\in\mathbb{N}$ there exists $C_k>0$ such that 
$$|f(x)|\leq C_k \dfrac{q^{-\frac{|x|}{2}}}{(1+|x|)^k}.$$

Following (\cite{AC2017},  Lemma 2.5) an operator $A:L^2(\Gamma\backslash X,\ell)\to L^2(\Gamma\backslash X,\ell)$ is a {\it compact operator} if and only if given a bounded sequence $(f_n)$ in $X$ such that $f_n \to 0$ pointwise, the sequence $(Af_n)$ converges to zero.

By (\cite{AC2017}, Theorem 3.2), $A:L^2(\Gamma\backslash X,\ell)\to L^2(\Gamma\backslash X,\ell)$ is a compact operator if and only if for all\linebreak $f\in L^2(\Gamma\backslash X,\ell)$, we have $\lim_{|x|\to\infty} f(x)=0$. 

Thus, Theorem 3.2 and Lemma 2.5 of \cite{AC2017}  immediately give the following.

\begin{proposition}\label{compactop}
    Let $A:L^2(\Gamma\backslash X,\ell)\to L^2(\Gamma\backslash X,\ell)$. If $Af$ is a rapidly decreasing
function, then $A$ is a compact operator. \end{proposition}.

\section{Truncation}\label{trunc}

In general for a function $\phi$ on $\calU\bsl G/K$ and $\Gamma$ a discrete subgroup of $G$ we define the {\it special Eisenstein series}
\begin{eqnarray}
E(\phi)(g) :=\sum_{\gamma\in \Gamma\cap\calU\bsl\Gamma}\phi(\gamma g),\label{genes}
\end{eqnarray}
assuming the absolute convergence. Then $E(\phi)$ is a function on $\Gamma\bsl X$. In particular, for $\Gamma=P_{1}^{-}$ we have that  $E_{s}=E(\Psi_{s})$.
 
Assume from now on that $\Gamma =P_1^-$. We recall that   $\mathcal{A}_0^+$ is a fundamental domain for $\Gamma=P_1^-$ on $X$ and that the vertices of $\mathcal{A}_0^+$ are
\begin{equation} \label{cusp}
V\mathcal{A}_0^+=\{P_1, P_{2}, w_{2} P_1,  w_{2}w_1 P_{2},  w_{2}w_1 w_{2} P_1, \dots\}.
\end{equation}
 By a slight abuse of terminology, we say that  $\mathcal{A}_0^+$ is the `cusp' of $\Gamma\bsl X$. 


\begin{theorem}\label{equals}  Let $f\in L^2(\Gamma \bsl X, \ell)$. Then  $f=C_{\mathcal U}f$ on the cusp of $\Gamma\bsl X$. \end{theorem}

Theorem~\ref{equals} follows immediately from Lemma~\ref{Cterm} below. 
\begin{lemma}\label{Cterm}  On the cusp of $\Gamma\bsl X$ we have
\begin{equation}\label{11.1}
f((w_2w_1)^n P_2)= C_{\mathcal{U}}f((w_2w_1)^n P_2),
\end{equation}
\begin{equation}\label{11.2}
f((w_2w_1)^n w_2P_1)= C_{\mathcal{U}}f((w_2w_1)^n w_2P_1)
  \end{equation}
  for any $n\geq 0$.
  \end{lemma}

\begin{proof} To prove (\ref{11.1}), we will  prove that
  \begin{equation}\label{11.3}
  f((w_2w_1)^nP_2)=f(u(w_2w_1)^nP_2)
  \end{equation}
    for any $u\in \Gamma\cap \mathcal{U}\backslash \mathcal{U}$.  Note that in this case we take $K=P_2$ and the left $\mathcal{U}$-invariance implies that 
  the Eisenstein series is equal to its constant term along $\mathcal{U}$.


It is easy to see that
$\Gamma\cap \mathcal{U}\backslash \mathcal{U}$ is isomorphic to the subgroup $\mathcal{U}'$ of  $\mathcal{U}$ which corresponds to the set of roots 
\[
\Delta^{re '}_1:=\Delta^{\re}_{1,+}\setminus \{\alpha_1\}=\{w_1\alpha_2, w_1w_2\alpha_1,\ldots\},
\] 
where $\Delta^{\re}_{1,+}:=\{\alpha_{1}, w_{1}\alpha_{2}, w_{1}w_{2}\alpha_{1},\dots\}$. Since $(w_1w_2)^n$ preserves $\Delta^{re '}_1$ (it shifts the roots in $\Delta^{\re}_{1,+}$ `upward'),  for $u\in \mathcal{U}'$ one has
\[
\mathrm{Ad}(w_1w_2)^n(u)\in \mathcal{U}'\subset P_2
\]
hence 
\[
u(w_2w_1)^nP_2 = (w_2w_1)^n P_2.
\]
 From this  (\ref{11.3}) is clear. The proof of (\ref{11.2}) is similar. In fact we only need to prove that
\begin{equation}\label{11.4}
f((w_2w_1)^nw_2P_1)=f(u(w_2w_1)^nw_2P_1)
\end{equation}
for $u\in \mathcal{U}'$. But we have $w_2(w_1w_2)^n(\Delta^{{\re}'}_1)\subset \Delta^{\re}_{2,+}$, hence Ad$(w_2(w_1w_2)^n)(\mathcal{U}')\subset B\subset P_1$.  This implies that
$u(w_2w_1)^nw_2P_1=(w_2w_1)^nw_2P_1$ and (\ref{11.4}) follows immediately.
\end{proof}

For a function $\phi$ on $\mathcal{U}\bsl G/K$, define its tail to be 
\[
{\rm tail}(\phi)(x) :=
\begin{cases} 
\phi(x), & \textrm{if }x\in V\mathcal{A}_0^+, \\
0, & \textrm{otherwise},
\end{cases}
\]
where $\mathcal{A}_0^+$ is the cusp with vertices given by \eqref{cusp}.
We now define the {\it truncation operator} on  $L^2(\Gamma\bsl X, \ell)$ by
\[
{\rm trunc}(f) := f-E({\rm tail}(C_{\mathcal{U}}f)),\quad f\in L^2(\Gamma \bsl X, \ell).
\]

  Recall that $V\mathcal{A}_0^+$ is the set of vertices of $\Gamma\bsl X$.
\begin{lemma}\label{trunc}
$({\rm trunc}(f)(x)$ equals zero for all $f\in L^2(\Gamma \bsl X, \ell)$ and $x\in V\mathcal{A}_0^+$.
\begin{proof}
By the definition, ${\rm tail}(C_{\mathcal{U}}f)=C_{\mathcal{U}}f$ on $V\mathcal{A}_0^+$.
So, on $V\mathcal{A}_0^+$
$${\rm trunc}(f) = f-E({\rm tail}(C_{\mathcal{U}}f))=f-E(C_{\mathcal{U}}f).$$
By Theorem~\ref{equals}, $f=C_{\mathcal{U}}f$ on $V\mathcal{A}_0^+$ and hence ${\rm trunc}(f) := f-E(f)$.

Therefore, for all $x\in V\mathcal{A}_0^+$
$$
{\rm trunc}(f)(x)= f(x)-E(f)(x)
= f(x)-\sum_{\gamma\in \Gamma\cap\mathcal U\bsl\Gamma}f(\gamma x).$$
  We note that both $x$ and $\gamma x$ correspond to the same Iwasawa cells, for all $x\in V\mathcal{A}_0^+$ and all $\gamma\in \Gamma\cap\mathcal{U}\bsl\Gamma$. Consequently ${\rm trunc}(f)$ is  zero on the cusp.

\end{proof}
\end{lemma}

Lemma~\ref{trunc} shows that for $f\in L^2(\Gamma\bsl X, \ell)$, ${\rm trunc}(f)$ is trivially a rapidly decreasing function on $\Gamma\backslash X$. (This is an analog of rapid decay on a Siegel set in the classical case.) Hence ${\rm trunc}$ is (also trivially) a compact operator on  $L^2(\Gamma\backslash X,\ell)$.  For $\mathcal{K}$ and $\mathcal{L}_{\mathcal{K}}$ as defined in Subsection~\ref{intop1}, we have the following.

\begin{theorem}\label{thm:LT} The operator $\mathcal{L}_{\mathcal{K}}\circ {\rm trunc}(f)$  is  a compact operator on  $L^2(\Gamma\backslash X,\ell)$.
\end{theorem}

 \begin{proof}

The operator $\mathcal{L}_{\mathcal{K}}\circ {\rm trunc}$  is a composition of  compact operators, hence $\mathcal{L}_{\mathcal{K}}\circ {\rm trunc}$ is compact.\end{proof}

This gives an analog of the well-known theorem by Selberg, Gelfand, Piateskii-Shapiro (see \cite{Garr1}, page 11) in the classical case.
  





\section{Meromorphic continuation}

We will use a refinement by Bernstein  of Selberg's method for meromorphic continuation of Eisenstein series (see \cite{BL}, \cite{Garr1, Garr2}). In particular, we will use the {\it Continuation Principle} (Theorem~\ref{bernstein}) and 
the {\it Compact Operator Criterion} (Corollary~\ref{compactop}) from the Appendix (Section~\ref{appendix}). These results require a careful discussion of extension of weak holomorphy of functions to strong holomorphy, also discussed in Section~\ref{appendix}.

\subsection{Meromorphic continuation of $E_{s}$}

Our strategy for proving meromorphic continuation of $E_{s}$ uses the Bernstein Continuation principle (\cite{BL, Garr1}  adapted to our setting.

Let   $g\in \mathcal{U} a_g K$, with $a_g\in\mathcal T$ and $K=P_{j}$ for $j=1,2$. For notational convenience, from now on we write
 $$a^s(g):=\Psi_{s}(g) = \Psi_s(a_g) \in \mathbb{C},$$
 where $a_g$ is the Iwasawa component of $g$.

 By Proposition \ref{ESCT}, the constant term of $E_s$ is 
$$C_{\mathcal{U}}E_{s}(a) = a^s+c(s)a^{1-s}.$$

Let $X_s$
be the following system of equations in $L^2(\Gamma\bsl X,\ell)$ parameterized by $s\in\C$:   

  (i) $\left[ a\frac{\partial}{\partial a} -(1-s))\right ]C_{\mathcal U}v_s=(2s-1) a^s$,
  
  (ii)  $\mathcal{L}_{\mathcal{K}}(v_{s})=\lambda _s v_{s}$,  where $\mathcal{K}$ and $\mathcal{L}_{\mathcal{K}}$ are as defined in Subsection~\ref{intop1}.

\begin{lemma} The system $X_s$ is holomorphically parameterized by $s\in\C$.
  \end{lemma}
\begin{proof}
   
By \cite{BL}, Section 2.2, it suffices to argue that the families of operators  
 
\begin{align*}
s&\mapsto\left(\left[ a\frac{\partial}{\partial a} -(1-s))\right ]C_{\mathcal U}v_s-(2s-1) a^s\right)\\
s&\mapsto (\mathcal{L}_{\mathcal{K}}-\lambda _s)(v_s)
\end{align*} 

are holomorphic for all $v_s\in L^2(\Gamma\bsl X,\ell)$. 
   
   Since $\lambda_s$ is a scalar-valued holomorphic function of $s$ and $\mathcal{L}_{\mathcal{K}}$ is constant in $s$, the family of operators $s\mapsto\mathcal{L}_{\mathcal{K}}-\lambda _s$
is holomorphic.

Since $C_{\mathcal U}$ is linear and continuous, and $a^s$ is holomorphic in $s$, the map
\[
s \mapsto \left(\left[ a\frac{\partial}{\partial a} -(1-s)\right ]C_{\mathcal U}v_s - (2s-1) a^s\right)
\]
is a holomorphic family of vectors in $L^2(\Gamma \backslash X, \ell)$. Therefore, by Garrett's Theorem [3.1], the system $X_s$ is holomorphically parameterized.

\end{proof}


The following results show that the constant term satisfies (i) in $X_s$ and $E_s$ satisfies (ii) in the region of convergence.
\begin{lemma} We have
$$( a\frac{\partial}{\partial a} -(1-s))(a^s+c(s)a^{1-s})=(2s-1) a^s.$$
\end{lemma}  

\begin{proof}
\begin{align*}
   ( a\frac{\partial}{\partial a} -(1-s))(a^s+c(s)a^{1-s})
& =  a\frac{\partial}{\partial a} (a^s+c(s)a^{1-s})-(1-s)(a^s+c(s)a^{1-s})\\
 &= ( sa ^s+(1-s)c(s)a^{1-s})- (1-s)(a^s+c(s)a^{1-s})\\
 &= (2s-1) a^s.
 \end{align*}
 
\end{proof}

\begin{lemma}\label{Teigenfn} We have $\mathcal{L}_{\mathcal{K}}(a^s)=\lambda_sa^s$ for some $\lambda_s\in\mathbb{C}$.
\end{lemma}  

\begin{proof}
For $\lambda\in\C$, let 
\begin{eqnarray}\label{vdef}
V=V_\lambda=\{f\in L^2(\Gamma\backslash X)\mid Tf = \lambda f\},
\end{eqnarray}
where $T$ is the adjacency operator as defined in (\ref{adjop}).
Let 
 $$V^K_{\lambda}=\{f\in V_\lambda\mid K\cdot f=f\}$$
 The function $a^s$ is an eigenfunction of $T$.  Moreover, any $f\in L^2(\Gamma\backslash X)$ is $K$-right invariant, so $a^s\in V^K$. Since $\mathcal{L}_{\mathcal{K}}$ is $K$-conjugation invariant, $\mathcal{L}_{\mathcal{K}}(a^s)\in V^{K}_{\lambda}$ and thus $\mathcal{L}_{\mathcal{K}}(a^s)=\lambda_sa^s$  for some $\lambda_s\in\mathbb{C}$. 
\end{proof}

\begin{corollary}\label{Es} We have $\mathcal{L}_{\mathcal{K}}(E_{s})=\lambda _s E_{s}$ in the region of convergence.
    
\end{corollary}
\begin{proof} By Lemma~\ref{Teigenfn}, 
 $\mathcal{L}_{\mathcal{K}}(a^s)=\lambda_sa^s$ for some $\lambda_s\in\mathbb{C}$. 
In the region of convergence, Eisenstein series is a sum of left translates of $a^s$. So in the region of convergence we have $\mathcal{L}_{\mathcal{K}}(E_{s})=\lambda _s E_{s}$.
\end{proof}

 \begin{theorem}\label{uniqueness}(Uniqueness) The system $X_s$ has unique solution $E_s$ in the region of convergence.
\end{theorem}
 
 \begin{proof}

We let $K = P_1$ or $P_2$, and view $E_s$ as a function on $\Gamma \bsl G/K$. 
Since $\mathcal{L}_{\mathcal{K}}$ is non-zero, $\mathcal{L}_{\mathcal{K}}$ has non-zero eigenvalues. We may choose $\mathcal{L}_{\mathcal{K}}$ so that $\lambda _s$ is not constant and $\lambda_s\neq 0$.

By Lemma~\ref{Es},  $E_{s}$ is a solution of $X_s$. Suppose that there exists another solution $v_s$ to $X_s$. Then
 
 $(i')$  [$a\frac{\partial}{\partial a} -(1-s)]C_{\mathcal U}(E_{s}-v_s)=0,$
 
$(ii')$ ($\mathcal{L}_{\mathcal{K}}-\lambda _s )(E_{s}-v_s)=0.$
 
   Equation $(i')$ implies that 
\begin{equation} \label{decay}
 C_{\mathcal U}(E_{s}-v_s)=\beta(s)a^{1-s}
\end{equation}
 for some function $\beta(s)$. 
 

   Let $f_s=E_{s}-v_s$. By Theorem \ref{equals}, $f_s-C_{\mathcal U}f_s$  is identically zero on the cusp 
   of $\Gamma \bsl X$. This together with \eqref{decay} implies that 
   $$f_s=\beta(s)a^{1-s}.$$
Since  $\beta(s)a^{1-s}\in L^2(\Gamma\bsl G)$, we have $f_s\in L^2(\Gamma\bsl G)\cap L^2(\Gamma\bsl G/K,\ell)$.
 Let $\langle\cdot ,\cdot\rangle$ be the inner product on 
 $ L^2(\Gamma\bsl G)$. Then
 $$\lambda_s\langle f_s,f_s\rangle =\langle \mathcal{L}_{\mathcal{K}}f_s,f_s\rangle 
=\langle f_s,\mathcal{L}_{\mathcal{K}}f_s\rangle =\overline{\lambda_s}\langle f_s,f_s\rangle,$$
since $\mathcal{L}_{\mathcal{K}}$ is self-adjoint. 

It follows that either $\lambda_s\in\mathbb{R}$ or $\langle f_s,f_s\rangle=0$. However $\lambda_s$ is a non-constant  function of $s$, so $\lambda_s\notin\mathbb{R}$. 
Hence $f_s=E_{s}-v_s=0$, so $E_{s}$ is the unique solution to $X_s$ in the region of convergence. 
\end{proof}

\begin{theorem}\label{finiteness}(Finiteness) The system of equations $X_s$ has a finite holomorphic envelope.
\end{theorem}
\begin{proof} Fix $s_0\in\C$. Let $v\in L^2(\Gamma\bsl X,\ell)$ be a solution to $X_s$ for $s$ close to $s_0$. Then   the constant term of $v$ is of the form
$$C_{\mathcal{U}}v=b_1(s)a^s+b_2(s)a^{1-s},$$
 since $a^s$ and $a^{1-s}$ are two linearly independent solutions of the differential equation (i) of $X_s$. 
On $\Gamma\bsl X$, we have
$$v=E({\rm tail}(b_1(s)a^s+b_2(s)a^{1-s})).$$

Let  $V':=\C\oplus\C\oplus L^2(\Gamma\bsl X,\ell)$. We define a family of continuous linear maps
$$T_s:V'\longrightarrow L^2(\Gamma\bsl X,\ell),$$
$$(b_1, b_2,f) \mapsto  E({\rm tail}(b_1a^s+b_2a^{1-s}))+f.$$

Then for each $s$, $T_s$ is a continuous linear map. We also claim that $s\mapsto T_{s}$ is holomorphic in $s$. Since restriction to $L^2(\Gamma\bsl X,\ell)$ does not depend on $s$, the restriction map is holomorphic. Since $\phi\mapsto E(\phi)$ is linear, where $E(\phi)$ is defined as in (\ref{genes}), and the $L^2(\Gamma \bsl X, \ell)$-valued map 
\[
s\mapsto E({\rm tail}(a^s))
\]is holomorphic, we conclude that the map $s\mapsto T_{s}$  is holomorphic. 



 Next define a new system $X'_s$ defined by  a single homogeneous equation $T'_s(v')=0$, where
$$T_s{'}:V'\longrightarrow L^2(\Gamma\bsl X,\ell),$$
$$T_s{'}=(\mathcal{L}_{\mathcal{K}}-\lambda _s )\circ T_s.$$



We now apply Theorem~\ref{thm:weak-strong-holomorphy} to the family of operators $T_s: V' \to L^2(\Gamma \backslash X, \ell)$.

The target space $W = L^2(\Gamma \backslash X, \ell)$ is a Hilbert space, hence is a Banach space and is  quasi-complete. The domain $V' = \mathbb{C} \oplus \mathbb{C} \oplus L^2(\Gamma \backslash X, \ell)$ is also locally convex and complete. Therefore, $\mathrm{Homo}_{\text{str}}(V', W)$ is quasi-complete.

Fix $v = (b_1, b_2, f) \in V'$ and $\mu \in W^*$. The operator $T_s$ is given by:
\[
T_s(b_1, b_2, f) = E(\operatorname{tail}(b_1 a^s + b_2 a^{1-s})) + f.
\]
Then
\[
s \mapsto \mu(T_s v) = \mu(E(\operatorname{tail}(b_1 a^s + b_2 a^{1-s}))) + \mu(f)
\]
is holomorphic since:
\begin{itemize}
    \item $\mu(f)$ is constant in $s$;
    \item $a^s$ and $a^{1-s}$ are scalar holomorphic functions;
    \item $E$ is linear and continuous, and so is $\mu \circ E$.
\end{itemize}
Thus $s \mapsto \mu(T_s v)$ is holomorphic for all $v \in V'$, $\mu \in W^*$. Hence $T_s$ is weakly holomorphic.
By  Theorem~\ref{thm:weak-strong-holomorphy}, the family $s \mapsto T_s$ is strongly holomorphic.

We now  prove that $X'_s$ has a finite holomorphic envelope, locally in a neighborhood of some fixed $s_0$ (see Appendix~\ref{appendix} for the definition of finite holomorphic envelope). 



We claim that $T_s'$ has a left inverse modulo compact operators. We use the truncation operator 
to define
\[
{\mathbf A}: L^2(\Gamma\bsl X,\ell)\longrightarrow V',\quad {\mathbf A}(f)=(0,0, {\rm trunc}(f)).
\]
Restricted to $L^2(\Gamma\bsl X,\ell)$, the operator ${\mathbf A}\circ T'_{s_0}$ is given by
 $${\mathbf A}\circ T'_{s_0}(x) =\mathcal{L}_{\mathcal{K}}\circ {\rm trunc}(x)-\lambda _{s_0}{\rm trunc}(x) = \mathcal{L}_{\mathcal{K}}\circ {\rm trunc}(x)-\lambda _{s_0}x.  $$
So ${\mathbf A}\circ T'_{s_0}$ differs from the scalar operator $\lambda _{s_0}$ by  $\mathcal{L}_{\mathcal{K}}\circ {\rm trunc}$, which is a compact operator by Theorem \ref{thm:LT}.

Applying the Compact Operator Criterion (Corollary~\ref{compactop}), we deduce that $X_s'$ has a finite holomorphic envelope. Next, we use Proposition~\ref{dominance}  with $h_s:V'\to V'$ taken to be the identity map, which implies that $X_s$ has a finite holomorphic envelope.
\end{proof}
Finally, by the Continuation Principle (Theorem~\ref{bernstein}) we conclude the following.
\begin{corollary} 
$E_{s}$ has a meromorphic continuation to $s\in\C$. 
\end{corollary}

\appendix

\section{Meromorphic continuation of Eisenstein series via Bernstein's continuation principle
\vskip .2cm By Paul Garrett}\label{appendix}

 In this section, we develop the functional analytic framework necessary for proving meromorphic continuation. We refine and extend
 Selberg's method for the meromorphic continuation of Eisenstein series (see \cite{BL}, \cite{Garr1, Garr2}).

\subsection{Operator topologies}

For $V, W$ locally convex topological vector spaces, the space $\Hom^{\circ}(V, W)$ of continuous linear maps $V \to W$ is unambiguous as a $\mathbb{C}$-vector space, but has more than one useful topology. The \emph{strong operator topology} $\Hom_{{\rm str}}(V, W)$ has a sub-basis at 0 consisting of sets
\[
\Omega^{{\rm str}}_{v,U} = \{T \in \Hom^{\circ}(V, W) : Tv \in U\} \qquad (\text{for } v \in V \text{ and open } 0 \in U \subset W)
\]
For $W$ locally convex, it is standard that its topology can be given by a separating family $S$ of seminorms (attached to Minkowski functionals, for example). In these terms, the strong operator topology sub-basis opens at 0 can be equivalently described via a sub-basis at 0 consisting of sets
\[
\Omega^{{\rm str}}_{v,\sigma,\epsilon} = \{T \in \Hom^{\circ}(V, W) : \sigma(Tv) < \epsilon\} \qquad (v \in V, 0 < \epsilon \in \mathbb{R}, \sigma \in S).
\]
The \emph{weak operator topology} $\Hom_{{\rm wk}}(V, W)$ is given by a sub-basis at 0 of sets
\[
\Omega^{{\rm wk}}_{v,\mu,\epsilon} = \{T \in \Hom^{\circ}(V, W) : |\mu(Tv)| < \epsilon\} \qquad (v \in V, 0 < \epsilon \in \mathbb{R}, \mu \in W^*).
\]
The weak operator topology is a somewhat different construction than the weak topology $W_{{\rm wk}}$ on a topological vector space $W$, with its sub-basis at 0 of sets
\[
\mathcal{U}_{\mu,\epsilon} = \{w \in W : |\mu(w)| < \epsilon\} \qquad (\text{for } \mu \in W^*)
\]
That is, the topology $\Hom_{{\rm wk}}^{\circ}(V, W)$ is \emph{not} formed in reference to a dual of a topologized $\Hom^{\circ}(V, W)$, but with respect to a subspace, consisting of functionals of the form
\[
\Lambda_{v,\mu}(T) = \mu(T(v)) = (\mu \circ T)(v) \qquad (\text{for } v \in V, \mu \in W^*, T \in \Hom^{\circ}(V, W))
\]
In particular, $\Hom_{{\rm wk}}^{\circ}(V, W)$ is \emph{not} generally $\Hom_{{\rm str}}^{\circ}(V, W_{{\rm wk}})$. Indeed, with $V = W_{{\rm wk}}$ and with the topology of $W$ strictly stronger than $W_{{\rm wk}}$, $1_W \notin \Hom^{\circ}(W_{{\rm wk}}, W)$, while $1_W \in \Hom^{\circ}(W_{{\rm wk}}, W_{{\rm wk}})$.

Meanwhile, $\Hom^{\circ}(V, \mathbb{C}) \approx V^*$, and $\Hom^{\circ}(\mathbb{C}, W) \approx W$, so discussion of spaces of continuous linear maps does include those. In the context of the two topologies on $\Hom^{\circ}(V, W)$: the weak topology on $\mathbb{C}$ (or on any finite-dimensional space) is the same as the strong topology. Both weak and strong operator topologies give the weak-dual topology $V_{{\rm wk}}^* \approx \Hom^{\circ}(V, \mathbb{C})$. Mildly consistently with the weak-strong terminology, $\Hom_{{\rm str}}(\mathbb{C}, W) \approx W$, and $\Hom_{{\rm wk}}^{\circ}(\mathbb{C}, W) \approx W_{{\rm wk}}$.

\subsection{Quasi-completeness}

Quasi-completeness turns out to be an appropriate generalization, to the non-metric-space situation, of completeness for metric spaces. Fortunately, most reasonable spaces are quasi-complete. Recall that an \textit{LF-space} is a strict inductive limit of Fr\'echet spaces.

\begin{theorem} \label{thm:quasi-complete}
For $X$ a Fréchet space or LF-space, and $Y$ quasi-complete and locally convex, the space $\Hom^{\circ}(X, Y)$ of continuous linear maps $X \to Y$, with any locally convex topology fine enough so that evaluation $T \mapsto Tx$ is a continuous map $\Hom^{\circ}(X, Y) \to Y$ for every $x \in X$, is quasi-complete.
\end{theorem}

\begin{proof}
As usual, a set $E$ of continuous linear maps from $X \to Y$ is \emph{equicontinuous} when, for every neighborhood $U$ of 0 in $Y$, there is a neighborhood $N$ of 0 in $X$ so that $TN \subset U$ for every $T \in E$.

\begin{claim} \label{claim:equicont-colimit}
Let locally convex $V$ be a strict colimit of closed subspaces $V_i$. Let $Y$ be locally convex. A set $E$ of continuous linear maps from $V$ to $Y$ is equicontinuous if and only if for each index $i$ the collection of continuous linear maps $\{T|_{V_i} : T \in E\}$ is equicontinuous.
\end{claim}
\begin{proof}[Proof of Claim]
Given open $U \ni 0$ in $Y$, shrink $U$ if necessary so that $U$ is convex and balanced. For each index $i$, let $N_i$ be a convex, balanced neighborhood of 0 in $V_i$ so that $T(N_i) \subset U$ for all $T \in E$. Let $N$ be the image in the colimit of the convex hull of the union of the images of the $N_i$'s in the coproduct. By the convexity of $N$, still $TN \subset U$ for all $T \in E$. By the construction of the colimit as a quotient of the coproduct topology, given by the diamond topology, $N$ is an open neighborhood of 0 in the colimit. This gives the equicontinuity of $E$. The opposite implication is easier.
\end{proof}

Recall

\begin{claim}[Banach-Steinhaus] \label{claim:banach-steinhaus}
Let $X$ be a Fréchet space or LF-space and $Y$ locally convex. A set $E$ of linear maps $X \to Y$, such that every set of images $E_x = \{Tx : T \in E\}$ is bounded in $Y$, is equicontinuous.
\end{claim}
\begin{proof}[Proof of Claim]
First consider $X$ Fréchet. Given a neighborhood $U$ of 0 in $Y$, let $A = \cap_{T \in E} T^{-1}\overline{U}$. By assumption, $\cup_n nA = X$. By the Baire category theorem, the complete metric space $X$ is not a countable union of nowhere dense subsets, so at least one of the closed sets $nA$ has non-empty interior. Since (non-zero) scalar multiplication is a homeomorphism, $A$ itself has non-empty interior, containing some $x + N$ for a neighborhood $N$ of 0 and $x \in A$. For every $T \in E$,
\[
TN \subset T\{a-x : a \in A\} \subset \{u_1 - u_2 : u_1, u_2 \in \overline{U}\} = \overline{U} - \overline{U}
\]
By continuity of addition and scalar multiplication in $Y$, given an open neighborhood $U_0$ of 0, there is $U$ such that $\overline{U} - \overline{U} \subset U_0$. Thus, $TN \subset U_0$ for every $T \in E$, and $E$ is equicontinuous.

For $X = \cup_i X_i$ an LF-space, this argument shows that $E$ restricted to each $X_i$ is equicontinuous. As in the previous claim, this gives equicontinuity on the strict colimit.
\end{proof}

For proof of the theorem, let $E = \{T_i : i \in I\}$ be a bounded Cauchy net in $\Hom^{\circ}(X, Y)$, with directed set $I$. Attempt to define the limit of the net $E$ by $Tx = \lim_i T_i x$. For any topology as in the statement of the theorem, for each fixed $x \in X$ the net $T_i x$ is bounded and Cauchy in $Y$. By the quasi-completeness of $Y$, $T_i x$ converges to an element of $Y$ suggestively denoted $Tx$.

To prove linearity of $T$, fix $x_1, x_2$ in $X$, $a, b \in \mathbb{C}$ and fix a neighborhood $U_0$ of 0 in $Y$. Since $T$ is in the closure of $E$, for any open neighborhood $N$ of 0 in $\Hom(X, Y)$, there exists $T_i \in E \cap (T + N)$. In particular, for any neighborhood $U$ of 0 in $Y$, take
\[
N = \{S \in \Hom^{\circ}(X,Y) : S(ax_1 + bx_2) \in U, S(x_1) \in U, S(x_2) \in U\}
\]
Since $T_i$ is linear,
\begin{align*}
T(ax_1 + bx_2) - aT(x_1) - bT(x_2) \\
= (T(ax_1+bx_2) - aT(x_1) - bT(x_2)) - (T_i(ax_1 + bx_2) - aT_i(x_1) - bT_i(x_2))
\end{align*}
The latter expression is
\[
T(ax_1 + bx_2) - T_i(ax_1 + bx_2) + a(T(x_1) - T_i(x_1)) + b(T(x_2) - T_i(x_2)) \in U + aU + bU
\]
By choosing $U$ small enough so that $U + aU + bU \subset U_0$, $T(ax_1 + bx_2) - aT(x_1) - bT(x_2) \in U_0$. This holds for every neighborhood $U_0$ of 0 in $Y$, so $T(ax_1 + bx_2) - aT(x_1) - bT(x_2) = 0$, proving linearity of $T$.

Continuity of the limit operator $T$ exactly requires equicontinuity of $E = \{T_i : i \in I\}$. Indeed, for each $x \in X$, $\{T_i x : i \in I\}$ is bounded in $Y$, so by Banach-Steinhaus, $\{T_i : i \in I\}$ is equicontinuous. Fix a neighborhood $U$ of 0 in $Y$. Invoking the equicontinuity of $E$, let $N$ be a small enough neighborhood of 0 in $X$ so that $T(N) \subset U$ for all $T\in E$. Let $x \in N$. By the characterization of the topology on $\Hom^{\circ}(X, Y)$, $Tx - T_i x \in U$ for large enough $i$. Then $Tx \in U + T_i x \subset U + U$. Replacing $U$ by $U'$ such that $U' + U' \subset U$, $T$ is continuous.
\end{proof}

\subsection{Holomorphy}

Recall standard terminology: A $W$-valued function $f$ on a non-empty open set in $\mathbb{C}$ is \emph{(strongly) complex-differentiable}, or \emph{(strongly) holomorphic}, when $\lim_{z \to z_0} (f(z) - f(z_0))/(z - z_0)$ exists in $W$ for all $z_0$ in the region, where $z \to z_0$ specificially means for complex $z$ approaching $z_0$. When $W$ is quasi-complete (and locally convex), the existence of Gelfand-Pettis integrals gives the vector-valued version of Cauchy-Goursat theory: complex-differentiability entails expandability in $W$-valued power series, as well as the rest of the Cauchy-Goursat ideas.

It is also standard to say that a $W$-valued function $f$ on a non-empty open set in $\mathbb{C}$ is \emph{weakly complex-differentiable}, or \emph{weakly holomorphic}, when $\lim_{z \to z_0} (f(z) - f(z_0))/(z - z_0)$ exists in the weak topology $W_{{\rm wk}}$ on $W$, for all $z_0$ in the set. Again, when $W_{{\rm wk}}$ is quasi-complete (and locally convex), the existence of Gelfand-Pettis integrals gives the vector-valued version of Cauchy-Goursat theory.

The following operator-valued version is not a special case, but, rather, an extension of the weak-strong terminology:

A $\Hom^{\circ}(V, W)$-valued function $f$ on a non-empty open set in $\mathbb{C}$ is \emph{strongly complex-differentiable}, or \emph{strongly holomorphic}, when $\lim_{z \to z_0} (f(z) - f(z_0))/(z - z_0)$ exists in $\Hom_{{\rm str}}^{\circ}(V, W)$ for all $z_0$ in the region, where $z \to z_0$ specificially means for complex $z$ approaching $z_0$. A $\Hom^{\circ}(V, W)$-valued function $f$ on a non-empty open set in $\mathbb{C}$ is \emph{weakly complex-differentiable}, or \emph{weakly holomorphic}, when $\lim_{z \to z_0} (f(z) - f(z_0))/(z - z_0)$ exists in $\Hom_{{\rm wk}}^{\circ}(V, W)$ for all $z_0$ in the set. Again, when $\Hom^{\circ}(V, W)$ has a quasi-complete (locally convex) topology, the existence of Gelfand-Pettis integrals gives the vector-valued version of Cauchy-Goursat theory.

In fact, the weak-to-strong assertions for the simpler and more standard cases of $W$ and $W_{{\rm wk}}$ yield, after a small further argument, the corresponding facts, about weak-to-strong holomorphy for $\Hom_{{\rm str}}^{\circ}(V, W)$ and $\Hom_{{\rm wk}}^{\circ}(V, W)$.

\begin{theorem}[Weak-to-Strong Holomorphy] \label{thm:weak-strong-holomorphy}
For $\Hom_{{\rm wk}}^{\circ}(V, W)$ and $\Hom_{{\rm str}}^{\circ}(V, W)$ quasi-complete (locally convex), $\Hom_{{\rm wk}}^{\circ}(V, W)$-holomorphic functions are $\Hom_{{\rm str}}^{\circ}(V, W)$-holomorphic.
\end{theorem}

\begin{proof}
Once we know that $f(z)$ is $\Hom_{{\rm str}}^{\circ}(V, W)$-continuous, the quasi-completeness of $\Hom_{{\rm str}}^{\circ}(V, W)$ gives existence of Gelfand-Pettis integrals
\[
I(z) = \frac{1}{2\pi i} \int_\gamma \frac{f(\zeta)}{\zeta - z} d\zeta \in \Hom_{{\rm str}}^{\circ}(V, W) \qquad (\text{small-enough circle } \gamma \text{ around } z)
\]
uniquely characterized by
\[
\Lambda(I(z)) = \frac{1}{2\pi i} \int_\gamma \frac{\Lambda(f(\zeta))}{\zeta - z} d\zeta \qquad (\text{for all } \Lambda \in \Hom_{{\rm str}}(V, W)^*)
\]
Further, since $\gamma$ is compact, as consequence of the construction of Gelfand-Pettis integrals, $I(z)$ lies in the closure of the convex hull of the collection of values $\frac{f(\gamma(t))}{\gamma(t)-z} \gamma'(t)$ of the integrand, dilated by the total measure of the compact set integrated-over. It is standard that, by the quasi-completeness $\Hom_{{\rm str}}^{\circ}(V, W)$, this closure of convex hull is itself compact.

In particular, the weak identity holds for the special functionals $\Lambda_{v,\mu}(f(z)) = \mu(f(z)(v))$, with $\mu \in W^*$ and $v \in V$. The weak operator holomorphy certainly entails that the functions $z \mapsto \Lambda_{v,\mu}(f(z))$ are holomorphic $\mathbb{C}$-valued. Thus, Cauchy's formula applies:
\[
\mu(I(z)(v)) = \Lambda_{v,\mu}(I(z)) = \frac{1}{2\pi i} \int_\gamma \frac{\Lambda_{v,\mu}(f(\zeta))}{\zeta - z} d\zeta = \frac{1}{2\pi i} \int_\gamma \frac{\mu(f(\zeta)(v))}{\zeta - z} d\zeta = \mu(f(z)(v))
\]
By Hahn-Banach, the functionals $\mu \in W^*$ separate points in $W$, so the collection of the displayed equalities gives $I(z)(v) = f(z)(v)$, for all $z$ and for all $v \in V$. Thus, $I(z) = f(z)$, since equality of operators is equality of values. That is, without having concretely identified $\Hom_{{\rm str}}^{\circ}(V, W)^*$, but only using the smaller separating family of functionals $\Lambda_{v,\mu}$,
\[
f(z) = I(z) = \frac{1}{2\pi i} \int_\gamma \frac{f(\zeta)}{\zeta - z} d\zeta \qquad (\text{in } \Hom_{{\rm str}}^{\circ}(V, W))
\]
The $\Hom_{{\rm str}}^{\circ}(V, W)$-differentiability of $f$ at $z_0$ will follow from this latter integral expression, by the difference quotient definition of derivative. Without loss of generality take $z_0 = 0$ and $f(0) = 0$, and show that $\lim_{z \to 0} f(z)/z$ exists.

Fix a small-enough $D = \{z : |z| \le r\}$ such that every $\Lambda_{v,\mu}(f(z))$ is holomorphic on a neighborhood of $D$. Since $f(0) = 0$, every
\[
\Lambda_{v,\mu}(f(z))/z = \Lambda_{v,\mu}(f(z)/z)
\]
is also holomorphic on a neighborhood of $D$. Let $\gamma$ be the circle of radius $r$ traced counter-clockwise. Since $f(0) = 0$, the previous argument applies as well to $f(z)/z$, giving an analogous integral expression
\[
\frac{f(z)}{z} = \frac{1}{2\pi i} \int_\gamma \frac{f(\zeta)}{\zeta} \frac{1}{\zeta - z} d\zeta \qquad (\text{in } \Hom_{{\rm str}}^{\circ}(V, W), \text{ for } |z| < r/2)
\]
Expecting that the limit will pass inside the integral
\[
\lim_{z \to 0} \frac{f(z)}{z} = \frac{1}{2\pi i} \int_\gamma f(\zeta) \lim_{z \to 0} \frac{1}{\zeta(\zeta-z)} d\zeta = \frac{1}{2\pi i} \int_\gamma \frac{f(\zeta)}{\zeta^2} d\zeta \qquad (\text{in } \Hom_{{\rm str}}^{\circ}(V, W), \text{ for } |z| < r/2)
\]
we propose to show that
\[
\lim_{z \to 0} \frac{1}{2\pi i} \int_\gamma f(\zeta) \cdot \left( \frac{1}{\zeta(\zeta-z)} - \frac{1}{\zeta^2} \right) d\zeta = 0 \qquad (\text{in } \Hom_{{\rm str}}^{\circ}(V, W))
\]
Simplifying,
\[
\frac{1}{\zeta(\zeta-z)} - \frac{1}{\zeta^2} = \frac{z}{\zeta^2(\zeta-z)}
\]
The scalar $z$ comes out of the integral, and we want to show
\[
\lim_{z \to 0} z \cdot \frac{1}{2\pi i} \int_\gamma \frac{f(\zeta)}{\zeta^2(\zeta-z)} d\zeta = 0
\]
Thus, it suffices to show that the integral is \emph{bounded} for $|z| < r/2$.

For $|\zeta|=r$ and $|z| < r/2$, $|1/\zeta^2(\zeta-z)| < 2/r^3$. Let ${H} \subset \Hom_{{\rm str}}^{\circ}(V, W)$ be compact, convex, balanced containing all the values $f(\zeta)$ for $\zeta$ on $\gamma$. Using the balancedness,
\[
\frac{f(\zeta)}{\zeta^2(\zeta-z)} \in \frac{2}{r^3} \cdot {H} \qquad (\text{for } |\zeta|=r \text{ and } |z| < r/2)
\]
Since a Gelfand-Pettis integral lies inside the closure of the convex hull of the collection of values of the integrand (dilated by the measure of the set integrated over),
\[
z\cdot\frac{1}{2\pi i} \int_\gamma \frac{f(\zeta)}{\zeta^2(\zeta-z)} d\zeta \in  z\cdot 2\pi r\cdot \frac{2}{r^3} {H} \qquad (\text{for all } |z| < r/2)
\]
To show that the limit is 0, given open $0 \in U \subset \Hom_{{\rm str}}^{\circ}(V, W)$ (without loss of generality convex and balanced), we will prove existence of $0 < \delta$ sufficiently small (depending on $U$) so that, for $|z| < \delta$,
\[
z \cdot 2\pi r\cdot \frac{2}{r^3} \cdot{H} \subset .
\]
Since ${H}$ is compact, it is bounded, so, for each $U$ there is $t_U > 0$ such that ${H} \subset t_U \cdot U$. Then
\[
z \cdot 2\pi r\cdot \frac{2}{r^3} \cdot H \subset z \cdot 2\pi r\cdot \frac{2}{r^3} \cdot(t_U \cdot U) = z \cdot \left(2\pi r\cdot\frac{2}{r^3}t_U\right) \cdot U \qquad (\text{for all } |z| < r/2)
\]
Take $\delta = \left(2\pi r\cdot\frac{2 t_U}{r^3}\cdot t_U\right)^{-1}$. This proves the $\Hom_{{\rm str}}^{\circ}(V, W)$-differentiability of $f$ from its $\Hom_{{\rm str}}^{\circ}(V, W)$-continuity.
\end{proof}

\begin{corollary} \label{cor:cauchy-formulas}
The usual Cauchy integral formulas apply. In particular, $f$ is expressible as a convergent power series with coefficients given by $\Hom_{\rm{str}}^{\circ}(V, W)$-valued forms of Cauchy's formulas:
\[
f(z) = \sum_{n \ge 0} c_n (z - z_0)^n \quad \text{with} \quad c_n = \frac{f^{(n)}(z_0)}{n!} = \frac{1}{2\pi i} \int_\gamma \frac{f(w)}{(w - z_0)^{n+1}} dw
\]
for $\gamma$ a path with winding number +1 around $z_0$.
\end{corollary}

\begin{proof}
Without loss of generality, treat $z_0 = 0$, and $|z| < \rho|\zeta|$ with $\rho < 1$, and $|\zeta|=r$. The expansion
\[
\frac{1}{\zeta-z} = \frac{1}{\zeta} \frac{1}{1-z/\zeta} = \frac{1}{\zeta} \left( 1 + \frac{z}{\zeta} + \left(\frac{z}{\zeta}\right)^2 + \dots + \left(\frac{z}{\zeta}\right)^N + \frac{(z/\zeta)^{N+1}}{1-z/\zeta} \right)
\]
combined with an integration around $\gamma$ against $f(\zeta)$, and the Cauchy integral formula, give
\[
f(z) = \sum_{n=0}^N \left( \frac{1}{2\pi i} \int_\gamma \frac{f(\zeta)}{\zeta^{n+1}} d\zeta \right) z^n + \frac{1}{2\pi i} \int_\gamma \frac{f(\zeta)}{\zeta^{N+1}} \frac{z^{N+1}}{\zeta-z} d\zeta
\]
Much as in the previous proof, given a convex balanced neighborhood $U$ of 0 in $\Hom_{{\rm str}}^{\circ}(V, W)$, the closure ${H}$ of the convex hull of the compact set $\{f(\zeta) : |\zeta|=r\}$ is contained in some dilation $t_U U$ of $U$, and
\[
 \frac{1}{2\pi i} \int_\gamma \frac{f(\zeta)}{\zeta^{N+1}} \frac{z^{N+1}}{\zeta-z} d\zeta  \in  \frac{1}{r^{N+1}} t_U U \cdot \frac{1}{r(1-\rho)} \cdot (\rho r)^{N+1} = U \frac{t_U}{r(1-\rho)} \rho^{N+1}
\]
Since $0 < \rho < 1$, $\rho^{N+1}/r(1-\rho) < 1$ for sufficiently large $N$, so the leftover term is inside given $U$.
\end{proof}

\subsection{Variant Banach-Steinhaus/uniform boundedness}

From the Baire Category Theorem, we have a standard variant of Banach-Steinhaus/uniform boundedness):

\begin{theorem} \label{thm:uniform-boundedness}
Let $K$ be a compact, convex set in a topological vector space $V$, and $M$ a set of continuous linear maps $V \to W$ from $V$ to another topological vector space $W$. Suppose that for every individual $v \in K$ the collection of images $Mv = \{Tv : T \in M\} \subset W$ is bounded. Then these images are uniformly bounded, that is, the union of these bounded images, $\cup_{v \in K} Mv \subset W$, is bounded.
\end{theorem}

\begin{proof}
Let $X, Y$ be balanced neighborhoods of 0 in $W$ so that $\overline{X} + \overline{X} \subset Y$, and let
\[
E = \bigcap_{T \in M} T^{-1}(\overline{X}) \subset V
\]
By the boundedness of $Mv$, there is a positive integer $n$ such that $Mv \subset nX$, and then $v \in nE$. For every $v \in K$ there is such $n$, so
\[
K = \bigcup_n (K \cap nE)
\]
Since $E$ is closed, the Baire category theorem for the locally compact Hausdorff space $K$ implies that at least one set $K \cap n E$ has non-empty interior in $K$. For such $n$, let $v_0$ be an interior point of $K \cap n E$. Pick a balanced neighborhood $Z$ of 0 in $V$ such that
\[
K \cap (v_0 + Z) \subset n E
\]
Since $K$ is compact, it is bounded, so $K-v_0$ is bounded, and $K \subset v_0 + tZ$ for large enough positive real $t$. Since $K$ is convex, $z = (1-t^{-1})v_0 + t^{-1}v \in K$ for any $v \in K$ and $t \ge 1$. By boundedness of $K$,
\[
z - v_0 = t^{-1}(v - v_0) \in Z \qquad (\text{for large enough } t)
\]
so $z \in v_0 + Z$ for such $t$. Thus, for such $t$, $z \in K \cap (v_0 + Z) \subset n E$. From the definition of $E$, $TE \subset \overline{X}$, so $T(n E) = n T(E) \subset n \overline{X}$. Then $v = tz - (t-1)v_0$ yields
\[
Tv \in  t n \overline{X} - (t-1) n \overline{X} \subset t n (\overline{X}+\overline{X})
\]
by the balanced-ness of $X$. Since $\overline{X}+\overline{X} \subset Y$, we have $MK \subset t n Y$. Since $Y$ was arbitrary, this proves the boundedness of $MK$.
\end{proof}

\subsection{$\Hom_{{\rm wk}}^{\circ}(V, W)$-boundedness implies $\Hom_{{\rm str}}^{\circ}(V, W)$-boundedness}

First, recall the somewhat surprising, but standard,

\begin{theorem} \label{thm:weak-strong-boundedness}
A $W_{{\rm wk}}$-bounded subset of a locally convex topological vector space $W$ is $W$-bounded.
\end{theorem}

\begin{proof}
The second polar $N^{\circ\circ} \subset W$ of an open neighborhood $N$ of 0 in $W$ is
\[
N^{\circ\circ} = \{w \in W : |\lambda w| \le 1 \text{ for all } \lambda \in N^\circ\}
\]
where $N^\circ \subset W^*$ is the polar of $N$. Banach-Alaoglu asserts that $N^\circ$ is compact in the weak$^*$ dual topology in $W^*$.

\begin{claim} \label{claim:second-polar}
For $N$ a convex, balanced neighborhood of $0 \in W$, the second polar $N^{\circ\circ} \subset W$ of $N$ is the closure $\overline{N}$ of $N$.
\end{claim}
\begin{proof}[Proof of Claim]
Certainly $N$ is contained in $N^{\circ\circ}$, and $\overline{N}$ is contained in $N^{\circ\circ}$, since $N^{\circ\circ}$ is closed. For $w \in W$ but $w \notin \overline{N}$, by Hahn-Banach, there is $\lambda \in W^*$ such that $\lambda w > 1$ and $|\lambda w'| \le 1$ for all $w' \in N$. Thus, $\lambda$ is in $N^\circ$, and every element $w \in N^{\circ\circ}$ is in $\overline{N}$, so $N^{\circ\circ} = \overline{N}$.
\end{proof}

Returning to the proof of the theorem: take $E \subset W$ weakly bounded, and let $U$ be a neighborhood of 0 in $W$. By local convexity, there is a convex, balanced neighborhood $N$ of 0 with the closure $\overline{N}$ contained in $U$.

The $W_{{\rm wk}}$-boundedness of $E$ is that, for each $\lambda \in W^*$ there is a bound $b_\lambda$ such that $|\lambda w| \le b_\lambda$ for all $w \in E$. The functions $\lambda \mapsto \lambda w$ are continuous, so by variant Banach-Steinhaus there is a uniform constant $b < \infty$ such that $|\lambda w| \le b$ for $w \in E$ and $\lambda \in N^\circ$. Thus, $b^{-1}w$ is in the second polar $N^{\circ\circ}$ of $N$.

Since $N^{\circ\circ}$ is the closure $\overline{N}$ of $N$, $b^{-1}w \in \overline{N} \subset U$. By the balanced-ness of ${N}$, $E \subset t\overline{N} \subset tU$ for any $t > b$, so $E$ is bounded.
\end{proof}

The following is not a special case of the theorem, but an extension:

\begin{corollary} \label{cor:hom-wk-str-boundedness}
For locally convex topological vector spaces $V, W$, $\Hom_{{\rm wk}}^{\circ}(V, W)$-boundedness implies \newline $\Hom_{\rm{str}}^{\circ}(V, W)$-boundedness.
\end{corollary}

\begin{proof}
Let $E \subset \Hom^{\circ}(V, W)$ which is $\Hom_{\rm{wk}}^{\circ}(V, W)$-bounded. That is, for each sub-basis set
\[
U_{v,\lambda,\delta} = \{T \in \Hom^{\circ}(V, W) : |\lambda(Tv)| < \delta\} \qquad (v \in V, \lambda \in W^*, \delta > 0).
\]
Then $E \subset t U_{v,\lambda,\delta}$ for all sufficiently large $t > 0$. This is equivalent to the $W_{{\rm wk}}$-boundedness of the collection $E v = \{Tv : T \in E\}$, for each $v \in V$. Thus, by the theorem, $Ev$ is $W$-bounded. That is, for every $\Hom_{{\rm str}}^{\circ}(V, W)$ sub-basis set
\[
U_{v,Y} = \{T \in \Hom^{\circ}(V, W) : Tv \in Y\} \qquad (v \in V, \text{ open } 0 \in Y \subset W)
\]
we do have $E \subset t U_{v,Y}$ for all sufficiently large $t > 0$. That is, $E$ is $\Hom_{{\rm str}}^{\circ}(V, W)$-bounded.
\end{proof}

\subsection{Weak operator holomorphy implies strong operator continuity}\hfill\\

\begin{claim} \label{claim:wk-hol-implies-str-cont}
For locally convex topological vector spaces $V, W$, $\Hom_{{\rm wk}}^{\circ}(V, W)$-holomorphy of $f \in \Hom^{\circ}(V, W)$ implies $\Hom_{{\rm str}}(V, W)$-continuity.
\end{claim}

\begin{proof}
Without loss of generality, prove continuity at $z_0 = 0$ with $f(0) = 0$. Since $z \mapsto \lambda(f(z)(v))$ is holomorphic for each $\lambda \in W^*$ and $v \in V$, and vanishes at 0, each function $\lambda(f(z)(v))/z = \lambda(f(z)(v)/z)$ extends to a holomorphic function on a disk at 0. By Cauchy theory for scalar-valued holomorphic functions,
\[
\frac{\lambda(f(z)(v))}{z} = \frac{1}{2\pi i} \int_\gamma \frac{1}{\zeta-z} \frac{\lambda(f(\zeta)(v))}{\zeta} d\zeta
\]
where $\gamma$ is a circle of small-enough radius $r$ centered at 0, and $|z| < r/2$. With $M_\lambda$ the sup of $|\lambda(f(\zeta)(v))|$ on $\gamma$ for fixed $v \in V$,
\[
\left| \frac{\lambda(f(z)(v))}{z} \right| \le \frac{\text{length}\,\gamma}{2\pi} \frac{1}{r-r/2} \frac{M_\lambda}{r} = \frac{2\pi r}{2\pi} \frac{1}{r/2} \frac{M_\lambda}{r} = \frac{M_\lambda}{r/2}.
\]
That is, $E = \{f(z)/z : |z| \le r/2\} \subset \Hom^{\circ}(V, W)$ is $\Hom_{{\rm wk}}^{\circ}(V, W)$-bounded. Thus, by the previous, $E$ is $\Hom_{{\rm str}}^{\circ}(V, W)$-bounded.

That is, given a balanced convex neighborhood $\Omega$ of 0 in $\Hom_{{\rm str}}^{\circ}(V, W)$, there is $t_0 > 0$ such that for complex $\beta$ with $|\beta| \ge t_0$, $E \subset \beta \Omega$. That is, for $|z| \le r/2$, $f(z)/z \in \beta \Omega$, so $f(z) \in z \beta \Omega$. Thus, for $|z| < \delta \le r/2$, $f(z) \in \delta \cdot \beta \Omega$. As $f(0)=0$, for $|z| < 1/\beta$ (and $|z| \le r/2$), $f(z) - f(0) \in \Omega$. This is $\Hom_{{\rm str}}^{\circ}(V, W)$-continuity.
\end{proof}

\subsection{Composition of weak-operator families}

Let $V, W, X$ be locally convex topological vector spaces.

\begin{theorem} \label{thm:composition-weak}
Let $s \mapsto S_s$ be $\Hom_{{\rm wk}}^{\circ}(V, W)$-holomorphic and $t \mapsto T_t$ be $\Hom_{{\rm wk}}^{\circ}(W, X)$-holomorphic, for $s, t$ both in a non-empty open $\Omega \subset \mathbb{C}$. Then the diagonal $z \mapsto T_z \circ S_z$ is $\Hom_{{\rm wk}}^{\circ}(V, X)$-holomorphic, for $z \in \Omega$.
\end{theorem}

\begin{proof}
First, claim that, for all $v \in V$ and for all $\xi \in X^*$, the two-variable function $(s, t) \mapsto \xi((T_t \circ S_s)(v))$ is separately weak-operator holomorphic, for all $v \in V$ and $\xi \in X^*$. Indeed, for fixed $s$, $t \mapsto \xi(T_t(S_s v))$ is scalar-holomorphic, by the weak-operator holomorphy of $T_t$. And, for fixed $t$, $\xi \circ T_t \in W^*$, so $s \mapsto (\xi \circ T_t)(S_s v)$ is scalar-holomorphic. This gives the separate one-variable weak-operator holomorphy assertions.

Then invoke Hartogs' theorem: separate analyticity (without any further hypotheses) of $\mathbb{C}$-valued functions implies joint analyticity. Restricting to the diagonal,
\[
z \mapsto \xi((T_z \circ S_z)(v))
\]
is certainly holomorphic, for all $v$ and $\xi$, giving weak holomorphy of $T_z \circ S_z$.
\end{proof}

\subsection{Strong-operator holomorphy of composite maps}

Let $V, W, X$ be locally convex topological vector spaces, with $V$ Hilbert, Banach, Fréchet, or LF, and $W$ quasi-complete.

\begin{theorem} \label{thm:composition-strong}
Let $s \mapsto S_s$ be $\Hom_{{\rm wk}}^{\circ}(V, W)$-holomorphic and $t \mapsto T_t$ be $\Hom_{{\rm wk}}^{\circ}(W, X)$-holomorphic, for $s, t$ both in a non-empty open $\Omega \subset \mathbb{C}$. Then the diagonal $z \mapsto T_z \circ S_z$ is $\Hom_{{\rm str}}^{\circ}(V, X)$-holomorphic, for $z \in \Omega$.
\end{theorem}

\begin{proof}
The previous theorem shows that $z \mapsto S_z \circ T_z$ is \emph{weak}-operator holomorphic. Since $V$ is LF and $X$ is quasi-complete, the composite is \emph{strong}-operator holomorphic, from above.
\end{proof}

\subsection{Vector-valued power series}

Power series with values in a quasi-complete, locally compact vectorspace $V$ behave essentially as well as scalar-valued ones. First,

\begin{lemma} \label{lem:vector-series-convergence}
Let $c_n$ be a bounded sequence of vectors in the locally convex, quasi-complete topological vector space $V$. Let $z_n$ be a sequence of complex numbers, let $0 \le r_n$ be real numbers such that $|z_n| \le r_n$, and suppose that $\sum_n r_n < +\infty$. Then $\sum_n c_n z_n$ converges in $V$. Further, given a convex balanced neighborhood $U$ of 0 in $V$ let $t$ be a positive real such that for all complex $w$ with $|w| \ge t$ we have $\{c_n\} \subset tU$. Then
\[
\sum_n c_n z_n \in \left( \sum_n |z_n| \right)\cdot tU \subset \left( \sum_n r_n \right) \cdot tU
\]
\end{lemma}

\begin{proof}
For a convex balanced neighborhood $N$ of 0 in the topological vector space, with complex numbers $z$ and $w$ such that $|z| \le |w|$, then $zN \subset wN$, since $|z/w| \le 1$ implies $(z/w)N \subset N$, or $zN \subset wN$. Further, for an absolutely convergent series $\sum_n a_n$ of complex numbers, for any $n_0$
\[
\sum_{n \le n_0} (\a_n \cdot V) = \sum_{n \le n_0} (|\a_n| \cdot V  )\subset \left( \sum_{n \le n_0} |\a_n| \right) \cdot N \subset \left( \sum_{n < \infty} |\a_n| \right) \cdot N
\]
For a balanced open $U$ containing 0, let $t$ be large enough such that for any complex $w$ with $|w| \ge t$ the sequence $c_n$ is contained in $wU$. The previous discussion shows that
\[
\sum_{m \le \ell \le n} c_\ell z_\ell \in (|z_m| + \dots + |z_n|) \cdot tU
\]
Given $\epsilon > 0$, invoking absolute convergence, take $m$ sufficiently large such that $|z_m| + \dots + |z_n| < t \cdot \epsilon$ for all $n \ge m$. Then
\[
\sum_{m \le \ell \le n} c_\ell z_\ell \in t \cdot (\epsilon/t) \cdot U =  U
\]
Thus, the original series is convergent. Since $X$ is quasi-complete the limit exists in $V$. The last containment assertion follows from this discussion, as well.
\end{proof}

\begin{corollary} \label{cor:vector-series-holomorphy}
Let $c_n$ be a bounded sequence of vectors in a locally convex quasi-complete topological vector space $V$. Then on $|z| < 1$ the series $f(z) = \sum_n c_n z^n$ converges and gives a \emph{holomorphic} $V$-valued function. That is, the function is infinitely-many-times complex-differentiable.
\end{corollary}

\begin{proof}
The lemma shows that the series expressing $f(z)$ and its apparent $k^{\text{th}}$ derivative $\sum_n c_n \binom{n}{k} z^{n-k}$ all converge for $|z| < 1$. The usual direct proof of Abel's theorem on the differentiability of (scalar-valued) power series can be adapted to prove the infinite differentiability of the $X$-valued function given by this power series, as follows. Let
\[
g(z) = \sum_{n \ge 0} n c_n z^{n-1}
\]
Then
\[
\frac{f(z) - f(w)}{z-w} - g(w) = \sum_{n \ge 1} c_n \left( \frac{z^n - w^n}{z-w} - n w^{n-1} \right)
\]
For $n=1$, the expression in the parentheses is 1. For $n>1$, it is
\begin{align*}
(z^{n-1} + z^{n-2}w + \dots + z w^{n-2} + w^{n-1}) - n w^{n-1} \\
= (z^{n-1} - w^{n-1}) + (z^{n-2}w - w^{n-1}) + \dots + (z^2 w^{n-3} - w^{n-1}) + (z w^{n-2} - w^{n-1}) + (w^{n-1} - w^{n-1}) \\
= (z-w) [ (z^{n-2} + \dots + w^{n-2}) + w(z^{n-3} + \dots + w^{n-3}) + \dots + w^{n-3}(z+w) + w^{n-2} + 0 ] \\
= (z-w) \sum_{k=0}^{n-2} (k+1) z^{n-2-k} w^k
\end{align*}
For $|z| \le r$ and $|w| \le r$ the latter expression is dominated by
\[
|z-w| \cdot r^{n-2} \frac{n(n-1)}{2} < |z-w| \cdot n^2 r^{n-2}
\]
Let $U$ be a balanced neighborhood of 0 in $X$, and $t$ a sufficiently large real number such that for all complex $w$ with $|w| \ge t$ all $c_n$ lie in $wU$. For $|z| \le r < 1$ and $|w| \le r < 1$, by the lemma,
\[
\frac{f(z) - f(w)}{z-w} - g(w) = (z-w) \sum_{n \ge 2} c_n \left( \sum_{k=0}^{n-2} (k+1) z^{n-2-k} w^k \right) \in (z-w) \cdot \left( \sum_n n^2 r^{n-2} \right) \cdot tU
\]
Thus, as $z \to w$, eventually $\frac{f(z)-f(w)}{z-w} - g(w)$ lies in $U$.
\end{proof}

\begin{corollary} \label{cor:banach-series-holomorphy}
Let $c_n$ be a sequence of vectors in a Banach space $X$ such that for some $r > 0$ the series $\sum |c_n| r^n$ converges in $X$. Then for $|z| < r$ the series $f(z) = \sum c_n z^n$ converges and gives a holomorphic (infinitely-many times complex-differentiable) $X$-valued function.
\end{corollary}
\begin{proof}
(Proof omitted, standard result).
\end{proof}




A {\it Gelfand-Pettis} or {\it weak} integral of a function $s\to
f(s)$ on a measure space $(X,\mu)$ with values in a topological
vector space $V$ is an element $I\in V$ so that for all $\lambda\in V^*$
$$\lambda(I) = \int_X \lambda(f(s)) d\mu(s).$$
A topological vector space is {\it quasi-complete} when every {\it
bounded} (in the topological vector space sense, not necessarily the
metric sense) Cauchy {\it net} is convergent.

\begin{theorem} Continuous compactly-supported functions $f:X\to V$ with values
in {\it quasi-complete} (locally convex) topological vector spaces $V$
have Gelfand-Pettis integrals with respect to finite positive regular
Borel measures $\mu$ on compact spaces $X$, and these integrals are
{\it unique}. In particular, for a $\mu$ with total measure
$\mu(X)=1$, the integral $\int_X\,f(x)\,d\mu(s)$ lies in the closure
of the convex hull of the image $f(X)$ of the measure space $X$.
\end{theorem}

\begin{proof} Bourbaki's {\it Integration}. (Thanks to Jacquet for bringing
this reference to my attention.) 
\end{proof}
\begin{corollary} Let $T:V\to W$ be a continuous linear map, and let $f:X\to V$
be a continuous compactly supported $V$-valued function on a
topological measure space $X$ with finite positive Borel measure
$\mu$. Suppose that $V$ is locally convex and quasi-complete, so that
(from above) a Gelfand-Pettis integral of $f$ exists and is
unique. Suppose that $W$ is locally convex. Then
$$ T\left( \int_X \, f(x) \, d\mu(x) \right) \;\;=\;\; \int_X \, Tf(x) \,
d\mu(x). $$
In particular, $ T\left( \int_X \, f(x) \, d\mu(x) \right)$ is a
Gelfand-Pettis integral of $T\circ f$.
\end{corollary}

\begin{proof} First, the integral exists in $V$, from above. Second, for any
continuous linear functional $\lambda$ on $W$, $\lambda\circ T$ is a
continuous linear functional on $V$. Thus, by the defining property of
the Gelfand-Pettis integral, for every such $\lambda$
$$
(\lambda\circ T)\left( \int_X \, f(x) \, d\mu(x) \right) 
\;=\; \int_X \, (\lambda T f)(x) \, d\mu(x).
$$ 
This exactly asserts that $ T\left( \int_X \, f(x) \, d\mu(x) \right)$
is a Gelfand-Pettis integral of the $W$-valued function $T\circ f$.
Since the two vectors $ T\left( \int_X \, f(x) \, d\mu(x) \right)$ and
$\int_X \, Tf(x) \, d\mu(x)$ give identical values under all
$\lambda\in W^*$, and since $W$ is locally convex, these two vectors
are equal, as claimed. 
\end{proof}

\begin{corollary}\label{weaktostrong} For quasi-complete and locally convex $V$, weakly holomorphic
$V$-valued functions are (strongly) holomorphic.
\end{corollary}

\begin{proof} The Cauchy integral formulas involve continuous integrals on
compacts, so these integrals exist as Gelfand-Pettis integrals. Thus,
we can obtain $V$-valued convergent power series expansions for weakly
holomorphic $V$-valued functions, from which (strong) holomorphy
follows by an obvious extension of Abel's theorem that analytic
functions are holomorphic. \end{proof}

Give the space $\Hom^o(X,Y)$ of continuous mappings $T:X\to Y$
from an LF-space $X$ (strict colimit of Fr\'echet spaces, e.g., a
Fr\'echet space) to a quasi-complete space $Y$ the {\it weak operator}
topology as follows. For $x\in X$ and $\mu\in Y^*$, define a seminorm
$p_{x,\mu}$ on $\Hom^o(X,Y)$ by
$$
p_{x,\mu}(T) \;=\; |\mu(T(x))|.
$$

\begin{corollary} With the weak topoloogy $\Hom^o(X,Y)$ is quasi-complete.
\end{corollary}

\begin{proof} The collection of finite linear combinations of the functionals $T
\to \mu(Tx) $ is exactly the dual space of $\Hom^o(X,Y)$ with the weak
operator topology. Invoke the previous result. 
\end{proof}

\begin{corollary} A weakly holomorphic $\Hom^o(X,Y)$-valued function $T_s$ is
holomorphic when $\Hom^o(X,Y)$ is given the weak operator
topology. \end{corollary}


\subsection{A continuation principle}

Let $V$ be a topological vector space. Following \cite{BL}, a {\it
system of linear equations} $X_0$ {\it in} $V$ is a collection 
$$
X_0 \;=\; \{(W_i, T_i, w_i): i\in I \}
$$
where $I$ is a (not necessarily countable) set of indices, each $W_i$
is a topological vector space, 
$$
T_i:V\to W_i
$$
is a continuous linear map for each index $i$, and $w_i\in W_i$ are
the {\it targets}. A {\it solution} of the system $X_0$ is $v\in V$ such that
$T_i(v)=w_i$ for all indices $i$. Denote the set of solutions by $\Sol
X_0$.

When the systems of linear equations $X_s=\{W_i, T_{i,s}, w_{i,s}\}$
depend on a parameter $s$, with $T_{i,s}$ and $w_{i,s}$ weakly
holomorphic in $s$, say that the {\it parametrized linear system}
$X=\{X_s:s\in S\}$ is {\it holomorphic} in $s$. Note that $\{W_i\}$
does not depend upon $s$.

For $X=\{X_s\}$ a parametrized system of linear equations in a space
$V$, holomorphic in $s$, suppose there is a finite-dimensional space
$F$, a weakly holomorphic family $\{f_s\}$ of continuous linear maps
$f_s:F\to V$ such that, for each $s$, $\im f_s \supset \Sol X_s$. Then we say that $f_s$ is a
{\it finite holomorphic envelope} for the system $X$ or that $X$ is of
{\it finite type}.

For $U_\alpha, \alpha\in A$ an open cover of the parameter space, and
for each $\alpha\in A$ $\{f_s^{(\alpha)}:s\in U_\alpha\}$ is a finite
envelope for the system $X^{(\alpha)}=\{X_s: s\in U_\alpha\}$, say
that $\{f_s^{(\alpha)}:s\in U_\alpha, \alpha\in A\}$ is a {\it locally
finite holomorphic envelope} of $X$.

\begin{theorem}\label{bernstein} ({\it Bernstein}) {\it Continuation Principle:} Let $X=\{X_s:s\}$
be a {\it locally finite} system of linear equations
$$T_{i,s} : V \to W_i $$
for $s$ varying in a connected complex manifold, with each $W_i$
(locally convex and) {\it quasi-complete}. Then the {\it
continuation principle} holds. That is, if for $s$ in some non-empty
open subset there is a unique solution $v_s$, then this
solution depends meromorphically upon $s$, has a
meromorphic continuation to $s$ in the whole manifold, and for fixed
$s$ off a locally finite set of analytic hypersurfaces inside the complex
manifold, the solution $v_s$ is the {\it unique} solution to the system $X_s$.
\end{theorem}

\begin{proof} This reduces to a holomorphically parametrized version of Cramer's
rule, in light of comments above about weak-to-strong principles and
composition of weakly holomorphic maps. 

It is sufficient to check the continuation principle locally, so
let $f_s:F\to V$ be a family of morphisms so that $\im f_s \supset
\Sol X_s$, with $F$ finite-dimensional. We can reformulate the statement in
terms of the finite-dimensional space $F$. Namely, put 
$$K^+_s = \{v\in F : f_s(v) \in \Sol X_s\}
= \{\hbox{ inverse images in $F$ of solutions }\}.$$ 
(The set $K^+_s$ is an affine
subspace of $F$.) By elementary finite-dimensional linear algebra,
$X_s$ has a unique solution if and only if 
$$ \dim K^+_s \;=\; \dim \ker f_s.$$
The weak holomorphy of $T_{i,s}$ and $f_s$ yield the weak
holomorphy of the composite $T_{i,s}\circ f_s$ from the
finite-dimensional space $F$ to $W_i$, by the corollary of Hartogs'
theorem above. The finite-dimensional space $F$ is certainly LF,
and $W_i$ is quasi-complete, so by invocation of results above
on weak holomorphy the space $\Hom^o(F,W_i)$ is quasi-complete,
and a weakly holomorphic family in $\Hom^o(F,W_i)$ is in fact
holomorphic. 

Take $F=\C^n$. Using linear functionals on $V$ and $W_i$ which
separate points we can describe $\ker f_s$ and $K^+_s$ by systems of
linear equations of the forms 
$$\ker f_s = \{(x_1,\ldots,x_n) \in F: \sum_j \,
a_{\alpha\, j}\,x_j=0, \;\; \alpha \in A \},$$
$$ K^+_s = \{\hbox{ inverse images of solutions }\} = \{ (x_1,\ldots,x_n)\in F:
\sum_j \, b_{\beta\,j}\,x_j = c_{\beta}, 
 \;\;\beta \in B\},$$
where 
$a_{\alpha\,j}$, $b_{\beta\,j}$, $c_{\beta}$ all depend implicitly
upon $s$, and are holomorphic $\C$-valued functions of $s$. (The index
sets $A,B$ may be of arbitrary cardinality.) Arrange these
coefficients into matrices $M_s$, $N_s$, $Q_s$ holomorphically
parametrized by $s$, with entries 
$$M_s(\alpha,j) = a_{\alpha\,j},
\;\;\;\;\; 
N_s(\beta,j) = b_{\beta\,j}
\;\;\;\;\;
\textrm{and}
\quad
Q_s(\beta,j) = 
\left\{
\begin{matrix}
b_{\beta\,j} & \hbox{ for $1\le j \le n$} \cr
c_{\beta}  & \hbox{ for $ j=n $} \cr
\end{matrix}
\right.$$
of sizes
$\card(A)$-by-$n$, 
$\card(B)$-by-$n$ and
$\card(B)$-by-$(n+1)$ respectively. 
We have
$$ \dim \ker f_s = n - \rank M_s. $$
Certainly for all $s$
$$\rank N_s \le  \rank Q_s, $$
and if the inequality is {\it strict} then there is {\it no solution}
to the system $X_s$.
By finite-dimensional linear algebra, assuming that $\rank N_s = \rank
Q_s$,
$$ \dim K^+_s = n - \rank B_s.$$
Therefore, the condition that $\dim K^+_s = \dim \ker f_s$ can be
rewritten as
$$\rank M_s = \rank N_s = \rank Q_s.$$

Let $S_o$ be the dense subset (actually, $S_o$ is the complement
of an analytic subset) of the parameter space where $\rank M_s$,
$\rank N_s$, and $\rank 
Q_s$ all take their maximum values. Since by hypothesis $S_o\cap
\Omega$ is not empty, and since the ranks are equal for $s\in\Omega$,
all those maximal ranks are equal to the same number $r$. Then for all
$s\in S_o$ the rank condition holds and $X_s$ has a solution, and the
solution is unique.

In order to prove the continuation principle we must show that
$X=\{X_s\}$ has a meromorphic solution $v_s$. Making use of the
finite envelope of the system $X$, to find a meromorphic solution of
$X$ it is enough to find a meromorphic solution of the parametrized
system $Y = \{ Y_s \} $ where
$$ Y_s = \{ \sum\,b_{\beta\,i}x_i=c_{\beta}: \hbox{ for all }\beta
\}$$
with implicit dependence upon $s$.
Let $r$ be the maximum rank, as above. Choose $s_o\in S_o$ and choose an
$r$-by-$r$ minor 
$$D_{s_o} = \{ b_{\beta,j}: \beta\in\{\beta_1,\ldots,\beta_r\},
j \in \{j_1,\ldots,j_r\}\}$$ 
of full rank, inside the matrix
$N_{s_o}$, with constraints on the indices as indicated. Let
$S_1\subset S_o$ be the set of 
points $s$ where $D_s$ has full rank, that is, where
$\det D_s\not=0$. Consider the system of equations
$$Z = \{ \sum_{j\in\{j_1,\ldots,j_r\}}\,b_{\beta\,j} \, x_j = c_{\beta} :
\beta\in\{\beta_1,\ldots,\beta_r\}\} \;\;\;\hbox{(with $s$ implicit)}$$
By Cramer's Rule, for $s\in S_1$ this system has a unique solution
$(x_{1,s},\ldots,x_{r,s})$. Further, since the coefficients are
holomorphic in $s$, the expression for the solution obtained via
Cramer's rule show that the solution is meromorphic in $s$. 
Extending this solution by 
$x_j= 0$ for $j$ not among $j_1,\ldots,j_r$, we see that it satisfies
the $r$ equations corresponding to rows
$\beta\in\{\beta_1,\ldots,\beta_r\}$ of the system $Y_s$. Then for
$s\in S_1$ the equality $\rank N_s=\rank Q_s=r$ implies that after
satisfying the first $r$ equations of $Y_s$ it will 
automatically satisfy the rest of the equations in the system $Y_s$.

Thus, the system has a {\it weakly} holomorphic solution. Earlier
observations on weak-to-strong principles assure that this solution is
holomorphic. This proves the continuation principle. \end{proof}


\subsection{Finite envelope criteria}

\begin{proposition}\label{dominance} ({\it Dominance}) (Called {\it inference} by Bernstein.) Let
$X'=\{X'_s\}$ be another holomorphically parametrized system of
equations in a linear space $V'$, with the same parameter space as a
system $X=\{X_s\}$ on a space $V$. We say that
$X'$ {\it dominates} $X$ when there 
is a family of morphisms $h_s:V'\to V$, weakly holomorphic in
$s$, so that
 $$\Sol X_s \;\subset\; h_s (\Sol X'_s)$$
for all $s$. If $X'_s$ is locally finite then $X_s$ is locally finite.
\end{proposition}

\begin{proof} The fact that compositions of weakly holomorphic mappings are
weakly holomorphic assures that $X'_s$ really meets the definition of
{\it system}. Granting this, the conclusion is clear. \end{proof}


\begin{theorem} ({\it Banach-space criterion}) Let $V$ be a Banach space, and
$X$ a single parametrized homogeneous equation $T_s(v)=0$, with
$T_s:V\to W$, where $W$ is also a Banach space, and where $s\to T_s$
is holomorphic for the uniform-norm Banach-space topology on
$\Hom^o(V,W)$. If for some fixed $s_o$ there exists an operator
${\mathbf A}:W\to V$ so that ${\mathbf A}\circ T_{s_o}$ has {\it finite-dimensional
kernel} and {\it closed image}, then $X_s$ is of {\it finite type} in
some neighborhood of $s$.
\end{theorem}

\begin{proof} Let $V_1$ be the image of ${\mathbf A}\circ T_{s_o}$, and 
$V_o$ the kernel of ${\mathbf A}\circ T_{s_o}$.

We claim that finite dimensional $V_o\subset V$ has a continuous
linear $p:V\to V_o$ which is the identity on $V_o$.
Indeed, for a basis
$v_1,\ldots,v_n$ of $V_o$, and for $v\in V_o$, the coefficients
$c_i(v)$ in the expression $v=\sum_i c_i(v)v_i$ are continuous linear
functionals on $V_o$. By Hahn-Banach, each $c_i$ extends to a
continuous linear functional $\lambda_i$ on $V$, and
$p(v)=\lambda_1(v)v_1+\ldots+\lambda_n(v)v_n$ is as desired.

Let $q={\mathbf A}\circ T_{s_o}:V\to V_1$.

Let $X'_s$ be a new system in $V$, given by a single equation
$T'_s(v)=0$, where $T'_s=q\circ T_s:V\to V_1$. If $T_s(v)=0$, then
$T'_s(v)=0$, so $X'_s$ {\it dominates} $X_s$.

Since $V_1\subset V$ is {\it closed}, it is a Banach space. Consider
the holomorphic family of maps 
$$
\varphi_s \;=\; p \oplus T'_s \;:\; V \to V_o\oplus V_1
$$
where $V_o\oplus V_1$ is given its natural Banach space structure. The
function $s\to \varphi_s$ is holomorphic for the operator-norm topology on
$\Hom^o(V,V_o\oplus V_1)$.

By construction, $\varphi_{s_o}$ is a bijection, so by the Open Mapping
Theorem it is an isomorphism. The continuous inverse $\varphi_{s_o}^{-1}$
has an operator norm $\delta^{-1}$ with $0<\delta^{-1}<+\infty$. With
$s$ sufficiently near $s_o$ so that $|\varphi_{s_o}-\varphi_s|<\delta/2$,
$$
|\varphi_s(x)| \;\ge\; |\varphi_{s_o}(x)| - |\varphi_s(x)-\varphi_{s_o}(x)| \;\ge\;
\delta\cdot |x|-{\delta\over 2}\cdot |x|\;\ge\; {\delta\over 2}\cdot |x|.
$$
Thus, $\varphi_s$ is an isomorphism for $s$ sufficiently near $s_o$. The map $s\to \varphi_s^{-1}$ is holomorphic on a neighborhood of $s_o$,
since the operator-norm topology restricted to invertible elements in
$\Hom^o(V,V_o\oplus V_1)$ is the same as the operator-norm topology
restricted to invertible elements in $\Hom^o(V_o\oplus V_1, V)$. This
follows from the continuity of $T\to T^{-1}$ on a neighborhood of an
invertible operator.

There is a finite envelope $\varphi_s^{-1}(V_o\oplus\{0\})$ for $X'_s$. By
{\it dominance}, there is a finite envelope for $X_s$. \end{proof}


\begin{corollary}\label{compactop} {\it (Compact operator criterion)} Let $V$ be a Banach space with
system $X_s$ given by a single equation $T_s:V\to W$, with Banach space
$W$, requiring $T_s(v)=0$, with $s\to T_s$ holomorphic for the
operator-norm topology. Suppose for some $s_o$ the operator $T_{s_o}$
has a left inverse modulo compact operators, that is, there exists
${\mathbf A}:W\to V$ such that
$$
{\mathbf A}\circ T_{s_o} \;=\; 1_V + \hbox{(compact operator)}.
$$
Then $X_s$ is of finite type in some neighborhood of $s_o$.
\end{corollary}

\begin{proof} Let $K$ be that compact operator. The kernel $V_o=\ker(1_V+K)$ is the
$-1$ eigenspace for $K$, finite-dimensional by the spectral theory of
compact (not necessarily self-adjoint or normal) operators. Similarly,
the image $V_1$ is closed. Thus, the theorem applies. \end{proof}

\bibliographystyle{amsalpha}
\bibliography{Eisref}{}

@article {AC2017,
    AUTHOR = {Allen, Robert F. and Craig, Isaac M.},
     TITLE = {Multiplication operators on weighted {B}anach spaces of a
              tree},
   JOURNAL = {Bull. Korean Math. Soc.},
  FJOURNAL = {Bulletin of the Korean Mathematical Society},
    VOLUME = {54},
      YEAR = {2017},
    NUMBER = {3},
     PAGES = {747--761},
      ISSN = {1015-8634,2234-3016},
   MRCLASS = {47B38 (05C05)},
  MRNUMBER = {3659146},
MRREVIEWER = {Gerg\H o\ Nagy},
       DOI = {10.4134/BKMS.b160007},
       URL = {https://doi.org/10.4134/BKMS.b160007},
}

@article {A,
    AUTHOR = {Arthur, James},
     TITLE = {A truncation process for reductive groups},
   JOURNAL = {Bull. Amer. Math. Soc.},
  FJOURNAL = {Bulletin of the American Mathematical Society},
    VOLUME = {83},
      YEAR = {1977},
    NUMBER = {4},
     PAGES = {748--750},
      ISSN = {0002-9904},
   MRCLASS = {22E55 (32N10)},
  MRNUMBER = {480886},
       DOI = {10.1090/S0002-9904-1977-14380-2},
       URL = {https://doi.org/10.1090/S0002-9904-1977-14380-2},
}

@article {AC,
    AUTHOR = {Ali, Abid and Carbone, Lisa},
     TITLE = {Congruence subgroups of lattices in rank 2 {K}ac-{M}oody
              groups over finite fields},
   JOURNAL = {Comm. Algebra},
  FJOURNAL = {Communications in Algebra},
    VOLUME = {44},
      YEAR = {2016},
    NUMBER = {3},
     PAGES = {1236--1264},
      ISSN = {0092-7872,1532-4125},
   MRCLASS = {20G44 (20M05)},
  MRNUMBER = {3463141},
MRREVIEWER = {Zhenheng\ Li},
       DOI = {10.1080/00927872.2015.1012673},
       URL = {https://doi.org/10.1080/00927872.2015.1012673},
}

@article {BL,
    AUTHOR = {Bernstein, Joseph and Lapid, Erez},
     TITLE = {On the meromorphic continuation of {E}isenstein series},
   JOURNAL = {J. Amer. Math. Soc.},
  FJOURNAL = {Journal of the American Mathematical Society},
    VOLUME = {37},
      YEAR = {2024},
    NUMBER = {1},
     PAGES = {187--234},
      ISSN = {0894-0347,1088-6834},
   MRCLASS = {11F70},
  MRNUMBER = {4654611},
       DOI = {10.1090/jams/1020},
       URL = {https://doi.org/10.1090/jams/1020},
}

@article {BG,
    AUTHOR = {Braverman, A. and Gaitsgory, D.},
     TITLE = {Geometric {E}isenstein series},
   JOURNAL = {Invent. Math.},
  FJOURNAL = {Inventiones Mathematicae},
    VOLUME = {150},
      YEAR = {2002},
    NUMBER = {2},
     PAGES = {287--384},
      ISSN = {0020-9910,1432-1297},
   MRCLASS = {11G45 (11F60 14F43)},
  MRNUMBER = {1933587},
MRREVIEWER = {Igor\ Yu.\ Potemine},
       DOI = {10.1007/s00222-002-0237-8},
       URL = {https://doi.org/10.1007/s00222-002-0237-8},
}

@book {Bu,
    AUTHOR = {Bump, Daniel},
     TITLE = {Automorphic forms and representations},
    SERIES = {Cambridge Studies in Advanced Mathematics},
    VOLUME = {55},
 PUBLISHER = {Cambridge University Press, Cambridge},
      YEAR = {1997},
     PAGES = {xiv+574},
      ISBN = {0-521-55098-X},
   MRCLASS = {11F70 (11F41 11R39 22E50 22E55)},
  MRNUMBER = {1431508},
MRREVIEWER = {Solomon\ Friedberg},
       DOI = {10.1017/CBO9780511609572},
       URL = {https://doi.org/10.1017/CBO9780511609572},
}

@article {CG,
    AUTHOR = {Carbone, Lisa and Garland, Howard},
     TITLE = {Existence of lattices in {K}ac-{M}oody groups over finite
              fields},
   JOURNAL = {Commun. Contemp. Math.},
  FJOURNAL = {Communications in Contemporary Mathematics},
    VOLUME = {5},
      YEAR = {2003},
    NUMBER = {5},
     PAGES = {813--867},
      ISSN = {0219-1997,1793-6683},
   MRCLASS = {17B67 (20E42 22E40 22F50)},
  MRNUMBER = {2017720},
MRREVIEWER = {Guy\ Rousseau},
       DOI = {10.1142/S0219199703001117},
       URL = {https://doi.org/10.1142/S0219199703001117},
}

@article {CER,
    AUTHOR = {Carbone, Lisa and Ershov, Mikhail and Ritter, Gordon},
     TITLE = {Abstract simplicity of complete {K}ac-{M}oody groups over
              finite fields},
   JOURNAL = {J. Pure Appl. Algebra},
  FJOURNAL = {Journal of Pure and Applied Algebra},
    VOLUME = {212},
      YEAR = {2008},
    NUMBER = {10},
     PAGES = {2147--2162},
      ISSN = {0022-4049,1873-1376},
   MRCLASS = {20E42 (17B67)},
  MRNUMBER = {2418160},
MRREVIEWER = {Pierre-Emmanuel\ Caprace},
       DOI = {10.1016/j.jpaa.2008.03.023},
       URL = {https://doi.org/10.1016/j.jpaa.2008.03.023},
}

@article {CLL,
    AUTHOR = {Carbone, Lisa and Lee, Kyu-Hwan and Liu, Dongwen},
     TITLE = {Eisenstein series on rank 2 hyperbolic {K}ac-{M}oody groups},
   JOURNAL = {Math. Ann.},
  FJOURNAL = {Mathematische Annalen},
    VOLUME = {367},
      YEAR = {2017},
    NUMBER = {3-4},
     PAGES = {1173--1197},
      ISSN = {0025-5831,1432-1807},
   MRCLASS = {11F70 (20G44 22E30)},
  MRNUMBER = {3623222},
MRREVIEWER = {Anne-Marie\ H.\ Aubert},
       DOI = {10.1007/s00208-016-1428-8},
       URL = {https://doi.org/10.1007/s00208-016-1428-8},
}

@article {CFF,
    AUTHOR = {Carbone, Lisa and Feingold, Alex J. and Freyn, Walter},
     TITLE = {A lightcone embedding of the twin building of a hyperbolic
              {K}ac-{M}oody group},
   JOURNAL = {SIGMA Symmetry Integrability Geom. Methods Appl.},
  FJOURNAL = {SIGMA. Symmetry, Integrability and Geometry. Methods and
              Applications},
    VOLUME = {16},
      YEAR = {2020},
     PAGES = {Paper No. 045, 47},
      ISSN = {1815-0659},
   MRCLASS = {20G44 (20E42 20F05 22E65 51E24)},
  MRNUMBER = {4104416},
MRREVIEWER = {Guy\ Rousseau},
       DOI = {10.3842/SIGMA.2020.045},
       URL = {https://doi.org/10.3842/SIGMA.2020.045},
}

@article {CMS,
    AUTHOR = {Cowling, Michael and Meda, Stefano and Setti, Alberto G.},
     TITLE = {Estimates for functions of the {L}aplace operator on
              homogeneous trees},
   JOURNAL = {Trans. Amer. Math. Soc.},
  FJOURNAL = {Transactions of the American Mathematical Society},
    VOLUME = {352},
      YEAR = {2000},
    NUMBER = {9},
     PAGES = {4271--4293},
      ISSN = {0002-9947,1088-6850},
   MRCLASS = {43A85 (20E08 35J05 43A90)},
  MRNUMBER = {1653343},
MRREVIEWER = {Wojciech\ Lisiecki},
       DOI = {10.1090/S0002-9947-00-02460-0},
       URL = {https://doi.org/10.1090/S0002-9947-00-02460-0},
}

@article {Cow,
    AUTHOR = {Cowling, Michael and Meda, Stefano and Setti, Alberto G.},
     TITLE = {An overview of harmonic analysis on the group of isometries of
              a homogeneous tree},
   JOURNAL = {Exposition. Math.},
  FJOURNAL = {Expositiones Mathematicae. International Journal},
    VOLUME = {16},
      YEAR = {1998},
    NUMBER = {5},
     PAGES = {385--423},
      ISSN = {0723-0869},
   MRCLASS = {43A85 (39A12)},
  MRNUMBER = {1656839},
}

@article {DJ,
    AUTHOR = {Dymara, Jan and Januszkiewicz, Tadeusz},
     TITLE = {Cohomology of buildings and their automorphism groups},
   JOURNAL = {Invent. Math.},
  FJOURNAL = {Inventiones Mathematicae},
    VOLUME = {150},
      YEAR = {2002},
    NUMBER = {3},
     PAGES = {579--627},
      ISSN = {0020-9910,1432-1297},
   MRCLASS = {20E42 (20F55 20J06 22E50)},
  MRNUMBER = {1946553},
MRREVIEWER = {Alain\ Valette},
       DOI = {10.1007/s00222-002-0242-y},
       URL = {https://doi.org/10.1007/s00222-002-0242-y},
}

@article {E,
    AUTHOR = {Efrat, Isaac},
     TITLE = {Automorphic spectra on the tree of {${\rm PGL}_2$}},
   JOURNAL = {Enseign. Math. (2)},
  FJOURNAL = {L'Enseignement Math\'ematique. Revue Internationale. 2e
              S\'erie},
    VOLUME = {37},
      YEAR = {1991},
    NUMBER = {1-2},
     PAGES = {31--43},
      ISSN = {0013-8584},
   MRCLASS = {11F72 (11F12 22E40)},
  MRNUMBER = {1115742},
MRREVIEWER = {James\ Lee\ Hafner},
}

@article {FK,
    AUTHOR = {Fleig, Philipp and Kleinschmidt, Axel},
     TITLE = {Eisenstein series for infinite-dimensional {U}-duality groups},
   JOURNAL = {J. High Energy Phys.},
  FJOURNAL = {Journal of High Energy Physics},
      YEAR = {2012},
    NUMBER = {6},
     PAGES = {054, front matter+47},
      ISSN = {1126-6708,1029-8479},
   MRCLASS = {81T30 (11Z05 81P70 81R05 81R10)},
  MRNUMBER = {3006875},
MRREVIEWER = {Doru\ \c Stef\u anescu},
       DOI = {10.1007/JHEP06(2012)054},
       URL = {https://doi.org/10.1007/JHEP06(2012)054},
}

@book {FGKP,
    AUTHOR = {Fleig, Philipp and Gustafsson, Henrik P. A. and Kleinschmidt,
              Axel and Persson, Daniel},
     TITLE = {Eisenstein series and automorphic representations},
    SERIES = {Cambridge Studies in Advanced Mathematics},
    VOLUME = {176},
      NOTE = {With applications in string theory},
 PUBLISHER = {Cambridge University Press, Cambridge},
      YEAR = {2018},
     PAGES = {xviii+567},
      ISBN = {978-1-107-18992-8},
   MRCLASS = {11-02 (11F12 11F25 11F27 11F70 81-01 81T30 83E30)},
  MRNUMBER = {3793195},
MRREVIEWER = {Anton\ Deitmar},
       DOI = {10.1017/9781316995860},
       URL = {https://doi.org/10.1017/9781316995860},
}

@article {GarlandLoopGp,
    AUTHOR = {Garland, Howard},
     TITLE = {The arithmetic theory of loop groups},
   JOURNAL = {Inst. Hautes \'Etudes Sci. Publ. Math.},
  FJOURNAL = {Institut des Hautes \'Etudes Scientifiques. Publications
              Math\'ematiques},
    NUMBER = {52},
      YEAR = {1980},
     PAGES = {5--136},
      ISSN = {0073-8301,1618-1913},
   MRCLASS = {20G25 (14L17 17B10 17B65 20G05 22E40)},
  MRNUMBER = {601519},
MRREVIEWER = {S.\ I.\ Gel\cprime fand},
       URL = {http://www.numdam.org/item?id=PMIHES_1980__52__5_0},
}

@article {Ga1,
    AUTHOR = {Garland, Howard},
     TITLE = {Eisenstein series on arithmetic quotients of loop groups},
   JOURNAL = {Math. Res. Lett.},
  FJOURNAL = {Mathematical Research Letters},
    VOLUME = {6},
      YEAR = {1999},
    NUMBER = {5-6},
     PAGES = {723--733},
      ISSN = {1073-2780},
   MRCLASS = {11M36 (11F72 22E67)},
  MRNUMBER = {1739228},
MRREVIEWER = {Paul\ E.\ Gunnells},
       DOI = {10.4310/MRL.1999.v6.n6.a11},
       URL = {https://doi.org/10.4310/MRL.1999.v6.n6.a11},
}

@incollection {Ga2,
    AUTHOR = {Garland, H.},
     TITLE = {Certain {E}isenstein series on loop groups: convergence and
              the constant term},
 BOOKTITLE = {Algebraic groups and arithmetic},
     PAGES = {275--319},
 PUBLISHER = {Tata Inst. Fund. Res., Mumbai},
      YEAR = {2004},
      ISBN = {81-7319-618-4},
   MRCLASS = {11M36 (11F72 22E67)},
  MRNUMBER = {2094114},
MRREVIEWER = {Anton\ Deitmar},
}

@article {Ga5,
    AUTHOR = {Garland, Howard},
     TITLE = {Eisenstein series on loop groups: {M}aass-{S}elberg relations.
              {II}},
   JOURNAL = {Amer. J. Math.},
  FJOURNAL = {American Journal of Mathematics},
    VOLUME = {129},
      YEAR = {2007},
    NUMBER = {3},
     PAGES = {723--784},
      ISSN = {0002-9327,1080-6377},
   MRCLASS = {22E67 (11F55 11M36)},
  MRNUMBER = {2325102},
MRREVIEWER = {Dmitry\ A.\ Timash\"ev},
       DOI = {10.1353/ajm.2007.0017},
       URL = {https://doi.org/10.1353/ajm.2007.0017},
}

@incollection {Ga6,
    AUTHOR = {Garland, Howard},
     TITLE = {Eisenstein series on loop groups: {M}aass-{S}elberg relations.
              {I}},
 BOOKTITLE = {Algebraic groups and homogeneous spaces},
    SERIES = {Tata Inst. Fund. Res. Stud. Math.},
    VOLUME = {19},
     PAGES = {275--300},
 PUBLISHER = {Tata Inst. Fund. Res., Mumbai},
      YEAR = {2007},
      ISBN = {978-81-7319-802-1},
   MRCLASS = {22E67 (11F55 11M36)},
  MRNUMBER = {2348907},
MRREVIEWER = {Dmitry\ A.\ Timash\"ev},
}

@article {Ga7,
    AUTHOR = {Garland, Howard},
     TITLE = {Eisenstein series on loop groups: {M}aass-{S}elberg relations.
              {III}},
   JOURNAL = {Amer. J. Math.},
  FJOURNAL = {American Journal of Mathematics},
    VOLUME = {129},
      YEAR = {2007},
    NUMBER = {5},
     PAGES = {1277--1353},
      ISSN = {0002-9327,1080-6377},
   MRCLASS = {22E67 (11F55 11M36)},
  MRNUMBER = {2354321},
MRREVIEWER = {Dmitry\ A.\ Timash\"ev},
       DOI = {10.1353/ajm.2007.0032},
       URL = {https://doi.org/10.1353/ajm.2007.0032},
}

@incollection {Ga8,
    AUTHOR = {Garland, Howard},
     TITLE = {Eisenstein series on loop groups: {M}aass-{S}elberg relations.
              {IV}},
 BOOKTITLE = {Lie algebras, vertex operator algebras and their applications},
    SERIES = {Contemp. Math.},
    VOLUME = {442},
     PAGES = {115--158},
 PUBLISHER = {Amer. Math. Soc., Providence, RI},
      YEAR = {2007},
      ISBN = {978-0-8218-3986-7; 0-8218-3986-1},
   MRCLASS = {22E67 (11F55 11M36)},
  MRNUMBER = {2372559},
MRREVIEWER = {Dmitry\ A.\ Timash\"ev},
       DOI = {10.1090/conm/442/08522},
       URL = {https://doi.org/10.1090/conm/442/08522},
}

@book {Garr1,
    AUTHOR = {Garrett, Paul},
     TITLE = {Bernstein's continuation principle},
    SERIES = {http://www.math.umn.edu/~garrett/m/v/},
     YEAR = {2024},

}

@book {Garr2,
    AUTHOR = {Garrett, Paul},
     TITLE = {Modern analysis of automorphic forms by example. {V}ol. 1},
    SERIES = {Cambridge Studies in Advanced Mathematics},
    VOLUME = {173},
 PUBLISHER = {Cambridge University Press, Cambridge},
      YEAR = {2018},
     PAGES = {xxii+384},
      ISBN = {978-1-107-15400-1},
   MRCLASS = {11-02 (11Fxx 20G15 20G35)},
  MRNUMBER = {3837525},
MRREVIEWER = {Lei\ Yang},
}

@article {H,
    AUTHOR = {Harder, G.},
     TITLE = {Chevalley groups over function fields and automorphic forms},
   JOURNAL = {Ann. of Math. (2)},
  FJOURNAL = {Annals of Mathematics. Second Series},
    VOLUME = {100},
      YEAR = {1974},
     PAGES = {249--306},
      ISSN = {0003-486X},
   MRCLASS = {10D99 (14G25 22E55)},
  MRNUMBER = {563090},
MRREVIEWER = {Stephen\ Gelbart},
       DOI = {10.2307/1971073},
       URL = {https://doi.org/10.2307/1971073},
}

@book{K,
	Address = {Cambridge},
	Author = {Kac, Victor G.},
	Edition = {Third},
	Isbn = {0-521-37215-1; 0-521-46693-8},
	Mrclass = {17B65 (17B67 17B68 58F07)},
	Mrnumber = {92k:17038},
	Pages = {xxii+400},
	Publisher = {Cambridge University Press},
	Title = {Infinite-dimensional {L}ie algebras},
	Year = {1990}}

@book{Ku,
	Author = {Kumar, Shrawan},
	Publisher = {Springer Science \& Business Media},
	Title = {Kac-Moody groups, their flag varieties and representation theory},
	Volume = {204},
	Year = {2012}}

@article {LL,
    AUTHOR = {Lee, Kyu-Hwan and Lombardo, Philip},
     TITLE = {Eisenstein series on affine {K}ac-{M}oody groups over function
              fields},
   JOURNAL = {Trans. Amer. Math. Soc.},
  FJOURNAL = {Transactions of the American Mathematical Society},
    VOLUME = {366},
      YEAR = {2014},
    NUMBER = {4},
     PAGES = {2121--2165},
      ISSN = {0002-9947,1088-6850},
   MRCLASS = {22E67 (11F70 17B67)},
  MRNUMBER = {3152725},
MRREVIEWER = {Stephen\ Slebarski},
       DOI = {10.1090/S0002-9947-2013-06078-3},
       URL = {https://doi.org/10.1090/S0002-9947-2013-06078-3},
}

@article {Li,
    AUTHOR = {Li, Wen Ch'ing Winnie},
     TITLE = {Eisenstein series and decomposition theory over function
              fields},
   JOURNAL = {Math. Ann.},
  FJOURNAL = {Mathematische Annalen},
    VOLUME = {240},
      YEAR = {1979},
    NUMBER = {2},
     PAGES = {115--139},
      ISSN = {0025-5831,1432-1807},
   MRCLASS = {10D15},
  MRNUMBER = {524661},
MRREVIEWER = {Stephen\ Gelbart},
       DOI = {10.1007/BF01364628},
       URL = {https://doi.org/10.1007/BF01364628},
}

@article {Liu,
    AUTHOR = {Liu, Dongwen},
     TITLE = {Eisenstein series on loop groups},
   JOURNAL = {Trans. Amer. Math. Soc.},
  FJOURNAL = {Transactions of the American Mathematical Society},
    VOLUME = {367},
      YEAR = {2015},
    NUMBER = {3},
     PAGES = {2079--2135},
      ISSN = {0002-9947,1088-6850},
   MRCLASS = {22E67 (11M36)},
  MRNUMBER = {3286509},
MRREVIEWER = {Karl-Hermann\ Neeb},
       DOI = {10.1090/S0002-9947-2014-06220-X},
       URL = {https://doi.org/10.1090/S0002-9947-2014-06220-X},
}

@article {Morris,
    AUTHOR = {Morris, L. E.},
     TITLE = {Eisenstein series for reductive groups over global function
              fields. {II}. {T}he general case},
   JOURNAL = {Canadian J. Math.},
  FJOURNAL = {Canadian Journal of Mathematics. Journal Canadien de
              Math\'ematiques},
    VOLUME = {34},
      YEAR = {1982},
    NUMBER = {5},
 }

@article {N,
    AUTHOR = {Nagoshi, Hirofumi},
     TITLE = {Selberg zeta functions over function fields},
   JOURNAL = {J. Number Theory},
  FJOURNAL = {Journal of Number Theory},
    VOLUME = {90},
      YEAR = {2001},
    NUMBER = {2},
     PAGES = {207--238},
      ISSN = {0022-314X,1096-1658},
   MRCLASS = {11M38 (11T60)},
  MRNUMBER = {1858074},
MRREVIEWER = {Maki\ Nakasuji},
       DOI = {10.1006/jnth.2001.2658},
       URL = {https://doi.org/10.1006/jnth.2001.2658},
}

@article {Re1,
    AUTHOR = {R\'emy, Bertrand},
     TITLE = {Groupes de {K}ac-{M}oody d\'eploy\'es et presque d\'eploy\'es},
   JOURNAL = {Ast\'erisque},
  FJOURNAL = {Ast\'erisque},
    NUMBER = {277},
      YEAR = {2002},
     PAGES = {viii+348},
      ISSN = {0303-1179,2492-5926},
   MRCLASS = {20E42 (17B67 20G15 51E24)},
  MRNUMBER = {1909671},
MRREVIEWER = {Guy\ Rousseau},
}

@article {RR,
    AUTHOR = {R\'emy, Bertrand and Ronan, Mark},
     TITLE = {Topological groups of {K}ac-{M}oody type, right-angled
              twinnings and their lattices},
   JOURNAL = {Comment. Math. Helv.},
  FJOURNAL = {Commentarii Mathematici Helvetici. A Journal of the Swiss
              Mathematical Society},
    VOLUME = {81},
      YEAR = {2006},
    NUMBER = {1},
     PAGES = {191--219},
      ISSN = {0010-2571,1420-8946},
   MRCLASS = {20E42 (17B67 20G25 22E40 51E24)},
  MRNUMBER = {2208804},
MRREVIEWER = {Guy\ Rousseau},
       DOI = {10.4171/CMH/49},
       URL = {https://doi.org/10.4171/CMH/49},
}

@incollection {So,
    AUTHOR = {Solomon, Louis},
     TITLE = {The {S}teinberg character of a finite group with {$BN$}-pair},
 BOOKTITLE = {Theory of {F}inite {G}roups ({S}ymposium, {H}arvard {U}niv.,
              {C}ambridge, {M}ass., 1968)},
     PAGES = {213--221},
 PUBLISHER = {W. A. Benjamin, Inc., New York-Amsterdam},
      YEAR = {1969},
   MRCLASS = {20.25},
  MRNUMBER = {246951},
MRREVIEWER = {T.\ Ono},
}

@incollection {Ti2,
    AUTHOR = {Tits, Jacques},
     TITLE = {Groups and group functors attached to {K}ac-{M}oody data},
 BOOKTITLE = {Workshop {B}onn 1984 ({B}onn, 1984)},
    SERIES = {Lecture Notes in Math.},
    VOLUME = {1111},
     PAGES = {193--223},
 PUBLISHER = {Springer, Berlin},
      YEAR = {1985},
      ISBN = {3-540-15195-8},
   MRCLASS = {22E65 (17B67)},
  MRNUMBER = {797422},
MRREVIEWER = {Andrew\ Pressley},
       DOI = {10.1007/BFb0084591},
       URL = {https://doi.org/10.1007/BFb0084591},
}

@incollection {Ti4,
    AUTHOR = {Tits, Jacques},
     TITLE = {Groupes associ\'es aux alg\`ebres de {K}ac-{M}oody},
      NOTE = {S\'eminaire Bourbaki, Vol. 1988/89},
   JOURNAL = {Ast\'erisque},
  FJOURNAL = {Ast\'erisque},
    NUMBER = {177-178},
      YEAR = {1989},
     PAGES = {Exp. No. 700, 7--31},
      ISSN = {0303-1179,2492-5926},
   MRCLASS = {22E65 (14L15 17B67 20G15)},
  MRNUMBER = {1040566},
MRREVIEWER = {Guy\ Rousseau},
}

@book {We,
    AUTHOR = {Weiss, Richard M.},
     TITLE = {The structure of spherical buildings},
 PUBLISHER = {Princeton University Press, Princeton, NJ},
      YEAR = {2003},
     PAGES = {xiv+135},
      ISBN = {0-691-11733-0},
   MRCLASS = {51E24 (20E42 20F55)},
  MRNUMBER = {2034361},
MRREVIEWER = {Theo\ Grundh\"ofer},
}
\end{document}